\documentclass[10pt]{article}
\usepackage{color, amsmath, amsthm, amsfonts, amssymb, listings, graphics, epsfig, enumerate, bbm,dsfont,soul,mathrsfs}
\usepackage{subcaption,graphicx, 
adjustbox,algorithm,algorithmic,multirow}

\usepackage[margin=1in]{geometry}
\usepackage[fleqn,tbtags]{mathtools}
\usepackage[round]{natbib}
\usepackage[english]{babel}

\title{Multi-objective risk-averse two-stage\\ stochastic programming problems}
\author{\c{C}a\u{g}{\i}n Ararat\thanks{Bilkent University, Department of Industrial Engineering, Ankara, Turkey.} \thanks{\c{C}.~Ararat and \"{O}.~\c{C}avu\c{s} contributed equally to this work.}
\and
\"{O}zlem \c{C}avu\c{s}\footnotemark[1] \footnotemark[2]
\and
Ali \.{I}rfan Mahmuto\u{g}ullar{\i}\footnotemark[1]
}
\date{November 15, 2017}

\addto\captionsenglish {}

\makeatletter \renewenvironment{proof}[1][\proofname] {\par\pushQED{\qed}\normalfont\topsep6\p@\@plus6\p@\relax\trivlist\item[\hskip\labelsep\bfseries#1\@addpunct{.}]\ignorespaces}{\popQED\endtrivlist\@endpefalse} \makeatother
\newtheorem{theorem}{Theorem}[section]

\newtheorem{lemma}[theorem]{Lemma}
\newtheorem{proposition}[theorem]{Proposition}

\newtheorem{definition}[theorem]{Definition}
\theoremstyle{definition}
\newtheorem{example}[theorem]{Example}
\newtheorem{remark}[theorem]{Remark}

\numberwithin{equation}{section}

\newcommand{\R}{\mathbb{R}}
\newcommand{\sm}{\!\setminus\!}

\DeclareMathOperator{\cl}{cl}

\DeclareMathOperator{\co}{co}

\DeclareMathOperator{\maxi}{maximize}
\DeclareMathOperator{\argmin}{arg\,min}
\DeclareMathOperator{\argmax}{arg\,max}

\let\abs=\envert

\newcommand{\W}{\mathscr{W}}
\newcommand{\J}{\mathcal{J}}
\newcommand{\G}{\mathcal{G}}
\newcommand{\F}{\mathcal{F}}
\newcommand{\X}{\mathcal{X}}
\newcommand{\A}{\mathcal{A}}
\renewcommand{\L}{\mathbb{L}}
\renewcommand{\P}{\mathscr{P}}
\newcommand{\D}{\mathscr{D}}
\newcommand{\E}{\mathbb{E}}
\newcommand{\M}{\mathbb{M}}

\renewcommand{\a}{\alpha}
\renewcommand{\b}{\beta}

\renewcommand{\O}{\Omega}

\newcommand{\1}{\mathbf{1}}

\newcommand{\of}[1]{\ensuremath{\left( #1 \right)}}
\newcommand{\cb}[1]{\ensuremath{ \left\{ #1 \right\} }}
\newcommand{\sqb}[1]{\ensuremath{ \left[ #1 \right] }}
\newcommand{\norm}[1]{\ensuremath{ \left\Vert #1 \right\Vert }}

\def\prehp(#1,#2){\ensuremath{  #1 \cdot #2 }}

\begin{document}

\maketitle
\thispagestyle{empty}

\renewcommand{\sectionmark}[1]{}

\begin{abstract}
We consider a multi-objective risk-averse two-stage stochastic programming problem with a multivariate convex risk measure. We suggest a convex vector optimization formulation with set-valued constraints and propose an extended version of Benson's algorithm to solve this problem. Using Lagrangian duality, we develop scenario-wise decomposition methods to solve the two scalarization problems appearing in Benson's algorithm. Then, we propose a procedure to recover the primal solutions of these scalarization problems from the solutions of their Lagrangian dual problems. Finally, we test our algorithms on a multi-asset portfolio optimization problem under transaction costs.\\
\\[-5pt]
\\[-5pt]
\end{abstract}

\section{Introduction}\label{intro}

We consider a \emph{multi-objective risk-averse two-stage stochastic programming problem} of the general form
\begin{align*}
&\text{min}\; \;       z\;\;\text{w.r.t.}\;\;{\R^J_+}\\
&\text{s.t. }\;\;                      z\in R(Cx + Qy)\\
& \quad\quad\;      (x,y)\in\X, z\in\R^J.
\end{align*}
In this formulation, $x$ is the \emph{first-stage} decision variable, $y$ is the \emph{second-stage} decision variable and $\X$ is a compact finite-dimensional set defined by linear constraints. $C, Q$ are cost parameters which are matrices of appropriate dimension. We assume that $C$ is deterministic and $Q$ is random. $R(\cdot)$ is a \emph{multivariate convex risk measure}, which is a set-valued mapping from the space of $J$-dimensional random vectors into the power set of $\R^J$ (\emph{see} \cite{hh:duality}). In other words, $R(Cx+Qy)$ is the set of deterministic cost vectors $z\in\R^J$ for which $Cx+Qy-z$ becomes acceptable in a certain sense.

The above problem is a \emph{vector optimization} problem and solving it is understood as computing the \emph{upper image} $\P$ of the problem defined by
\[
\P= \cl \cb{z\in\R^J\mid z\in R(Cx+Qy),\, (x,y)\in\X},
\]
whose boundary is the so-called \emph{efficient frontier}. Here, $\cl$ denotes the closure operator. One would be interested in finding a set $\mathcal{Z}$ of \emph{weakly efficient solutions} $(x,y,z)$ with $z\in R(Cx+Qy)$ for some $(x,y)\in\X$ such that there is no $z^{\prime}\in R(Cx^{\prime}+Qy^{\prime})$ with $(x^{\prime},y^{\prime})\in\X$ and $z^{\prime}<z$. Here, $``<"$ denotes the componentwise strict order in $\R^J$. The $z$ components of these solutions are on the efficient frontier. In addition, the set $\mathcal{Z}$ is supposed to construct $\mathscr{P}$ in the sense that
\[
\mathscr{P}=\cl\co(\cb{z\in\R^J\mid (x,y,z)\in\mathcal{Z}}+\R^J_+),
\]
where $\co$ denotes the convex hull operator. Our aim is to compute $\mathscr{P}$ approximately using a finite set of weakly efficient solutions.

Algorithms for computing upper images of vector optimization problems are extensively studied in the literature. A seminal contribution in this field is the algorithm for linear vector optimization problems by \cite{benson}, which computes the set of all weakly efficient solutions of the problem and works on the outer approximation of the upper image rather than the feasible region itself. Benson's algorithm has been generalized recently in \cite{ehrgottconvex} and \cite{convexbenson} for (ordinary) convex vector optimization problems, namely, optimization problems with a vector-valued objective function and a vector-valued constraint that are convex with respect to certain underlying cones, e.g., the positive orthants in the respective dimensions. While the algorithm in \cite{ehrgottconvex} relies on the differentiability of the involved functions, \cite{convexbenson} makes no assumption on differentiability and obtains finer approximations of the upper image by making use of the so-called \emph{geometric dual problem}.

In the literature, there is a limited number of studies on multi-objective two-stage stochastic optimization problems. Some examples of these studies are \cite{abbas}, \cite{supply}, where the decision maker is risk-neutral, that is, one takes $R(Cx+Qy)=\E\sqb{Cx+Qy}+\R^J_+$. In principle, multi-objective risk-neutral two-stage stochastic optimization problems with linear constraints and continuous variables can be formulated as linear vector optimization problems and they can be solved using the algorithm in \cite{benson}. If the number of scenarios is not too large, then the problem can be solved in reasonable computation time. Otherwise, one should look for an efficient method, generally, based on scenario decompositions.

To the best of our knowledge, for the risk-averse case, there is no study on multi-objective two-stage stochastic programming problems. However, single-objective mean-risk type problems can be seen as scalarizations of two-objective stochastic programming problems (\emph{see, for instance,} \cite{ahmed}, \cite{miller}). On the other hand, \cite{dentcheva}, \cite{noyankucukyavuz} work on single-objective problems with multivariate stochastic ordering constraints. As pointed out in the recent survey \cite{survey}, there is a need for a general methodology for the formulation and solution of multi-objective risk-averse stochastic problems.

The main contributions of the present study can be summarized as follows:
\begin{enumerate}[1.]
\item To the best of our knowledge, this is the first study focusing on multi-objective risk-averse two-stage stochastic programming problems in a general setting.
\item We propose a vector optimization formulation for our problem using multivariate convex risk measures. Such risk measures include, but are not limited to, multivariate coherent risk measures and multivariate utility-based risk measures.
\item To solve our problem, we suggest an extended version of the convex Benson algorithm in \cite{convexbenson} that is developed for a convex vector optimization problem with a vector-valued constraint. Different from \cite{convexbenson}, we deal with set-valued risk constraints and dualize them using the dual representation of multivariate convex risk measures (\emph{see} \cite{hh:duality}) and the Lagrange duality for set-valued constraints (\emph{see} \cite{borwein}).
\item The convex Benson algorithm in \cite{convexbenson} cannot be used for some multivariate risk measures, specifically, for higher-order nonsmooth risk measures. On the other hand, our method is general and can be used for any risk measures for which subgradients can be calculated. An example of such risk measures is higher-order mean semideviation (\emph{see} \cite{shapiro} and the references therein).
\item Two risk-averse two-stage stochastic scalarization problems, namely, the problem of weighted sum scalarization and the problem of scalarization by a reference variable, have to be solved during the procedure of the convex Benson algorithm. As the number of scenarios gets larger, these problems cannot be solved in reasonable computation time. Therefore, based on Lagrangian duality, we propose scenario-wise decomposable dual problems for these scalarization problems and suggest a solution procedure based on the bundle algorithm (\emph{see} \cite{lemarechal}, \cite{ruszbook} and the references therein).
\item Our scenario-wise decomposition algorithms for the scalarization problems can be embedded into other algorithms using the same type of scalarization problems. See \cite[Chapter~12]{jahnbook} for examples of such algorithms.
\item We propose a procedure to recover the primal solutions of the scalarization problems from the solutions of their Lagrangian dual problems.
\end{enumerate}

The rest of the paper is organized as follows: In Section~\ref{problemdefn}, we provide some preliminary definitions and results for multivariate convex risk measures. In Section~\ref{problem}, we provide the problem formulation and recall the related notions of optimality. Section~\ref{algorithms} is devoted to the convex Benson algorithm. The two scalarization problems in this algorithm are treated separately in Section~\ref{scalarization}. In particular, we propose scenario-wise decomposition algorithms and procedures to recover primal solutions. Computational results are provided in Section~\ref{computation}. Some proofs related to Section~\ref{scalarization} are collected in the appendix.

\section{Multivariate convex risk measures}\label{problemdefn}

We work on a finite probability space $\O = \cb{1, \ldots, I}$ with $I\geq 2$. For each $i\in\O$, let $p_i> 0$ be the probability of the elementary event $\cb{i}$ so that $\sum_{i\in\O} p_i = 1$.

Let us introduce the notation for (random) vectors and matrices. Let $J\geq 1$ be a given integer and $\J=\cb{1,\ldots,J}$. $\R^J_+$ and $\R^J_{++}$ denote the set of all elements of the Euclidean space $\R^J$ whose components are nonnegative and positive, respectively. For $w=(w^1,\ldots,w^J)^\mathsf{T},z=(z^1,\ldots,z^J)^\mathsf{T}\in\R^J$, their scalar product and Hadamard product are defined as
\[
w^{\mathsf{T}}z = \sum_{j\in\J} w^j z^j \in\R,\quad w\cdot z = (w^1 z^1 , \ldots, w^J z^J)^\mathsf{T}\in\R^J,
\]
respectively. For a set $\mathcal{Z}\subseteq \R^J$, its associated indicator function (in the sense of convex analysis) is defined by
\[
I_{\mathcal{Z}}(z)=\begin{cases}0 & \text{ if }z\in\mathcal{Z}, \\ +\infty & \text{ else,}\end{cases}
\]
for each $z\in \R^J$. We denote by $\L^J$ the set of all $J$-dimensional random cost vectors $u=(u^1,\ldots,u^J)^{\mathsf{T}}$, which is clearly isomorphic to the space $\R^{J\times I}$ of $J\times I$-dimensional real matrices. We write $\L=\L^1$ for $J=1$. For $u\in \L^J$, we denote by $u_i=(u^1_i,\ldots,u^J_i)^{\mathsf{T}}\in\R^J$ its realization at $i\in \O$, and define the expected value of $u$ as
\[
\E\sqb{u}=\sum_{i\in\O} p_i u_i \in \R^J.
\]
Similarly, given another integer $N\geq 1$, we denote by $\L^{J\times N}$ the set of all $J\times N$-dimensional random matrices $Q$ with realizations $Q_1,\ldots,Q_I$.

The elements of $\L^J$ will be used to denote random cost vectors; hence, lower values are preferable. To that end, we introduce $\L^J_{+}$, the set of all elements in $\L^J$ whose components are nonnegative random variables. Given $u,v\in \L^J$, we write $u\leq v$ if and only if $u_i^j\leq v_i^j$ for every $i\in\O$ and $j\in\J$, that is, $v\in u+\L^J_+$. We call a set-valued function $R\colon \L^J\to 2^{\R^J}$ a \emph{multivariate convex risk measure} if it satisfies the following axioms (\emph{see} \cite{hh:duality}):
\begin{enumerate}
\item[(A1)] \textbf{Monotonicity: }$u\leq v$ implies $R(u)\supseteq R(v)$ for every $u,v\in \L^J$.
\item[(A2)] \textbf{Translativity: }$R(u+z)=R(u)+z$ for every $u\in \L^J$ and $z\in\R^J$.
\item[(A3)] \textbf{Finiteness: }$R(u)\notin\cb{\emptyset,\R^J}$ for every $u\in \L^J$.
\item[(A4)] \textbf{Convexity: }$R(\gamma u +(1-\gamma) v)\supseteq \gamma R(u) + (1-\gamma) R(v)$ for every $u,v\in \L^J$, $\gamma \in (0,1)$.
\item[(A5)] \textbf{Closedness: }The \emph{acceptance set} $\A\coloneqq \cb{u\in \L^J\mid 0\in R(u)}$ of $R$ is a closed set.
\end{enumerate}
A multivariate convex risk measure $R$ is called \emph{coherent} if it also satisfies the following axiom:
\begin{enumerate}
\item[(A6)] \textbf{Positive homogeneity: }$R(\gamma u)=\gamma R(u)$ for every $u\in\L^J$, $\gamma>0$.
\end{enumerate}

\begin{remark}\label{completelattice}
It is easy to check that the values of a multivariate convex risk measure $R$ are in the collection of all \emph{closed convex upper subsets} of $\R^J$, that is,
\[
\G = \cb{E \subseteq \R^J \mid E = \cl \co (E+\R^J_+)},
\]
where $\cl$ and $\co$ denote the closure and convex hull operators, respectively. In other words, for every $u\in\L^J$, the set $R(u)$ is a closed convex set with the property $R(u)=R(u)+\R^J_+$. The collection $\G$, when equipped with the superset relation $\supseteq$, is a complete lattice in the sense that every nonempty subset $\mathcal{E}$ of $\G$ has an infimum (and also a supremum) which is uniquely given by $\inf \mathcal{E}=\cl \co \bigcup_{E\in \mathcal{E}}E$ as an element of $\G$ (\emph{see} Example~2.13 in \cite{setoptsurvey}). The complete lattice property of $\G$ makes it possible to study optimization problems with $\G$-valued objective functions and constraints, as will also be crucial in the approach of the present paper.
\end{remark}

A multivariate convex risk measure $R$ can be represented in terms of vectors $\mu$ of probability measures and weight vectors $w$ in the cone $\R^J_+\sm\cb{0}$, which is called its \emph{dual representation}. To state this representation, we provide the following definitions and notation.

Let $\M_1^J$ be the set of all $J$-dimensional vectors $\mu=(\mu^1,\ldots, \mu^J) $ of probability measures on $\O$, that is, for each $j\in\J$, the probability measure $\mu^j$ assigns the probability $\mu^j_i $ to the elementary event $\cb{i}$ for $i\in\O$. For $\mu\in \M_1^J$ and $i\in\O$, we also write $\mu_i\coloneqq (\mu_i^1,\ldots,\mu_i^J)^{\mathsf{T}}\in \R^J$. Finally, for $\mu\in\M_1^J$ and $u\in \L^J$, we define the expectation of $u$ under $\mu$ by
\[
\E^\mu\sqb{u}=\of{\E^{\mu^1}[u^1],\ldots,\E^{\mu^J}[u^J]}^\mathsf{T}= \sum_{i\in\O} \mu_i \cdot u_i.
\]

A multivariate convex risk measure $R$ has the following dual representation (\emph{see} Theorem~6.1 in \cite{hh:duality}): for every $u\in \L^J$,
\begin{align*}
R(u) &= \bigcap_{\mu\in \M_1^J, w\in\R^J_+\sm\cb{0}} \of{\E^\mu \sqb{u}+\cb{z\in\R^J\mid w^{\mathsf{T}}z\geq -\b(\mu,w)}}\\
&= \bigcap_{ w\in\R^J_+\sm\cb{0}} \cb{z\in\R^J\mid w^{\mathsf{T}}z\geq \sup_{\mu\in \M_1^J}\of{w^{\mathsf{T}}\E^\mu \sqb{u} -\b(\mu,w)}},
\end{align*}
where $\b$ is the \emph{minimal penalty function} of $R$ defined by
\begin{equation}\label{beta}
\b(\mu,w)=\sup_{u\in \A} w^{\mathsf{T}}\E^\mu\sqb{u}=\sup\cb{w^{\mathsf{T}}\E^\mu\sqb{u}\mid 0\in R(u), u\in \L^J},
\end{equation}
for each $\mu\in\M_1^J$, $w\in\R^J_+\sm\cb{0}$. Note that $\b(\cdot,w)$ and $\b(\mu,\cdot)$ are convex functions as they are suprema of linear functions.

The \emph{scalarization} of $R$ by a weight vector $w\in\R^J_+\sm\cb{0}$ is defined as the function
\begin{equation}\label{phi}
u\mapsto \varphi_w(u)\coloneqq \inf_{z\in R(u)}w^{\mathsf{T}}z
\end{equation}
on $\L^J$. As an immediate consequence of the dual representation of $R$, we also obtain a dual representation for its scalarization:
\begin{equation}\label{scalar-dual}
\varphi_w(u)= \sup_{\mu\in \M_1^J}\of{w^{\mathsf{T}}\E^\mu \sqb{u} -\b(\mu,w)}.
\end{equation}

Some examples of multivariate coherent and convex risk measures are the \emph{multivariate conditional value-at-risk (multivariate CVaR)} and the \emph{multivariate entropic risk measure}, respectively.

\begin{example}[Multivariate CVaR]\label{multiCVaR}
Let $C\subseteq\R^J$ be a polyhedral closed convex cone with $\R^J_+\subseteq C \neq \R^J$. The multivariate conditional value-at-risk is defined by
\begin{equation}\label{cvardefn}
R(u)=\of{CVaR_{\nu^1}(u^1),\ldots,CVaR_{\nu^J}(u^J)}^{\mathsf{T}}+C,
\end{equation}
where
\[
CVaR_{\nu^j}(u^j)=\inf_{z^{j}\in\R}\of{z^{j}+\frac{1}{1-\nu^{j}}\mathbb{E}\sqb{(u^{j}-z^j)^{+}}},
\]
for each $u\in\L^J$ and $j\in\J$ (\emph{see} Definition~2.1 and Remark~2.3 in \cite{hry:avar}). Here, $\nu^j\in (0,1)$ is a risk-aversion parameter and $(x)^+\coloneqq\max\cb{x,0}$ for $x\in\R$. The minimal penalty function of $R$ is given by
\[
\b(\mu,w)=\begin{cases}0& \text{ if }w\in C^+ \text{ and }\frac{\mu^j_i}{p_i}\leq \frac{1}{1-\nu^j}, \;\forall  i\in\O,j\in\J,\\ +\infty &\text{ else,}\end{cases}
\]
where $C^+$ is the positive dual cone of $C$ defined by
\[
C^+=\cb{w\in \R^J\mid w^\mathsf{T}z\geq 0,\;\forall z\in C}.
\]
Note that \eqref{cvardefn} is the multivariate extension of the well-known conditional value-at-risk (\emph{see} \cite{uryasev1}, \cite{uryasev2}). 
\end{example}

\begin{example}[Multivariate entropic risk measure]\label{multient} Consider the vector-valued exponential utility function $U\colon\R^J\to\R^J$ defined by
\[
U(x)=(U^1(x^1),\ldots,U^J(x^J))^\mathsf{T},
\]
where
\[
U^{j}(x^j)=\frac{1-e^{\delta^{j}x^j}}{\delta^{j}},
\]
for each $x\in\R^J$ and $j\in\J$. Here, $\delta^j>0$ is a risk-aversion parameter. Note that $U^j(\cdot)$ is a concave decreasing function. Let $C\subseteq\R^J$ be a polyhedral closed convex cone with $\R^J_+\subseteq C \neq \R^J$. The multivariate entropic risk measure $R\colon \L^J\to 2^{\R^J}$ is defined as
\begin{equation}\label{entropicdefn}
R(u) = \cb{z\in\R^{J}\mid\E\sqb{U(u-z)}\in C},
\end{equation}
for each $u\in\L^J$ (\emph{see} Section~4.1 in \cite{sdrm}). Since $\R^J_+\subseteq C$, larger values of the expected utility are prefered. Moreover, as each $U^j(\cdot)$ is a decreasing function, $z\in R(u)$ implies $z^\prime\in R(u)$ for every $z^\prime \geq z$.

Finally, the minimal penalty function of $R$ is given by (\emph{see} Proposition~4.4 in \cite{sdrm})
\[
\beta(\mu,w) = \sum_{j\in\J}\frac{w^{j}}{\delta^{j}}\of{H(\mu^j||p)-1+\log w^j}+\inf_{s\in C^+}\sum_{j\in\J}\frac{1}{\delta^{j}}\of{s^{j}-w^j\log s^{j}},
\]
where $H(\mu^j||p)$ is the \emph{relative entropy} of $\mu^j$ with respect to $p$ defined by
\[
H(\mu^j||p)=\sum_{i\in\O}\mu^j_i\log\of{\frac{\mu^j_i}{p_i}}.
\]
Note that \eqref{entropicdefn} is the multivariate extension of the well-known entropic risk measure (\emph{see} \cite{fs}). 
\end{example}

\section{Problem formulation}\label{problem}

We consider a multi-objective risk-averse two-stage stochastic programming problem. The decision variables and the parameters of the problem consist of deterministic and random vectors and matrices of different dimensions. To that end, let us fix some integers $J,K,L,M,N\geq 1$ and deterministic parameters $A\in \R^{K\times M}$ and $b\in\R^K$. At the first stage, the decision-maker chooses a deterministic vector $x\in\R^M_+$ with associated cost $Cx$, where $C\in \R^{J\times M}$. At the second stage, the decision-maker chooses a random vector $y\in \L^N_+$ based on the first-stage decision $x\in\R^M$ as well as the random parameters $W\in \L^{L\times N}, T\in \L^{L\times M}, h\in \L^{L}$. The random cost associated with $y$ is $Qy\in \L^J$, where $Q\in \L^{J\times N}$.

Given feasible choices of the decision variables $x\in\R^M$ and $y\in \L^N$, the risk associated with the second-stage cost vector $Qy\in \L^J$ is quantified via a multivariate convex risk measure $R\colon\L^J\to 2^{\R^J}$. The set $R(Qy)$ consists of the deterministic cost vectors in $\R^J$ that can make $Qy$ acceptable in the following sense:
\[
R(Qy) = \cb{z\in\R^J \mid Qy - z \in \A},
\]
where $\A=\cb{u\in\L^J \mid 0 \in R(u)}$ is the \emph{acceptance set} of the risk measure. Hence, $R(Qy)$ collects the deterministic cost reductions from $Qy$ that would yield an acceptable level of risk for the resulting random cost. Together with the deterministic cost vector $Cx$, the overall risk associated with $x$ and $y$ is given by the set
\[
Cx+R(Qy)= \cb{Cx+z\mid Qy-z\in\A}=\cb{Cx+z\mid z\in R(Qy)}=R(Cx+Qy),
\]
where the last equality holds thanks to the translativity property (A2).

Our aim is to calculate the ``minimal" vectors $z\in R(Cx+Qy)$ over all feasible choices of $x$ and $y$. Using \emph{vector optimization}, we formulate our problem as follows:
\begin{align*}
&\text{min}\; \;                    z\;\;\text{w.r.t.}\;\;{\R^J_+}\tag{$P_V$}\\
&\text{s.t. }\;\;                      z\in R(Cx + Qy)\\
& \quad\quad\;      Ax = b\\
& \quad\quad\;     T_i x + W_i y_i = h_i ,\quad  \forall i\in\O\\
& \quad\quad\;     z\in\R^J,\; x\in \R^M_+,\; y_i\in\R^N_+,\quad  \forall i\in\O.
\end{align*}
Let
\[
\X\coloneqq \cb{(x,y)\in\R^M_+\times \L^N_+\mid Ax=b,\; T_i x+W_iy_i=h_i, \; \forall i\in\O}.
\]
We assume that $\X$ is a compact set. Let us denote by $\mathscr{R}$ the image of the feasible region of $(P_V)$ under the objective function, that is,
\[
\mathscr{R}=\cb{z\in\R^J\mid z\in R(Cx+Qy), (x,y)\in \X}=\bigcup_{(x,y)\in\X}R(Cx+Qy).
\]
The \emph{upper image} of $(P_V)$ is defined as the set
\begin{equation}\label{upperimage}
\P = \cl \mathscr{R}= \cl \bigcup_{(x,y)\in \X} R(Cx+Qy).
\end{equation}
In particular, we have $\P\in\G$, that is, $\P$ is a closed convex upper set; see Remark~\ref{completelattice}.

Finding the ``minimal" $z$ vectors of $(P_V)$ is understood as computing the boundary of the set $\P$. For completeness, we recall the minimality notions for $(P_V)$.

\begin{definition}\label{weakminimizer}
A point $(x,y,z)\in\X\times\R^J$ is called a \emph{weak minimizer (weakly efficient solution)} of $(P_V)$ if $z\in R(Cx+Qy)$ and $z$ is a weakly minimal element of $\mathscr{R}$, that is, there exists no $z^\prime\in \mathscr{R}$ such that $z\in z^\prime+\R^J_{++}$.
\end{definition}

\begin{definition}\label{weaksolution}
(Definition~3.2 in \cite{convexbenson}) A set $\mathcal{Z}\subseteq \X\times\R^J$ is called a \emph{weak solution} of $(P_V)$ if the following conditions are satisfied:
\begin{enumerate}
\item \textbf{Infimality:} it holds $\cl\co\of{\cb{z\in\R^J\mid (x,y,z)\in \mathcal{Z}}+\R^J_+}=\P$,
\item \textbf{Minimality:} each $(x,y,z)\in\mathcal{Z}$ is a weak minimizer of $(P_V)$.
\end{enumerate}
\end{definition}

Ideally, one would be interested in computing a weak solution $\mathcal{Z}$ of $(P_V)$. However, except for some special cases (e.g. when the values of $R$ and the upper image $\P$ are polyhedral sets), such $\mathcal{Z}$ consists of infinitely many feasible points, that is, it is impossible to recover $\P$ using only finitely many values of $R$. Therefore, our aim is to propose algorithms to compute $\P$ approximately through finitely many feasible points.

\begin{definition}\label{approximatesolution}
(Definition~3.3 in \cite{convexbenson}) Let $\epsilon>0$. A nonempty finite set $\bar{\mathcal{Z}}\subseteq \X\times\R^J$ is called a \emph{finite weak $\epsilon$-solution} of $(P_V)$ if the following conditions are satisfied:
\begin{enumerate}
\item $\boldsymbol{\epsilon}$\textbf{-Infimality:} it holds $\co\of{\cb{z\in\R^J\mid (x,y,z)\in \bar{\mathcal{Z}}}}+\R^J_+-\epsilon\1\supseteq\P$,
\item \textbf{Minimality:} each $(x,y,z)\in\bar{\mathcal{Z}}$ is a weak minimizer of $(P_V)$.
\end{enumerate}
\end{definition}

As noted in \cite{convexbenson}, a finite weak $\epsilon$-solution $\bar{\mathcal{Z}}$ provides an outer and an inner approximation of $\P$ in the sense that
\begin{equation}\label{outerinner}
\co\of{\cb{z\in\R^J\mid (x,y,z)\in \bar{\mathcal{Z}}}}+\R^J_+-\epsilon\1\supseteq\P\supseteq \co\of{\cb{z\in\R^J\mid (x,y,z)\in \bar{\mathcal{Z}}}}+\R^J_+.
\end{equation}

Let us also introduce the \emph{weighted sum scalarization} problem with weight vector $w\in\R^J_+\sm\cb{0}$:
\begin{align}\label{weightedsum}
\text{min}\; w^{\mathsf{T}}z\quad \text{s.t.}\;\;z\in R(Cx+Qy),\; (x,y)\in\X.\tag{$P_1(w)$}
\end{align}
Define $\P_1(w)$ as the optimal value of $(P_1(w))$. For the remainder of this section, we provide a discussion on the existence of optimal solutions of $(P_1(w))$ as well as the relationship between $(P_1(w))$ and $(P_V)$.

\begin{proposition}\label{existenceP1}
Let $w\in\R^J_+\sm\cb{0}$. Then, there exists an optimal solution $(x,y,z)$ of $(P_1(w))$.
\end{proposition}

\begin{proof}
Note that $\P_1(w)=\inf_{(x,y)\in\X}\varphi_w(Cx+Qy)$, where $\varphi_w(\cdot)$ is the scalarization of $R$ by $w$ as defined in \eqref{phi}. Since $\varphi_w(\cdot)$ admits the dual representation in \eqref{scalar-dual}, it is a lower semicontinuous function on $\L^J$. Moreover, $\X$ is a compact set by assumption. By Theorem~2.43 in \cite{aliprantis}, it follows that an optimal solution of $(P_1(w))$ exists.
\end{proof}

\begin{remark}
Note that the feasible region $\cb{(x,y,z)\in \X\times\R^J\mid z\in R(Cx+Qy)}$ of $(P_V)$ is not compact in general due to the multivariate risk measure $R$, which has unbounded values. However, in \cite{convexbenson}, the feasible region of a vector optimization problem is assumed to be compact. Therefore, by assuming only $\X$ to be compact, Proposition~\ref{existenceP1} generalizes the analogous result in \cite{convexbenson}.
\end{remark}

The following proposition is stated in \cite{convexbenson} without a proof. It can be shown as a direct application of Theorem~5.28 in \cite{jahnbook}.

\begin{proposition}\label{weakeff}
(Proposition~3.4 in \cite{convexbenson}) Let $w\in\R^J_+\sm\cb{0}$. Every optimal solution $(x,y,z)$ of $(P_1(w))$ is a weak minimizer of $(P_V)$.
\end{proposition}

Proposition~\ref{weakeff} implies that, in the weak sense, solving $(P_V)$ is understood as solving the family $(P_1(w))_{w\in\R^J_+\sm\cb{0}}$ of weighted sum scalarizations.

\section{Convex Benson algorithms for $(P_V)$}\label{algorithms}

The convex Benson algorithms have a primal and a dual variant. While the primal approximation algorithm computes a sequence of outer approximations for the upper image $\P$ in the sense of \eqref{outerinner}, the dual approximation algorithm works on an associated vector maximization problem, called the \emph{geometric dual problem}. To explain the details of these algorithms, we should define the concept of geometric duality as well as a new scalarization problem $(P_2(v))$, called the problem of \emph{scalarization by a reference variable} $v\in\R^J$.

\subsection{The problem of scalarization by a reference variable}\label{p2section}

The problem $(P_2(v))$ is required to find the minimum step-length to enter the upper image $\P$ from a point $v\in\R^J \sm \P$ along the direction $\1=(1,\ldots,1)^{\mathsf{T}}\in\R^J$. It is formulated as
\begin{equation}\label{p2vdef}
\text{min}\; \a\quad \text{s.t.}\; v+\a\1 \in R(Cx+Qy),\; (x,y)\in\X,\; \a\in\R.\tag{$P_2(v)$}
\end{equation}
Note that $(P_2(v))$ is a scalar convex optimization problem with a set-valued constraint. We denote by $\P_2(v)$ the optimal value of $(P_2(v))$. We relax the set-valued constraint
$v+\a\1 \in R(Cx+Qy)$ in a Lagrangian fashion and obtain the following dual problem using the results of Section~3.2 in \cite{borwein}:
\begin{equation}\label{d2vdef}
\underset{\gamma\in\R^{J}_+}{\maxi}\; \inf_{(x,y)\in\X, \a\in\R}\of{\a+\inf_{z\in R(Cx+Qy)-v-\a\1} \gamma^\mathsf{T}z }.\tag{$LD_2(v)$}
\end{equation}
Note that $(LD_2(v))$ is constructed by rewriting the risk constraint of $(P_2(v))$ as $0\in R(Cx+Qy)-v-\a\1$ and calculating the support function of the set $R(Cx+Qy)-v-\a\1$ by the dual variable $\gamma\in\R^J_+$. The next proposition states the strong duality relationship between $(P_2(v))$ and $(LD_2(v))$.

\begin{proposition}\label{d2v}
(Theorem~19 and Equation~(3.23) in \cite{borwein}) Let $v\in\R^J$. Then, there exist optimal solutions $(x_{(v)},y_{(v)},\a_{(v)})$ of $(P_2(v))$ and $\gamma_{(v)}$ of $(LD_2(v))$, and the optimal values of the two problems coincide.
\end{proposition}

Finally, we recall the relationship between $(P_2(v))$ and $(P_V)$. The next proposition is provided without a proof since the proof in \cite{convexbenson} can be directly applied to our case.

\begin{proposition}\label{weakmin-primal}
(Proposition~4.5 in \cite{convexbenson}) Let $v\in\R^J$. If $(x_{(v)},y_{(v)},\a_{(v)})$ is an optimal solution of $(P_2(v))$, then $(x_{(v)},y_{(v)},v+\a_{(v)}\1)$ is a weak minimizer of $(P_V)$.
\end{proposition}

\subsection{Geometric duality}

Let $\W$ be the unit simplex in $\R^J$, that is,
\[
\W = \cb{w\in\R^J_+\mid w^{\mathsf{T}}\1=1}.
\]
For each $j\in \J$, let $e_{(j)}$ be the $j^{\text{th}}$ unit vector in $\R^J$, that is, the $j^{\text{th}}$ entry of $e^{(j)}$ is one and all other entries are zero.

The geometric dual of problem $(P_V)$ is defined as the vector maximization problem
\begin{align*}\tag{$D_V$}
&\text{max }\; \;  (w^1,\ldots,w^{J-1},\P_1(w))^\mathsf{T}\;\; \text{w.r.t.}\;\; K \\
  &\text{s.t. }\quad\; w \in \W,
\end{align*}
where $K$ is the so-called \emph{ordering cone} defined as $K = \cb{\lambda e_{(J)}\mid \lambda\geq 0}$. Similar to the upper image $\P$ of $(P_V)$, we can define the lower image $\D$ of $(D_V)$ as
\begin{equation*}
  \D := \cb{(w^1,\ldots,w^{J-1},p) \in \R^J\mid  w=(w^1,\ldots,w^{J-1},w^J) \in \W, p \leq \P_1(w)}.
\end{equation*}

\begin{remark}\label{completelatice3}
In analogy with Remark~\ref{completelattice}, the lower image $\D$ is a closed convex $K$-lower set, that is, $\cl\co(\D-K)=\D$.
\end{remark}

Next, we state the relationship between $\D$ and the optimal solutions of $(P_1(w))$, $(P_2(v))$, $(LD_2(v))$.

\begin{proposition}\label{geoduality}
(Proposition~3.5 in \cite{convexbenson}) Let $w\in\W$. If $(P_1(w))$ has a finite optimal value $\P_1(w)$, then $(w^1,\ldots,w^{J-1},\P_1(w))^{\mathsf{T}}$ is a boundary point of $\D$ and it is also a $K$-maximal element of $\D$, that is, there is no $d\in \D$ such that $d^J>\P_1(w)$.
\end{proposition}

\begin{proposition}\label{outerapprox}
(Propositions~4.6, 4.7 in \cite{convexbenson}) Let $v\in\R^J$. If $(x_{(v)},y_{(v)},\a_{(v)})$ is an optimal solution of $(P_2(v))$ and $\gamma_{(v)}$ is an optimal solution of $(LD_2(v))$, then $\gamma_{(v)}$ is a maximizer of $(D_V)$, that is, $(\gamma_{(v)}^1,\ldots,\gamma_{(v)}^{J-1},\P_1(\gamma_{(v)}))^{\mathsf{T}}$ is a $K$-maximal element of the lower image $\D$. Moreover, $\{z \in \mathbb{R}^J\mid  \gamma_{(v)}^\mathsf{T}z\geq   \gamma_{(v)}^\mathsf{T}(v + \alpha_{(v)}\1) \}$ is a supporting halfspace of $\P$ at the point $(v+\a_{(v)}\1)$.
\end{proposition}

\begin{proposition}\label{outerapproxD}
Let $w\in\R^J_+\sm\cb{0}$. If $(x_{(w)},y_{(w)},z_{(w)})$ is an optimal solution of $(P_1(w))$, then $\{d\in\R^J\mid (z_{(w)}^J-z_{(w)}^1,\ldots,z_{(w)}^J-z_{(w)}^{J-1},1)^\mathsf{T}d\leq z_{(w)}^J\}$ is a supporting halfspace of $\D$ at the point $(w^1,\ldots,w^{J-1},\P_1(w))$.
\end{proposition}

\begin{proof}
From Proposition~\ref{geoduality}, $d\coloneqq (w^1,\ldots,w^{J-1},\P_1(w))$ is a boundary point of $\D$. Moreover, it follows that
\[
(z_{(w)}^J-z_{(w)}^1,\ldots,z_{(w)}^J-z_{(w)}^{J-1},1)^\mathsf{T}d=-w^{\mathsf{T}}z_{(w)}+z_{(w)}^J+\P_1(w)=z^J
\]
since $\P_1(w)=w^{\mathsf{T}}z_{(w)}$. Hence, the assertion of the proposition follows.
\end{proof}

\begin{proposition}\label{approximation}
(Proposition~3.10 in \cite{convexbenson}) Let $\epsilon>0$.
\begin{enumerate}[(a)]
\item Let $\bar{\mathcal{Z}}$ be a finite weak $\epsilon$-solution of $(P_V)$. Then,
\[
\P^{in}(\bar{\mathcal{Z}})\coloneqq\co\of{\cb{z\in\R^J\mid (x,y,z)\in\bar{\mathcal{Z}}}}+\R^J_+
\]
is an inner approximation of the upper image $\P$, that is, $\P^{in}(\bar{\mathcal{Z}})\subseteq \P$. Moreover,
\[
\D^{out}(\bar{\mathcal{Z}})=\cb{d\in\R^J\mid \of{z^J-z^1,\ldots,z^J-z^{J-1},1}^\mathsf{T}d\leq z^J, \;\forall z\in\bar{\mathcal{Z}}}
\]
is an outer approximation of the lower image $\D$, that is, $\D\subseteq \D^{out}(\bar{\mathcal{Z}})$.
\item Let $\bar{\W}$ be a finite $\epsilon$-solution of $(D_V)$. Then,
\[
\D^{in}(\bar{\W})\coloneqq\co(\cb{(w^1,\ldots,w^{J-1},\P_1(w))^{\mathsf{T}}\mid w\in \bar{\W}})-K
\]
is an inner approximation of $\D$, that is, $\D^{in}(\bar{\W})\subseteq \D$. Moreover,
\[
\P^{out}(\bar{\W})=\cb{z\in\R^J\mid w^\mathsf{T}z\geq \P_1(w), \;\forall w\in\bar{\W}}
\]
is an outer approximation of $\P$, that is, $\P\subseteq \P^{out}(\bar{\W})$.
\end{enumerate}
\end{proposition}

The problems $(P_1(w))$, $(P_2(v))$ and the above propositions form a basis for the primal and dual convex Benson algorithms. These algorithms are explained briefly in the following sections.

\subsection{Primal algorithm}

The primal algorithm starts with an initial outer approximation $\P^0$ for the upper image $\P$. To construct $\P^0$, for each $j\in\J$, the algorithm computes the supporting halfspace of $\P$ with direction vector $e_{(j)}$ by solving the weighted-sum scalarization problem $(P_1(e_{(j)}))$. If $(x_{(j)},y_{(j)},z_{(j)})$ is an optimal solution of $(P_1(e_{(j)}))$, then this halfspace supports the upper image $\P$ at the point $z_{(j)}$. Then, $\P^0$ is defined as the intersection of these $J$ supporting halfspaces.

The algorithm iteratively obtains a sequence  $\P^0\supseteq\P^1\supseteq\P^2\supseteq\ldots\supseteq\P$ of finer outer approximations, it updates a set $\bar{\mathcal{Z}}$ and $\bar{\W}$ of weak minimizers and maximizers for $(P_V)$ and $(D_V)$, respectively. At iteration $k$, the algorithm first computes $\mathcal{V}^k$, that is the set of all vertices of $\P^k$. For each vertex $v\in\mathcal{V}^k$, an optimal solution $(x_{(v)},y_{(v)},\a_{(v)})$ to $(P_2(v))$ is computed. The optimal $\a_{(v)}$ is the minimum step-length required to find a  boundary point $(v+\a_{(v)}\1)$ of $\P$. Since the triplet $(x_{(v)},y_{(v)},v+\a_{(v)}\1)$ is a weak minimizer of $(P_V)$ by Proposition~\ref{weakmin-primal}, it is added to the set $\bar{\mathcal{Z}}$. Then, an optimal solution $\gamma_{(v)}$ of the dual problem $(LD_2(v))$ is computed, which is a maximizer for $(D_V)$ (\emph{see} Proposition~\ref{outerapprox}) and is added to the set $\bar{\W}$. This procedure is continued until a vertex $v$ with a step-length greater than an error parameter $\epsilon>0$ is detected. For such $v$, using Proposition~\ref{outerapprox}, a supporting halfspace of $\P$ at point $(v+\a_{(v)}\1)$ is obtained. The outer approximation is updated as $\P^{k+1}$ by intersecting $\P^k$ with this supporting halfspace. The algorithm terminates when all the vertices are in $\epsilon$-distance to the upper image $\P$.

At the termination, the algorithm computes inner and outer approximations $\P^{in}(\bar{\mathcal{Z}}), \P^{out}(\bar{\W})$ for the upper image $\P$ and $\D^{in}(\bar{\W}), \D^{out}(\bar{\mathcal{Z}})$ for the lower image $\D$ using Proposition~\ref{approximation}. Note that both $\P^{out}(\bar{\W})$ and $\P^{k}$ are outer approximations for $\P$. However, $\P^{out}(\bar{\W})$ is a finer outer approximation than $\P^{k}$. The reason is that when $\P^k$ is updated, only the vertices in more than $\epsilon$-distance to $\P$ are used. On the other hand, all the vertices are considered when calculating $\P^{out}(\bar{\W})$. Furthermore, the algorithm returns a finite weak $\epsilon$-solution $\bar{\mathcal{Z}}$ to $(P_V)$ and a finite $\epsilon$-solution $\bar{\W}$ to $(D_V)$ (\emph{see} Theorem~4.9 in \cite{convexbenson}).

The steps of the primal algorithm are provided as Algorithm~\ref{primalalg}.

\begin{algorithm*}[h!]                  
\caption{Primal Approximation Algorithm }    
\label{primalalg}                           
\begin{algorithmic}[1]
    \STATE Compute an optimal solution $(x_{(j)}, y_{(j)}, z_{(j)})$ to $(P_1(e_{(j)}))$ for each $j \in\J$;

    \STATE Let $\P^0 = \{z \in \mathbb{R}^J: e_{(j)}^\mathsf{T}z \geq \P_1(e_{(j)}), \;\forall j\in\J\}$;

    \STATE $k \leftarrow 0$; \\
            $\bar{\mathcal{Z}} \leftarrow \{(x_{(j)}, y_{(j)},z_{(j)}) \mid  j \in\J\}$;\\
            $\bar{\W} \leftarrow \{e_{(j)} \mid  j\in\J\}$; \\
    \REPEAT

    \STATE $\mathcal{M} \leftarrow \mathbb{R}^J$;

    \STATE Compute the set $\mathcal{V}^k$ of the vertices of $\P^k$;

    \FOR{each $v\in\mathcal{V}^k$}


    \STATE Compute an optimal solution $(x_{(v)},y_{(v)},\alpha_{(v)})$ of $(P_2(v))$ and an optimal solution $\gamma_{(v)}$ of $(LD_2(v))$;

    \STATE $\bar{\mathcal{Z}} \leftarrow \bar{\mathcal{Z}} \cup \{(x_{(v)},y_{(v)},v + \alpha_{(v)}\textbf{1})\}$;\\

          $\bar{\W} \leftarrow \bar{\W} \cup \{\gamma_{(v)}\}$;\\


    \IF{$\alpha_{(v)} > \epsilon$}

        \STATE $\mathcal{M} \leftarrow \mathcal{M} \cap \left\{z \in \mathbb{R}^J: \gamma_{(v)}^\mathsf{T}z \geq  \gamma_{(v)}^\mathsf{T}(v + \alpha_{(v)}\textbf{1}) \right\}$ ;
         \STATE break;

    \ENDIF
    \ENDFOR
    \IF{$\mathcal{M} \neq \mathbb{R}^J$}
        \STATE $\P^{k+1} \leftarrow \P^k \cap \mathcal{M}$, $k \leftarrow k+1$;

    \ENDIF
    \UNTIL{$\mathcal{M} = \mathbb{R}^J$};
    \STATE Compute $\P^{in}(\bar{\mathcal{Z}}), \P^{out}(\bar{\W}), \D^{in}(\bar{\W}), \D^{out}(\bar{\mathcal{Z}})$ as in Proposition~\ref{approximation};
    \RETURN $\left\{
              \begin{array}{l}
                \bar{\mathcal{Z}} \hbox{: A finite weak $\epsilon$-solution to $(P_V)$;} \\
                \bar{\W} \hbox{: A finite $\epsilon$-solution to $(D_V)$;} \\
                \P^{in}(\bar{\mathcal{Z}}), \P^{out}(\bar{\W}), \D^{in}(\bar{\W}), \D^{out}(\bar{\mathcal{Z}});
              \end{array}
            \right.$
\end{algorithmic}
\end{algorithm*}

\subsection{Dual algorithm}

The steps of the dual algorithm follow in a way that is similar to the primal algorithm; however, as a major difference, at each iteration, an outer approximation for the dual image $\D$ is obtained. Moreover, the dual algorithm does not require solving $(P_2(v))$; only $(P_1(w))$ is solved for different weights $w$ in the initialization step as well as in the iterations. An optimal solution of $(P_1(w))$ is used to update the outer approximation of $\D$ as in Proposition~\ref{outerapproxD}.

At the termination, the algorithm computes inner and outer approximations for the upper image $\P$ and lower image $\D$ using Proposition~\ref{approximation}. Furthermore, the algorithm returns a finite weak $\epsilon$-solution $\bar{\mathcal{Z}}$ to $(P_V)$ and a finite $\epsilon$-solution $\bar{\W}$ to $(D_V)$ (\emph{see} Theorem~4.14 in \cite{convexbenson}).

The steps of the dual algorithm are provided as Algorithm~\ref{alg1}.

\begin{algorithm*}[h!]                  
\caption{Dual Approximation Algorithm }    
\label{alg1}                           
\begin{algorithmic}[1]
    \STATE Compute an optimal solution $(x_{(\eta)}, y_{(\eta)}), z_{(\eta)})$ to $(P_1(\eta))$ for $\eta =(\frac{1}{J},\ldots,\frac{1}{J})^\mathsf{T}$ ;

    \STATE Let $\D^0 = \{d\in\R^J \mid \P_1(\eta) \geq d^J \} $;

    \STATE $k \leftarrow 0$; \\
           $\bar{\mathcal{Z}} \leftarrow \{(x_{(\eta)}, y_{(\eta)},z_{(\eta)})\}$;\\
           $\bar{\W} \leftarrow \{\eta\}$; \\

    \REPEAT

    \STATE $\mathcal{M} \leftarrow \mathbb{R}^J$;

    \STATE Compute the set $\mathcal{V}^k$ of vertices of $\D^k$;

    \FOR{each $t=(t^1,\ldots,t^{J-1},t^J)^{\mathsf{T}}\in\mathcal{V}^k$}

     \STATE Let $w=(t^1,\ldots,t^{J-1},1-\sum_{j=1}^{J-1}t^j)^{\mathsf{T}}$;

    \STATE Compute an optimal solution $(x_{(w)}, y_{(w)}, z_{(w)})$ to $(P_1(w))$;

    \STATE $\bar{\mathcal{Z}} \leftarrow \bar{\mathcal{Z}} \cup \{(x_{(w)}, y_{(w)},z_{(w)})\}$; \\

     \IF{$w \in\mathbb{R}_{++}^J$ or $t^J - \P_1(w) \leq \epsilon$}
        \STATE $\bar{\W} \leftarrow \bar{\W} \cup \{w\}$; \\

     \ENDIF

    \IF{$t^J - \P_1(w) > \epsilon$}
        \STATE $\mathcal{M} \leftarrow \mathcal{M} \cap \cb{d\in\R^J\mid (z_{(w)}^J-z_{(w)}^1,\ldots,z_{(w)}^J-z_{(w)}^{J-1},1)^\mathsf{T}d\leq z_{(w)}^J}$;
        \STATE break;
    \ENDIF
    \ENDFOR
     \IF{$\mathcal{M} \neq \mathbb{R}^J$}
        \STATE $\D^{k+1} \leftarrow \D^k \cap \mathcal{M}$, $k \leftarrow k+1$;
    \ENDIF
    \UNTIL{$\mathcal{M} = \mathbb{R}^J$};
    \STATE Compute $\P^{in}(\bar{\mathcal{Z}}), \P^{out}(\bar{\W}); \D^{in}(\bar{\W}), \D^{out}(\bar{\mathcal{Z}})$ as in Proposition~\ref{approximation};
    \RETURN $\left\{
              \begin{array}{l}
                \bar{\mathcal{Z}} \hbox{: A finite weak $\epsilon$-solution to $(P_V)$;} \\
                \bar{\W} \hbox{: A finite $\epsilon$-solution to $(D_V)$;} \\
                \P^{in}(\bar{\mathcal{Z}}), \P^{out}(\bar{\W}), \D^{in}(\bar{\W}), \D^{out}(\bar{\mathcal{Z}});
              \end{array}
            \right.$
\end{algorithmic}
\end{algorithm*}

\section{Scenario decomposition for scalar problems}\label{scalarization}

In this section, we are interested in solving the scalarization problems $(P_1(w))$ and $(P_2(v))$. Note that these problems are single-objective multivariate risk-averse two-stage stochastic programming problems. For such problems, the problem size increases as the number of scenarios, $I$, gets larger. An efficient solution procedure is possible by scenario-wise decompositions. In the univariate case, for risk-neutral two-stage stochastic programming problems, see \cite{birgebook}, \cite{birgepaper}, \cite{kall}, \cite{ruszdecomp}, \cite{wets} for solution methodologies by scenario-wise decomposition. For scenario decompositions in two-stage risk-averse stochastic programming problems, the reader is refered to \cite{ahmed}, \cite{miller}, \cite{fabian}, \cite{kristoffersen} for problems with a single coherent risk-averse objective function and to \cite{chanceconstrained} for chance-constrained problems. Scenario-wise decomposition solution methodology is also possible for multi-stage stochastic programming problems with dynamic coherent risk measure as suggested in \cite{collado}. Different from these studies, the scalarization problems we solve are two-stage risk-averse stochastic programming problems with \emph{multivariate convex} risk measures; therefore, these problems require different solutions techniques than the existing ones.

\subsection{The problem of weighted sum scalarization}

Let $w\in\R^J_+\sm\cb{0}$. The weighted sum scalarization problem $(P_1(w))$ defined in Section~\ref{problem} can be rewritten more explicitly as:
\begin{align*}
&\text{min }\; \;                    w^{\mathsf{T}}z\tag{$P_1(w)$}\\
&\text{s.t.}\quad \;                      z\in R(Cx + Qy)\\
& \quad\quad\;\;    Ax = b\\
& \quad\quad\;\;     T_i x + W_i y_i = h_i, \quad  \forall i\in\O\\
& \quad\quad\;\;      z\in\R^J,\; x\in \R^M_+,\; y_i\in\R^N_+,\quad  \forall i\in\O.
\end{align*}

We propose a Lagrangian dual reformulation of $(P_1(w))$ whose objective function is scenario-wise decomposable. The details are provided in Section~\ref{P1decompose}. Based on this dual reformulation, in Section~\ref{P1cuttingplane}, we propose a dual cutting-plane algorithm for $(P_1(w))$, called the \emph{dual bundle method}, which provides an optimal dual solution. As the Benson algorithms in Section~\ref{algorithms} require an optimal primal solution in addition to an optimal dual solution, in Section~\ref{P1recovery}, we show that such a primal solution can be obtained from the dual of the so-called \emph{master problem} in the dual bundle method.

\subsubsection{Scenario decomposition}\label{P1decompose}

To derive a decomposition algorithm for $(P_1(w))$, we \emph{randomize} the first stage variable $x\in\R^M$ and treat it as an element $x\in \L^M$ with realizations $x_1,\ldots,x_I\in\R^M$. To ensure the equivalence of the new formulation with the previous one, we add the so-called \emph{nonanticipativity constraints}
\[
p_i(x_i -\E\sqb{x})  = 0, \; \forall i\in\O,
\]
which are equivalent to $x_1=\ldots=x_I$.

Let us introduce
\begin{equation}\label{f-defn}
\F\coloneqq\cb{(x,y)\in \L^M\times \L^N \mid (x_i,y_i)\in  \F_i,\; \forall i\in\O},
\end{equation}
where, for each $i\in\O$,
\[
\F_i \coloneqq \cb{(x_i,y_i)\in\R^M_+ \times \R^N_+\mid Ax_i = b,\; T_i x_i + W_i y_i = h_i}.
\]
With this notation and using the nonanticipativity constraints, we may rewrite $(P_1(w))$ as follows:
\begin{align*}
&\text{min}\; \;                     w^\mathsf{T}z\tag{$P^\prime_1(w)$}\\
&\text{s.t.}\quad                      z\in R(Cx+Qy)\\
& \quad\quad\;\;      p_i\of{x_i - \E\sqb{x}} = 0,\quad \forall i\in\O\\
& \quad\quad\;\;     (x,y)\in \F,\; z\in \R^J.
\end{align*}
Note that the optimal value of $(P_1^\prime(w))$ is $\mathscr{P}_1(w)$.

The following theorem provides a dual formulation of $(P^\prime_1(w))$ by relaxing the nonanticipativity constraints in a Lagrangian fashion. We call this dual formulation as $(D_1(w))$.

\begin{theorem}\label{p1thm}
It holds
\[
\P_1(w)=\sup_{\mu\in \M_1^J,\lambda\in\L^M} \cb{\sum_{i\in\O}  f_i(\mu_i,\lambda_i,w)-\b(\mu,w)\mid \E\sqb{\lambda}=0},\tag{$D_1(w)$}
\]
where, for each $i\in\O$, $\mu_i\in\R^J_+$, $\lambda_i\in\R^M$,
\begin{equation}\label{scenariosubp}
f_i(\mu_i,\lambda_i,w)\coloneqq \inf_{(x_i,y_i)\in\F_i} \of{w^{\mathsf{T}}\sqb{\mu_i \cdot (Cx_i+Q_i y_i)}+p_i \lambda_i^{\mathsf{T}}x_i},
\end{equation}
and $\beta$ is defined by \eqref{beta}.
\end{theorem}

\begin{proof}
We may write
\begin{align}
\P_1(w) &= \inf_{(x,y)\in \F, z\in\R^J}\cb{w^{\mathsf{T}}z\mid z\in R(Cx+Qy),\; p_i\of{x_i - \E\sqb{x}} = 0,\; \forall i\in\O}\label{p1eq1}\\
 &=\inf_{(x,y)\in \F}\cb{\inf_{z\in R(Cx+Qy)}w^{\mathsf{T}}z \mid p_i\of{x_i - \E\sqb{x}} = 0,\; \forall i\in\O}\label{p1eq2}\\
 &=\inf_{(x,y)\in \F}\cb{\sup_{\mu\in\M_1^J}\of{w^{\mathsf{T}}\E^\mu\sqb{Cx+Qy}-\b(\mu,w)} \mid p_i\of{x_i - \E\sqb{x}} = 0,\; \forall i\in\O},\label{p1eq3}
\end{align}
where the passage to the last line is by \eqref{scalar-dual}. Using the minimax theorem of \cite{sion}, we may interchange the infimum and the supremum in the last line. This yields
\begin{align}\label{P1F}
\P_1(w)
=\sup_{\mu\in \M_1^J}\of{F(\mu,w) -\b(\mu,w)},
\end{align}
where, for each $\mu\in\M_1^J$,
\begin{equation}\label{Ffunction}
F(\mu,w)\coloneqq \inf_{(x,y)\in\F}\cb{w^{\mathsf{T}}\E^\mu\sqb{Cx+Qy}\mid p_i\of{x_i - \E\sqb{x} }= 0,\; \forall i\in\O}.
\end{equation}
Let us fix $\mu\in\M_1^J$. Note that $F(\mu,w)$ is the optimal value of a large-scale linear program where the only coupling constraints between the decision variables for different scenarios are the nonanticipativity constraints. To obtain a formulation of this problem that can be decomposed into a subproblem for each scenario, we dualize the nonanticipativity constraints. The reader is referred to Section~2.4.2 of \cite{shapiro} for the details on the dualization of nonanticipativity constraints. To that end, let us assign Lagrange multipliers $\tilde{\lambda}_1, \ldots,\tilde{\lambda}_I \in \R^M$ for the non-anticipativity constraints. Note that we may consider them as the realizations of a random Lagrange multiplier $\tilde{\lambda}\in\L^M$. By strong duality for linear programming,
\begin{align*}
F(\mu,w)&=\sup_{\tilde{\lambda}\in \L^M} \inf_{(x,y)\in\F} \ell (x,y,\tilde{\lambda}),
\end{align*}
where the Lagrangian $\ell$ is defined by
\begin{align*}
\ell (x,y,\tilde{\lambda})&\coloneqq w^{\mathsf{T}}\E^\mu\sqb{Cx+Qy}+\sum_{i\in\O}p_i {\tilde{\lambda}_i}^\mathsf{T}\of{x_i - \E\sqb{x}}  \\
&=w^{\mathsf{T}}\E^\mu\sqb{Cx+Qy}+\sum_{i\in\O}p_i \big(\tilde{\lambda}_i - \E\big[\tilde{\lambda}\big]\big)^{\mathsf{T}}x_i,
\end{align*}
for each $x\in \L^M, y\in\L^N, \tilde{\lambda}\in\L^M$. Given such $x,y,\tilde{\lambda}$, note that
\begin{align*}
\ell(x,y,\tilde{\lambda})=\ell(x,y,\lambda)&= w^{\mathsf{T}}\E^\mu\sqb{Cx+Qy}+\sum_{i\in\O}p_i \lambda_i^{\mathsf{T}}x_i\\
&=w^{\mathsf{T}}\E^\mu\sqb{Cx+Qy}+\E\sqb{\lambda^\mathsf{T}x},
\end{align*}
if we set $\lambda \coloneqq \tilde{\lambda} - \E\big[\tilde{\lambda}\big]$. In this case, $\E\sqb{\lambda} = 0$. Therefore, we obtain
\begin{align}
F(\mu,w)&=\sup_{\lambda\in \L^M}\cb{ \inf_{(x,y)\in\F} \ell (x,y,\lambda)\mid \E\sqb{\lambda}=0}\notag\\
&= \sup_{\lambda\in \L^M} \cb{\inf_{(x,y)\in\F} \of{w^{\mathsf{T}}\E^\mu\sqb{Cx+Qy}+\E\sqb{\lambda^\mathsf{T}x}}\mid \E\sqb{\lambda}=0}\notag\\
&= \sup_{\lambda\in \L^M}\cb{\sum_{i\in\O} f_i(\mu_i,\lambda_i,w)\mid\E\sqb{\lambda}=0},\label{Flagrange}
\end{align}
where $f_i(\mu_i,\lambda_i,w)$, defined by \eqref{scenariosubp}, is the optimal value of the subproblem for scenario $i\in\O$. The assertion of the theorem follows from \eqref{P1F} and \eqref{Flagrange}.
\end{proof}

\subsubsection{The dual bundle method}\label{P1cuttingplane}

To solve $(D_1(w))$ given in Theorem~\ref{p1thm}, we propose a dual bundle method which constructs \emph{affine} upper approximations for $f_i(\cdot,\cdot,w)$, $i\in\O$, and $-\beta(\cdot,w)$. The upper approximations are based on the subgradients of these functions at points $(\mu^{(\ell)},\lambda^{(\ell)})$ that are generated iteratively by solving the so-called \emph{master problem}. The reader is referred to \cite{ruszbook} for the details of the bundle method.

For $i\in\O,\mu^\prime_i\in\R^J_+,\lambda^\prime\in\R^M$, we denote by $\partial_{\mu_i,\lambda_i}f_i(\mu^\prime_i,\lambda^\prime_i,w)$ the subdifferential of the concave function $f_i(\cdot,\cdot,w)$ at the point $(\mu^\prime_i,\lambda^\prime_i)$, that is, $\partial_{\mu_i,\lambda_i} f_i(\mu^\prime_i,\lambda^\prime_i,w)$ is the set of all vectors $(g_{\mu^\prime_i},g_{\lambda^\prime_i})\in\R^{J+M}$ such that
\begin{equation}\label{fcut}
f_i(\mu_i,\lambda_i,w)\leq f_i(\mu^\prime_i,\lambda^\prime_i,w)+g_{\mu^\prime_i}^{\mathsf{T}}(\mu_i-\mu^\prime_i)+g_{\lambda^\prime_i}^{\mathsf{T}}(\lambda_i-\lambda^\prime_i),
\end{equation}
for all $\mu_i\in\R^J_+, \lambda_i\in\R^M$. Note that the right hand side of \eqref{fcut} provides an \emph{affine} upper approximation for $f_i(\cdot,\cdot,w)$. For this reason, \eqref{fcut} is called as a \emph{cut}.

Similarly, we denote by $\partial_{\mu}(-\beta)(\mu^\prime,w)$ the subdifferential of the concave function $-\beta(\cdot,w)$ at a point $\mu^\prime\in\M_1^J$, which is the set of all vectors $\rho_{\mu^\prime}=(\rho_{\mu^\prime_1},\ldots,\rho_{\mu^\prime_I})\in\R^{J\times I}$ such that
\begin{equation}\label{betacut}
-\beta(\mu,w)\leq -\beta(\mu^\prime,w)+\sum_{i\in\O}\rho_{\mu^\prime_i}^{\mathsf{T}}(\mu_i-\mu^\prime_i)
\end{equation}
for all $\mu\in\M_1^J$. We call \eqref{betacut} a \emph{cut} for $-\beta(\cdot,w)$.

In the next proposition, we show how to compute the subdifferential of the function $f_i(\cdot,\cdot,w)$ at a point $(\mu_i^\prime,\lambda^\prime_i)$.

\begin{proposition}\label{subgradientf}
For $i\in\O,\mu^\prime_i\in\R^J_+,\lambda_i^\prime\in\R^M$, let
\begin{equation*}
A_i(\mu^\prime_i,\lambda^\prime_i,w) \coloneqq \argmin_{(x_i,y_i)\in\F_i}\of{w^{\mathsf{T}}\sqb{\mu^\prime_i \cdot (Cx_i+Q_i y_i)}+p_i (\lambda^\prime_i)^{\mathsf{T}}x_i}.
\end{equation*}
Then,
\begin{align*}
&\partial_{\mu_i,\lambda_i}f_i(\mu^\prime_i,\lambda^\prime_i,w)=\co \cb{\of{w\cdot( Cx_i + Q_i y_i),p_i x_i}\mid (x_i,y_i)\in A_i(\mu^\prime_i,\lambda^\prime_i,w)}.\\
\end{align*}
\end{proposition}

\begin{proof}
Let $\varphi_i(x_i,y_i,\mu_i,\lambda_i,w)\coloneqq w^{\mathsf{T}}\sqb{\mu_i \cdot (Cx_i+Q_i y_i)}+p_i \lambda_i^{\mathsf{T}}x_i$. The function $\varphi_i(x_i,y_i,\cdot,\cdot,w)$ is affine for all $(x_i,y_i)\in\F_i$. The function $\varphi_i(\cdot,\cdot,\mu_i,\lambda_i,w)$ is also affine and continuous for all $\mu_i\in\R^J_+, \lambda_i\in\R^M$. Finally, the set $\F_i$ is compact by assumption. By Theorem~2.87 in \cite{ruszbook}, the assertion of the proposition follows.
\end{proof}

Next, we show how to compute a subgradient of the function $-\beta(\cdot,w)$ at a point $\mu^\prime$.

\begin{proposition}\label{subgradientbeta}
Recall that the set $\A=\cb{u\in\L^J\mid 0\in R(u)}$ is the acceptance set of $R$. For $\mu^\prime\in\L^J_+$, let
\begin{equation*}
B(\mu^\prime,w) \coloneqq \argmax_{u\in\A}w^{\mathsf{T}}\E^{\mu^\prime }\sqb{u}
\end{equation*}
and assume that $B(\mu^\prime,w)\neq \emptyset$.
Then,
\begin{align*}
&\partial_{\mu}(-\beta)(\mu^\prime,w)\supseteq \co \cb{(w\cdot u_1,\ldots,w\cdot u_I) \mid u=(u_1,\ldots,u_I)\in B(\mu^\prime,w)}.\\
\end{align*}
\end{proposition}

\begin{proof}
Let $\varphi(u,\mu,w)\coloneqq w^{\mathsf{T}}\E^{\mu}\sqb{u}$. The function $\varphi(u,\cdot,w)$ is affine for all $u\in\A$. The function $\varphi(\cdot,\mu,w)$ is also affine and continuous for all $\mu\in\M_1$. By Theorem~2.87 in \cite{ruszbook}, the assertion of the proposition follows.
\end{proof}

\begin{remark}
For practical risk measures, such as the multivariate entropic risk measure (\emph{see} Example~\ref{multient}), the function $-\beta(\cdot,w)$ is differentiable and the subdifferential is a singleton. For coherent multivariate risk measures, such as the multivariate CVaR (\emph{see} Example~\ref{multiCVaR}), there exists a convex cone $\mathcal{Q}\subseteq\M_1^J$ such that $-\beta(\mu,w)=0$ if $\mu\in \mathcal{Q}$ and $-\beta(\mu,w)=-\infty$ otherwise. For multivariate CVaR with risk-aversion parameter $\nu\in (0,1)^J$,
\[
\mathcal{Q}=\cb{\mu\in\M_1^J\;\Big\vert\; \frac{\mu^j_i}{p_i}\leq \frac{1}{1-\nu^j},\;\forall i\in\O,j\in\J}.
\]
For a coherent multivariate risk measure, $\partial_{\mu}(-\beta)(\mu^\prime)$ is the set of all normal directions of $\mathcal{Q}$ at $\mu^\prime$. It follows that the cut \eqref{betacut} is always satisfied; therefore, it can be ignored.
\end{remark}

At each iteration $k$ of the bundle method, we solve the \emph{master problem}
\begin{align}
&\text{ max }\; \;          \sum_{i\in\O}\vartheta_i+\eta-\sum_{i\in \O}\varrho \norm{\mu_i-\bar{\mu}^{(k)}_i}^2 -\sum_{i\in \O}\varrho \norm{\lambda_i-\bar{\lambda}^{(k)}_i}^2 \;\tag{$MP_1(w)$}\\
&\text{ s.t. }\quad\;                      \vartheta_i \leq f_i(\mu_i^{(\ell)},\lambda_i^{(\ell)},w) + g_{\mu^{(\ell)}_i}^{\mathsf{T}}(\mu_i-\mu^{(\ell)}_i) + g_{\lambda^{(\ell)}_i}^{\mathsf{T}}(\lambda_i-\lambda_i^{(\ell)}), \quad \forall i\in\O,  \ell\in\mathcal{L}\label{grad1}\\
& \;\quad\quad\quad  \eta \leq -\beta(\mu^{(\ell)},w) + \sum_{i\in\O}\rho_{\mu^{(\ell)}_i}^{\mathsf{T}}(\mu_i-\mu_i^{(\ell)}), \quad \forall \ell\in\mathcal{L}\label{grad2}\\
& \;\quad\quad\quad      \sum_{i\in\O}p_i\lambda_i=0\label{explambda}\\
& \;\quad\quad\quad      \sum_{i\in\O}\mu_i=\mathbf{1}\label{muprob}\\
& \;\quad\quad\quad      \mu_i\in\R^J_+, \lambda_i\in \R^M, \vartheta_i\in\R,\quad \forall i\in\O\label{mupos}\\
& \;\quad\quad\quad   \eta\in\R,
\end{align}
with $\mathcal{L}=\cb{1,\ldots,k}$, $\varrho>0$. Here, $\norm{\cdot}$ denotes the Euclidean norm on an appropriate dimension. Note that constraints \eqref{muprob} and \eqref{mupos} for $\mu$ are equivalent to having $\mu\in\M_1^J$, and constraint \eqref{explambda} for $\lambda$ is equivalent to having $\E[\lambda]=0$. $\bar{\mu}^{(k)}\in\M_1^J, \bar{\lambda}^{(k)}\in\L^M$ with $\E[\bar{\lambda}^{(k)}]=0$ are parameters of the problem, called the \emph{centers}, that are initialized and updated within the bundle method. The quadratic terms in the objective function are Moreau-Yosida regularization terms and they make the overall objective function strictly convex. These regularization terms enforce an optimal solution of $(MP_1(w))$ to be close to the centers.

Let $(\mu^{(k+1)},\lambda^{(k+1)},\vartheta^{(k+1)},\eta^{(k+1)})$ be an optimal solution for $(MP_1(w))$. Computing the subgradients
\begin{align*}
&(g_{\mu_i^{(k+1)}},g_{\lambda_i^{(k+1)}})\in\partial_{\mu_i,\lambda_i}f_i(\mu_i^{(k+1)},\lambda_i^{(k+1)}),\;\forall  i\in\O,\\
& (\rho_{\mu_1^{(k+1)}},\ldots,\rho_{\mu_I^{(k+1)}} ) \in \partial_{\mu}\beta(\mu^{(k+1)},w)
\end{align*}
at this optimal solution and using \eqref{fcut} and \eqref{betacut}, a cut for each of the functions $f_i(\cdot,\cdot,w), i\in\O,$ and $-\beta(\cdot,w)$ are added to $(MP_1(w))$ at the next iteration in order to improve the upper approximations for these functions.

The centers are updated in the following fashion. At iteration $k$, one checks if the difference between the objective value of $(D_1(w))$ evaluated at the point $(\mu^{(k)},\lambda^{(k)})$, that is, $\sum_{i\in\O}f_i(\mu_i^{(k)},\lambda_i^{(k)},w)-\beta(\mu^{(k)},w)$, and the objective value evaluated at the centers $\bar{\mu}^{(k-1)},\bar{\lambda}^{(k-1)}$, that is, $\sum_{i\in\O}f_i(\bar{\mu}^{(k-1)}_i,\bar{\lambda}^{(k-1)}_i,w)-\beta(\bar{\mu}^{(k-1)},w)$, is larger than a threshold. If so, this means that an optimal solution of $(D_1(w))$ is close to $(\mu^{(k)},\lambda^{(k)})$. Therefore, the new centers $\bar{\mu}^{(k)},\bar{\lambda}^{(k)}$ are set to $\mu^{(k)},\lambda^{(k)}$, respectively. This is called a \emph{descent step}. Otherwise, the centers remain unchanged, that is, $\bar{\mu}^{(k)},\bar{\lambda}^{(k)}$ are set to $\bar{\mu}^{(k-1)},\bar{\lambda}^{(k-1)}$, respectively.

The steps of our dual bundle method are provided as Algorithm~\ref{decomp1}. By \cite[Theorem~7.16]{ruszbook}, the bundle method generates a sequence $(\bar{\mu}^{(k)},\bar{\lambda}^{(k)})_{k\in\mathbb{N}}$ that converges to an optimal solution of $(D_1(w))$ as $k\rightarrow\infty$. In practice, the stopping condition in line~22 of Algorithm~\ref{decomp1} is not satisfied. Therefore, it is a general practice to stop the algorithm when
\begin{equation}\label{approxstop}
\sum_{i\in\O}\vartheta_i^{(k+1)}+\eta^{(k+1)} - \bar{F}^{(k+1)}\leq \varepsilon
\end{equation}
for some small constant $\varepsilon>0$.

\begin{remark}\label{singlecut}
Note that the objective function of $(MP_1(w))$ can be replaced with
\[
\vartheta+\eta-\sum_{i\in \O}\varrho \norm{\mu_i-\bar{\mu}^{(k)}_i}^2 -\sum_{i\in \O}\varrho \norm{\lambda_i-\bar{\lambda}^{(k)}_i}^2
\]
and constraint~\eqref{grad1} can be replaced with
\[
\vartheta \leq \sum_{i\in\O}\of{f_i(\mu_i^{(\ell)},\lambda_i^{(\ell)},w) + g_{\mu^{(\ell)}_i}^{\mathsf{T}}(\mu_i-\mu^{(\ell)}_i) + g_{\lambda^{(\ell)}_i}^{\mathsf{T}}(\lambda_i-\lambda_i^{(\ell)})}, \;\forall  \ell\in\mathcal{L}.
\]
This way one would obtain an upper approximation for the sum $\sum_{i\in\O}f_i(\cdot,\cdot,w)$. Compared to the multiple cuts in \eqref{grad1}, this provides a looser upper approximation for $\sum_{i\in\O}f_i(\cdot,\cdot,w)$. However, while one adds $I=\abs{\O}$ cuts at each iteration in the multiple cuts version, this approach adds a single cut.
\end{remark}

\begin{algorithm*}[h!]         
\caption{A Dual Bundle Method for $(P_1(w))$}    
\label{decomp1}                           
\begin{algorithmic}[1]
    \STATE $k \leftarrow 0$, $\mathcal{L}\leftarrow\emptyset$, $\gamma\in (0,1)$, $\vartheta_i^{(1)}\leftarrow\infty$ for each $i\in\O$, $\eta^{(1)}\leftarrow\infty$, $\bar{F}^{(1)}\leftarrow 0$;

    \STATE Let $\mu^{(1)}\in\M_1^J,\lambda^{(1)}\in\L^M$ be such that $\E\sqb{\lambda^{(1)}}=0$;

     \REPEAT
    \STATE $k \leftarrow k + 1$;

     \FOR{each $i\in\O$}
        \STATE  Compute an optimal solution $(x_i^{(k)},y_i^{(k)})$ and the optimal value $f_i(\mu_i^{(k)},\lambda_i^{(k)},w)$ of the subproblem
          \[
           \min_{(x_i,y_i)\in\F_i} \of{w^{\mathsf{T}}\sqb{\mu_i^{(k)} \cdot (Cx_i+Q_i y_i)}+p_i (\lambda_i^{(k)})^{\mathsf{T}}x_i};
          \]
          \STATE Compute subgradients $g_{\mu_i^{(k)}} = w\cdot(Cx_i^{(k)}+Q_iy_i^{(k)})$, $g_{\lambda_i^{(k)}} = p_ix_i^{(k)}$;

    \ENDFOR

    \STATE Compute $\beta(\mu^{(k)},w)$ and subgradient $(\rho_{\mu_1^{(k)}},\ldots,\rho_{\mu_I^{(k)}} ) \in \partial_{\mu}(-\beta)(\mu^{(k)},w)$ ;

     \STATE $F^{(k)}\leftarrow \sum_{i\in\O}  f_i(\mu_i^{(k)},\lambda_i^{(k)},w)-\b(\mu^{(k)},w)$;


     \IF{$F^{(k)} < \sum_{i\in\O}\vartheta_i^{(k)}+\eta^{(k)}$}
        \STATE  $\mathcal{L}\leftarrow \mathcal{L} \cup \{k\}$;
     \ENDIF

    \IF{($k = 1$) or ($k\geq 2$ and $F^{(k)} \geq (1-\gamma)\bar{F}^{(k)}+ \gamma (\sum_{i\in\O}\vartheta_i^{(k)}+\eta^{(k)})$)}
        \STATE $\bar{\mu}^{(k)} \leftarrow \mu^{(k)}$, $\bar{\lambda}^{(k)}\leftarrow \lambda^{(k)}$;
       \ELSE
         \STATE $\bar{\mu}^{(k)} \leftarrow \bar{\mu}^{(k-1)}$, $\bar{\lambda}^{(k)} \leftarrow \bar{\lambda}^{(k-1)}$;
     \ENDIF

    \STATE Solve the master problem. Let $(\mu^{(k+1)},\lambda^{(k+1)},\vartheta^{(k+1)},\eta^{(k+1)})$ be an optimal solution;


    \STATE  (Optional) Remove all cuts whose dual variables at the solution of master problem are zero;

 \STATE  $\bar{F}^{(k+1)}\leftarrow\sum_{i\in\O}  f_i(\bar{\mu}^{(k)}_i,\bar{\lambda}^{(k)}_i,w)-\b(\bar{\mu}^{(k)},w)$;

    \UNTIL{$\sum_{i\in\O}\vartheta_i^{(k+1)}+\eta^{(k+1)}=\bar{F}^{(k+1)}$};

    \RETURN $\left\{
              \begin{array}{ll}
                \bar{F}^{(k+1)}          & \hbox{: Optimal value $\P_1(w)$;} \\
                (\bar{\mu}^{(k)},\bar{\lambda}^{(k)}) & \hbox{: An optimal solution of $(D_1(w))$;} \\
              \end{array}
            \right.$
\end{algorithmic}
\end{algorithm*}

\subsubsection{Recovery of primal solution}\label{P1recovery}

Both the primal and the dual Benson algorithms require an optimal solution $(x_{(w)},y_{(w)},z_{(w)})$ of the problem $(P^\prime_1(w))$. Therefore, in Theorem~\ref{P1primal}, we suggest a procedure to recover an optimal primal solution from the solution of the master problem $(MP_1(w))$.

\begin{theorem}\label{P1primal}
Let $\mathcal{L}=\cb{1,\ldots,k}$ be the index set at the last iteration of the dual bundle method with the approximate stopping condition \eqref{approxstop} for some $\varepsilon>0$. Let $n+1$ be the first descent iteration after the approximate stopping condition is satisfied and let $\mathcal{L}^\prime=\cb{1,\ldots,n}$. For $(MP_1(w))$ with centers $\bar{\mu}^{(k)},\bar{\lambda}^{(k)}$ and index set $\mathcal{L}^{\prime}$, let $\tau=(\tau_i^{(\ell)})_{i\in\O,\ell\in\mathcal{L}^{\prime}},\theta=(\theta^{(\ell)})_{\ell\in\mathcal{L}^{\prime}},\sigma\in\R^M,\Psi\in\R^J,\nu=(\nu_i)_{i\in\O}$ be the Lagrangian dual variables assigned to the constraints \eqref{grad1}, \eqref{grad2}, \eqref{explambda}, \eqref{muprob}, \eqref{mupos}, respectively, with $\tau_i^{(\ell)}\geq 0,\theta^{(\ell)}\geq 0,\nu_i\in\R^J_+$ for each $i\in\O,\ell\in\mathcal{L}^{\prime}$. Let $(x_i^{(\ell)}, y_i^{(\ell)})$ be an optimal solution of the subproblem in line~6 of Algorithm~\ref{decomp1} for each $i\in\O$ and $\ell\in\mathcal{L}^{\prime}$. Let
\[
\of{\tau^{(n+1)}=(\tau_i^{(\ell,n+1)})_{i\in\O,\ell\in\mathcal{L}^\prime},\theta^{(n+1)}=(\theta^{(\ell,n+1)})_{\ell\in\mathcal{L}^\prime},\sigma^{(n+1)},\Psi^{(n+1)},\nu^{(n+1)}=(\nu_i^{(n+1)})_{i\in\O}}
\]
be a dual optimal solution for $(MP_1(w))$. Let $x_{(w)}=((x_{(w)})_i)_{i\in\O},y_{(w)}=((y_{(w)})_i)_{i\in\O}$ be defined by
\[
(x_{(w)})_i\coloneqq \sum_{\ell\in\mathcal{L}^\prime}\tau_i^{(\ell)}x_i^{(\ell)},\quad (y_{(w)})_i\coloneqq \sum_{\ell\in\mathcal{L}^\prime}\tau_i^{(\ell)}y_i^{(\ell)}.
\]
Moreover, let $z_{(w)}$ be a minimizer of the problem
\[
\inf_{z\in R(Cx_{(w)}+Qy_{(w)})}w^{\mathsf{T}}z.
\]
Then, $(x_{(w)},y_{(w)},z_{(w)})$ is an approximately optimal solution of $(P^\prime_1(w))$ in the following sense:
\begin{enumerate}[(a)]
\item $((x_{(w)})_i,(y_{(w)})_i)\in\F_i$ for each $i\in\O$.
\item $z_{(w)}\in R(Cx_{(w)}+Qy_{(w)})$.
\item As $\varepsilon\rightarrow 0$, it holds $(x_{(w)})_i-\sigma^{(n+1)}\rightarrow 0$ for each $i\in\O$.
\item As $\varepsilon\rightarrow 0$, it holds $w^{\mathsf{T}}z_{(w)}\rightarrow \mathscr{P}_1(w)$.
\end{enumerate}
\end{theorem}

The proof of Theorem~\ref{P1primal} is given in Appendix~\ref{proofof2}.

\subsection{The problem of scalarization by a reference variable}

Let $v\in\R^J\sm\P$. The problem $(P_2(v))$ defined in Section~\ref{p2section} is formulated to find the minimum step-length to enter $\P$ from $v$ along the direction $\1\in\R^J$ and it can be rewritten more explicitly as
\begin{align*}
&\text{min }\; \;                    \a\tag{$P_2(v)$}\\
&\text{s.t.}\; \quad                      v+\a\1\in R(Cx + Qy)\\
& \;\;\quad\quad      Ax = b\\
&\;\;\quad\quad     T_i x + W_i y_i = h_i \quad  \forall i\in\O\\
& \;\;\quad\quad     \a\in\R,\; x\in \R^M_+,\; y_i\in\R^N_+\quad  \forall i\in\O.
\end{align*}

We propose a scenario-wise decomposition solution methodology for $(P_2(v))$. Even the steps we follow are similar to the ones for $(P_1(w))$, the decomposition is more complicated because the weights are not parameters but instead they are decision variables in the dual problem of $(P_2(v))$ (\emph{see} Theorem~\ref{p2thm} below). Therefore, following the same steps as in $(P_1(w))$ results in a nonconvex optimization problem. In order to resolve this convexity issue, we propose a new formulation for $(P_2(v))$ by introducing finite measures to the dual representation of $R$.

The flow of this section is as follows: in Sections~\ref{p2decomposition} and~\ref{p2bundle}, we propose a scenario-wise decomposition solution methodology for $(P_2(v))$. Section~\ref{P2recovery} is devoted to the recovery of a primal solution.

\subsubsection{Scenario decomposition}\label{p2decomposition}

To derive a decomposition algorithm for $(P_2(v))$, we randomize the first stage variable $x\in\R^M$ as in $(P_1(w))$ and add the nonanticipativity constraints
\[
 p_i(x_i -\E\sqb{x})  = 0, \quad \forall i\in\O.
\]
Using the feasible region $\F$ defined by \eqref{f-defn}, we may rewrite $(P_2(v))$ as follows:
\begin{align*}
&\text{min }\; \;                     \a\tag{$P^\prime_2(v)$}\\
&\text{s.t.}\quad \;                v+\a\1\in R(Cx+Qy)\\
& \;\;\quad\quad    p_i\of{x_i - \E\sqb{x}} = 0\quad \forall i\in\O\\
& \;\;\quad\quad     (x,y)\in \F\\
& \;\;\quad\quad     \a\in \R
\end{align*}
Note that the optimal value of $(P_2^\prime(v))$ is $\mathscr{P}_2(v)$.

Different from the approach for $(P^\prime_1(w))$, in order to obtain a convex dual problem for $(P^\prime_2(v))$, we use finite measures $m$ instead of probability measures $\mu$ in the dual representation of $R$. To that end, let $\M_f^J$ be the set of all $J$-dimensional vectors $m=(m^1,\ldots,m^J)^\mathsf{T}$ of finite measures on $\O$, that is, for each $j\in\J$, the finite measure $m^j$ assigns $m^j_i$ to the elementary event $\cb{i}$ for $i\in\O$. For $m\in\M_f^J$ and $i\in\O$, we also write $m_i\coloneqq(m_i^1,\ldots,m_i^J)^\mathsf{T}\in\R^J$.

The following lemma provides the relationship between $\mu$ and $m$.

\begin{lemma}\label{conversionlemma}
For every $\mu\in \M_1^J$ and $\gamma\in\R^J_+\sm\cb{0}$, there exists $m\in\M^J_f$ such that
\begin{equation}\label{conversion}
\gamma^\mathsf{T}\E^\mu \sqb{u} = \sum_{i\in\O}m_i^\mathsf{T}u_i, \quad \gamma^\mathsf{T}\1 = \sum_{i\in\O}m_i^\mathsf{T}\1 ,
\end{equation}
for every $u\in \L^J$. Conversely, for every $m\in\M^J_f$, there exist $\mu\in \M_1^J$ and $\gamma\in\R^J_+\sm\cb{0}$ such that \eqref{conversion} holds for every $u\in\L^J$.
\end{lemma}

\begin{proof}
Let $\mu\in \M_1^J$ and $\gamma\in\R^J_+\sm\cb{0}$. Define $m\in\M^J_f$ by
\[
m^j_i = \gamma^j \mu^j_i,\quad\forall i\in\O,j\in\J.
\]
Then, trivially, \eqref{conversion} holds for every $u\in\L^J$. Conversely, let $m\in\M^J_f$. Define $\mu\in \M_1^J$ and $\gamma\in\R^J_+\sm\cb{0}$ by
\[
\gamma^j = \sum_{i\in\O}m^j_i, \quad \mu^j_i = \begin{cases}\frac{m^j_i}{\gamma^j}&\text{ if }\gamma^j> 0,\\ \frac{1}{I}& \text{ if }\gamma^j=0,\end{cases}\quad\forall i\in\O,j\in\J.
\]
Then, trivially, \eqref{conversion} holds for every $u\in\L^J$.
\end{proof}

Recall that $\beta$ is the minimal penalty function of $R$ as defined in \eqref{beta}. For $m\in\M^J_f$, let us define
\begin{equation}\label{betatilde}
\tilde{\beta}(m)=\sup_{u\in\A}\sum_{i\in\O}m_i^\mathsf{T}u_i.
\end{equation}
Similarly, recall the function $f_i$ defined in \eqref{scenariosubp}. For $m\in\M^J_f$ and $\lambda\in\L^M$, let us define
\begin{equation}\label{tildef}
\tilde{f}_i (m_i,\lambda_i)=\inf_{(x_i,y_i)\in\F_i}\of{m_i^\mathsf{T} (Cx_i+Q_iy_i)+p_i\lambda_i^\mathsf{T}x_i}.
\end{equation}
Therefore, if $\mu\in\M_1^J, \gamma\in\R^J_+\sm\cb{0}$ and $m\in\M^J_f$ are related as in Lemma~\ref{conversionlemma} and $\lambda\in\L^M$, then it is clear that
\[
\beta(\mu,\gamma)=\tilde{\beta}(m), \quad f_i(\mu_i,\lambda_i,\gamma)=\tilde{f}_i
(m_i,\lambda_i).
\]

\begin{example} Recall Example~\ref{multient} on the multivariate entropic risk measure. The function $\tilde{\beta}(\cdot)$ takes the form
\begin{align*}
\tilde{\b}(m)
&= \sum_{j\in\J}\frac{1}{\delta^{j}}\of{H(m^j||p)-{m^j}^\mathsf{T}\1}+\inf_{s\in C^+}\sum_{j\in\J}\frac{1}{\delta^{j}}\of{s^{j}-({m^j}^\mathsf{T}\1)\log s^{j}},
\end{align*}
where $\1=(1,\ldots,1)^\mathsf{T}\in\R^I$ and $H(m^j||p)$ is the relative entropy of $m^j$ with respect to $p$ defined by
\[
H(m^j||p)=\sum_{i\in\O}m^j_i\log\of{\frac{m^j_i}{p_i}}.
\]
\end{example}

\begin{example} Recall Example~\ref{multiCVaR} on the multivariate CVaR. The function $\tilde{\beta}(\cdot)$ takes the form
\[
\tilde{\b}(m)=\begin{cases}0& \text{ if }\frac{m^j_i}{p_i}\leq \frac{1}{1-\nu^j}\sum_{i\in\O}m^j_i,\quad \forall \; i\in\O,j\in\J,\\ +\infty &\text{ else,}\end{cases}
\]
for every $m\in\M^J_f$.
\end{example}
\begin{remark}\label{betaf}
Note that, in general, $\beta(\cdot,\cdot)$ is not a convex function since $(\mu,\gamma)\mapsto \gamma^\mathsf{T}\E^\mu\sqb{u}$ is not a convex function. On the other hand, $\tilde{\beta}(\cdot)$ is a convex function. Indeed, for each $i\in\O$ and $u\in\A$, $m_i\mapsto m_i^\mathsf{T}u_i$ is a linear function so that $m\mapsto\tilde{\beta}(m)$ is a convex function since it is the supremum of linear functions indexed by $u\in\A$. Similarly, $f_i(\cdot,\cdot,\cdot)$ is not a concave function in general. However, $(m_i,\lambda_i)\mapsto\tilde{f}(m_i,\lambda_i)$ is the infimum of linear functions indexed by $(x_i,y_i)\in\F_i$; therefore, it is a concave function.
\end{remark}

\begin{theorem}\label{p2thm}
It holds
\begin{align}
\P_2(v)&=\sup_{\mu\in\M_1^J,\lambda\in\L^M,\gamma\in\R^J_+}\Bigg\{\sum_{i\in\O} f_i(\mu_i,\lambda_i,\gamma)-\gamma^\mathsf{T}v-\b(\mu,\gamma)\mid \gamma^\mathsf{T}\1=1,\; \E\sqb{\lambda}=0\Bigg\}\notag\\
&=\sup_{m\in\M_f^J,\lambda\in\L^M}\Bigg\{\sum_{i\in\O} \tilde{f}_i(m_i,\lambda_i)-\sum_{i\in\O} m_i^\mathsf{T}v-\tilde{\b}(m)\mid\sum_{i\in\O} m_i^\mathsf{T}\1=1,\; \E\sqb{\lambda}=0\Bigg\}.\tag{$D_2(v)$}
\end{align}
\end{theorem}

In view of Remark~\ref{betaf}, while the first reformulation of $(P^\prime_2(v))$ provided in Theorem~\ref{p2thm} is not a convex optimization problem, the second reformulation, that is $(D_2(v))$, is a convex optimization problem.

The proof of Theorem~\ref{p2thm} uses Lemma~\ref{conversionlemma} and the following lemma of independent interest.

\begin{lemma}\label{scalarizationlemma}
For every $u\in\L^J$,
\begin{align*}
\inf\cb{ \a\in\R\mid v+\a\1\in R(u)}=\sup\cb{\gamma^\mathsf{T}\of{\E^\mu\sqb{u}-v}-\b(\mu,\gamma)\mid \mu\in\M_1^J,\; \gamma^\mathsf{T}\1=1, \gamma\in\R^J_+ }.
\end{align*}
\end{lemma}

\begin{proof}
Let $u\in\L^J$. Note that
\[
\inf\cb{ \a\in\R\mid v+\a\1\in R(u)}=\inf\cb{ \a\in\R\mid 0\in R(u)-v-\a\1}
\]
is the optimal value of a single-objective optimization problem with a set-valued constraint function $\a\mapsto H(\a)=R(u)-v-\a\1$. Using the Lagrange duality in \cite{borwein} for such problems, in particular, Theorem~19, we have
\begin{equation}\label{borweinlagrange}
\inf\cb{ \a\in\R\mid v+\a\1\in R(u)}=\sup_{\gamma\in\R^J}\inf_{\a\in\R}\of{\a+\inf_{z\in R(u)-v-\a\1}\gamma^\mathsf{T}z}.\\
\end{equation}
To be able to use this result, we check the following constraint qualification: $H$ is \emph{open} at $0\in\R^J$ in the sense that for every $\a\in\R$ with $0\in H(\a)$ and for every $\varepsilon>0$, there exists an open ball $V$ around $0\in\R^J$ such that
\begin{equation}\label{interior}
V\subseteq \bigcup_{\tilde{\a}\in(\a-\varepsilon,\a+\varepsilon)} H(\tilde{\a}).
\end{equation}
To that end, let $\a\in\R$ with $0\in H(\a)$, that is, $v+\a\1\in R(u)$. Let $\varepsilon>0$. Since $\1$ is an interior point of $\R^J_+$ and $R(u)+\R^J_+=R(u)$ due to the monotonicity and translativity of $R$, it follows that $v+(\a+\varepsilon)\1$ is an interior point of $R(u)$. On the other hand, note that
\begin{align*}
\bigcup_{\tilde{\a}\in(\a-\varepsilon,\a+\varepsilon)} H(\tilde{\a})&=\bigcup_{\tilde{\a}\in(\a-\varepsilon,\a+\varepsilon)} R(u-v-\tilde{\a}\1)\\
&=R(u-v-(\a+\varepsilon)\1)\\
&=R(u)-v-(\a+\varepsilon)\1
\end{align*}
thanks to the monotonicity and translativity of $R$. Hence, $0\in\R^J$ is an interior point of the above union. Therefore, \eqref{interior} holds for some open ball $V$ around $0\in\R^J$ and \eqref{borweinlagrange} follows.

Since $R(u)+\R^J_+=R(u)$ and $R(u)$ is a convex set as a consequence of the convexity of $R$, one can check that $\inf_{z\in R(u)}\gamma^\mathsf{T}z = -\infty$ for every $\gamma\notin\R^J_+$. Hence, the supremum in \eqref{borweinlagrange} can be evaluated over all $\gamma\in\R^J_+$. Finally, using \eqref{scalar-dual}, we obtain
\begin{align*}
&\inf\cb{ \a\in\R\mid v+\a\1\in R(u)}\\
&=\sup_{\gamma\in\R^J_+}\inf_{\a\in\R}\of{\a+\inf_{z\in R(u)-v-\a\1}\gamma^\mathsf{T}z}\\
&=\sup_{\gamma\in\R^J_+}\inf_{\a\in\R}\of{\a-\gamma^\mathsf{T}(v+\a\1)+\sup_{\mu\in\M_1^J}\of{\gamma^\mathsf{T}\E^\mu\sqb{u}-\b(\mu,\gamma)}}\\
&=\sup_{\gamma\in\R^J_+}\sqb{\inf_{\a\in\R}(1-\gamma^\mathsf{T}\1)\a+\sup_{\mu\in\M_1^J}\of{\gamma^\mathsf{T}(\E^\mu\sqb{u}-v)-\b(\mu,\gamma)}}\\
&= \sup\cb{\gamma^\mathsf{T}\of{\E^\mu\sqb{u}-v}-\b(\mu,\gamma)\mid \mu\in\M_1^J,\; \gamma^\mathsf{T}\1=1, \gamma\in\R^J_+ },
\end{align*}
where, in the last equality, we use the observation that $\inf_{\a\in\R}(1-\gamma^\mathsf{T}\1)\a = 0$ if $\gamma^\mathsf{T}\1=1$ and  $\inf_{\a\in\R}(1-\gamma^\mathsf{T}\1)\a = -\infty$ if $\gamma^\mathsf{T}\1\neq 1$.
\end{proof}

\begin{proof}[Proof of Theorem~\ref{p2thm}]
Using Lemma~\ref{scalarizationlemma}, we may write
\begin{align*}
\P_2(v) &=\inf_{(x,y)\in\F,\a\in\R}\cb{\a\mid v+\a\1\in R(Cx+Qy),\; p_i(x_i-\E\sqb{x})=0\;\forall i\in\O}\\
&=\inf_{(x,y)\in\F}\cb{\inf\cb{\a\in\R\mid v+\a\1\in R(Cx+Qy)}\mid  p_i(x_i-\E\sqb{x})=0\;\forall i\in\O}\\
&=\inf_{(x,y)\in\F}\cb{\sup_{\mu\in\M_1^J,\gamma\in\R^J_+,\gamma^\mathsf{T}\1=1}\of{\gamma^\mathsf{T}(\E^\mu\sqb{Cx+Qy}-v)-\b(\mu,\gamma)}\mid  p_i(x_i-\E\sqb{x})=0\;\forall i\in\O}.
\end{align*}
Using the minimax theorem of \cite{sion}, we may interchange the infimum and the supremum, and obtain
\begin{equation*}
\P_2(v)=\sup_{\mu\in\M_1^J,\gamma\in\R^J_+,\gamma^\mathsf{T}\1=1}\of{F(\mu,\gamma)-\gamma^\mathsf{T}v-\b(\mu,\gamma)},
\end{equation*}
where, for each $\mu\in\M_1^J$, $\gamma\in\R^J$, $F(\mu,\gamma)$ is defined by \eqref{Ffunction}. Hence, using \eqref{Flagrange}, we obtain
\[
\P_2(v)=\sup_{\mu\in\M_1^J,\lambda\in\L^M,\gamma\in\R^J_+}\cb{\sum_{i\in\O} f_i(\mu_i,\lambda_i,\gamma)-\gamma^\mathsf{T}v-\b(\mu,\gamma)\mid \gamma^\mathsf{T}\1=1,\; \E\sqb{\lambda}=0},
\]
where $f_i$ is defined by \eqref{scenariosubp}. Hence, the first reformulation follows. The second reformulation follows from the first reformulation and Lemma~\ref{conversionlemma}.
\end{proof}

\subsubsection{The dual bundle method}\label{p2bundle}

To solve $(D_2(v))$ provided in Theorem~\ref{p2thm}, we propose a dual bundle method similar to the one in Section~\ref{P1cuttingplane}.

At each iteration $k$ of the dual bundle method, we solve the master problem $(MP_2(v))$ given below. Here, $(g_{m_i},g_{\lambda_i})\in\R^{J+M}$ denotes a subgradient of the concave function $\tilde{f}_i(\cdot,\cdot)$ at the point $(m_i,\lambda_i)\in\R^J_+\times\R^M$. Similarly, $(\rho_{m_1},\ldots,\rho_{m_I})\in\R^{J\times I}$ denotes a subgradient of the concave function $-\tilde{\beta}(\cdot)$ at the point $m\in\M_f^J$. We call \eqref{p2con1} and \eqref{p2con2} a cut for $\tilde{f}_i(\cdot,\cdot)$ and $-\tilde{\beta}(\cdot)$, respectively.

\begin{align}
&\text{ max }\;\;         \sum_{i\in\O}\vartheta_i\negthinspace +\negthinspace \eta -\sum_{i\in\O} m_i^{\mathsf{T}}v\negthinspace -\sum_{i\in \O}\varrho \norm{m_i-\bar{m}^{(k)}_i}^2 \negthinspace -\sum_{i\in \O}\varrho \norm{\lambda_i\negthinspace -\bar{\lambda}^{(k)}_i}^2\tag{$MP_2(v)$}\\
&\text{ s.t. }\quad\;                      \vartheta_i \leq \tilde{f}_i(m_i^{(\ell)},\lambda_i^{(\ell)}) \negthinspace +\negthinspace g_{m^{(\ell)}_i}^{\mathsf{T}}(m_i\negthinspace -\negthinspace m^{(\ell)}_i) \negthinspace +\negthinspace  g_{\lambda^{(\ell)}_i}^{\mathsf{T}}(\lambda_i\negthinspace -\negthinspace \lambda_i^{(\ell)}), \quad \forall i\in\O,  \ell\in\mathcal{L}\label{p2con1}\\
& \;\quad\quad\quad  \eta \leq -\tilde{\beta}(m^{(\ell)}) + \sum_{i\in\O}\rho_{m^{(\ell)}_i}^{\mathsf{T}}(m_i-m_i^{(\ell)}), \quad \forall \ell\in\mathcal{L}\label{p2con2}\\
& \;\quad\quad\quad      \sum_{i\in\O}p_i\lambda_i=0\label{p2con3}\\
& \;\quad\quad\quad      \sum_{i\in\O} m_i^\mathsf{T}\1=1\label{p2con4}\\
& \;\quad\quad\quad  m_i \in \R_+^J, \lambda_i\in\R^M,\vartheta_i\in\R,\quad\forall i\in\O\label{p2con5}\\
& \;\quad\quad\quad \eta \in \R.
\end{align}
Here, $\mathcal{L}=\cb{1,\ldots,k}$, $\varrho>0$.

The steps of the dual bundle method are provided in Algorithm~\ref{decomp2}. Similar to \eqref{approxstop}, the algorithm stops in practice when
\begin{equation}\label{approxstop2}
\sum_{i\in\O}\vartheta_i^{(k+1)}+\eta^{(k+1)}-\sum_{i\in\O}(m_i^{(k+1)})^{\mathsf{T}}v-\bar{F}^{(k+1)}\leq \varepsilon
\end{equation}
for some $\varepsilon>0$.
Since the construction of this algorithm is similar to the construction of Algorithm~\ref{decomp1}, the details are omitted for brevity.

\begin{algorithm*}[h!]
\caption{A Dual Bundle Method for $(P_2(v))$}
\label{decomp2}
\begin{algorithmic}[1]
    \STATE $k \leftarrow 0$, $\mathcal{L}\leftarrow\emptyset$, $\gamma\in (0,1)$, $\vartheta_i^{(1)}\leftarrow\infty$ for each $i\in\O$, $\eta^{(1)}\leftarrow\infty$, $\bar{F}^{(1)}\leftarrow 0$;

   \STATE Let $m^{(1)}\in\M_f^J,\lambda^{(1)}\in\L^M$ be such that $\E\sqb{\lambda^{(1)}}=0$;

    \REPEAT

\STATE $k\leftarrow k+1$;

   \FOR{each $i\in\O$}
        \STATE  Compute an optimal solution $(x_i^{(k)},y_i^{(k)})$ and the optimal value $\tilde{f}_i(m_i^{(k)},\lambda_i^{(k)})$ of the subproblem
           \[
          \min_{(x_i,y_i)\in\F_i} \of{(m_i^{(k)})^{\mathsf{T}}(Cx_i+Q_i y_i)+p_i (\lambda_i^{(k)})^{\mathsf{T}}x_i};
          \]
          \STATE Compute subgradients $g_{m_i^{(k)}} = Cx_i^{(k)}+Q_iy_i^{(k)}$, $g_{\lambda_i^{(k)}} = p_ix_i^{(k+1)}$;

    \ENDFOR

 \STATE Compute $\tilde{\beta}(m^{(k)})$ and subgradient $(\rho_{m_1^{(k)}},\ldots,\rho_{m_I^{(k)}} ) \in \partial_{m}(-\tilde{\beta})(m^{(k)})$ ;

     \STATE $F^{(k)}\leftarrow \sum_{i\in\O}  \tilde{f}_i(m_i^{(k)},\lambda_i^{(k)})-\tilde{\beta}(m^{(k)})-\sum_{i\in\O}(m^{(k)}_i)^{\mathsf{T}}v$;

     \IF{$F^{(k)} < \sum_{i\in\O}\vartheta_i^{(k)}+\eta^{(k)}-\sum_{i\in\O}(m^{(k)}_i)^{\mathsf{T}}v$}
        \STATE  $\mathcal{L}\leftarrow \mathcal{L} \cup \{k\}$;
     \ENDIF

    \IF{($k = 1$) or ($k\geq 2$ and $F^{(k)} \geq (1-\gamma)\bar{F}^{(k)}+ \gamma (\sum_{i\in\O}\vartheta_i^{(k)}+\eta^{(k)}-\sum_{i\in\O}(m^{(k)}_i)^{\mathsf{T}}v)$)}
        \STATE $\bar{m}^{(k)} \leftarrow m^{(k)}$, $\bar{\lambda}^{(k)}\leftarrow \lambda^{(k)}$;
       \ELSE
         \STATE $\bar{m}^{(k)} \leftarrow \bar{m}^{(k-1)}$, $\bar{\lambda}^{(k)} \leftarrow \bar{\lambda}^{(k-1)}$;
     \ENDIF

 \STATE Solve the master problem. Let $(m^{(k+1)},\lambda^{(k+1)},\vartheta^{(k+1)},\eta^{(k+1)})$ be an optimal solution;

    \STATE  (Optional) Remove all cuts whose dual variables at the solution of master problem are zero;

 \STATE  $\bar{F}^{(k+1)}\leftarrow\sum_{i\in\O}  \tilde{f}_i(\bar{m}^{(k)}_i,\bar{\lambda}^{(k)}_i)-\tilde{\b}(\bar{m}^{(k)})-\sum_{i\in\O}(\bar{m}^{(k)}_i)^{\mathsf{T}}v$;

    \UNTIL{$\sum_{i\in\O}\vartheta_i^{(k+1)}+\eta^{(k+1)}-\sum_{i\in\O}(m^{(k+1)}_i)^{\mathsf{T}}v= \bar{F}^{(k+1)}$};

    \RETURN $\left\{
              \begin{array}{ll}
                \bar{F}^{(k+1)}          & \hbox{: Optimal value $\P_2(v)$;} \\
                (\bar{m}^{(k)},\bar{\lambda}^{(k)}) & \hbox{: An optimal solution of $(D_2(v))$;} \\
              \end{array}
            \right.$
\end{algorithmic}
\end{algorithm*}

Next, we provide a recipe for computing the subgradients $g_{m_i},g_{\lambda_i},\rho_m$. Let us denote by $\partial_{m_i,\lambda_i}\tilde{f}_i (m^\prime_i, \lambda^\prime_i)$ the subdifferential  of the function $\tilde{f}_i(\cdot,\cdot)$ at a point $(m_i^\prime,\lambda^\prime_i)\in\R_+^J\times\R^M$, and by $\partial_{m}(-\tilde{\beta})(m^\prime)$ the subdifferential of the function $-\tilde{\beta}(\cdot)$ at a point $m^\prime\in\L^J_+$. In the next proposition, we show how to compute $\partial_{m_i,\lambda_i} \tilde{f}_i(m^\prime_i,\lambda^\prime_i)$ and a subgradient $\rho_m$ of the function $-\tilde{\beta}(\cdot)$.

\begin{proposition}
\begin{enumerate}[(a)]
\item For $i\in\O,m^{\prime}_i\in\R^J_+,\lambda_i^\prime\in\R^M$, let
\begin{equation*}
\tilde{A}_i(m^\prime_i,\lambda^\prime_i) \coloneqq \argmin_{(x_i,y_i)\in\F_i}\of{(m^\prime_i)^{\mathsf{T}} (Cx_i+Q_i y_i)+p_i (\lambda^\prime_i)^{\mathsf{T}}x_i}.
\end{equation*}
Then,
\begin{equation*}
\partial_{m_i,\lambda_i}\tilde{f}_i(m^\prime_i,\lambda^\prime_i)=\co \cb{\of{Cx_i + Q_i y_i,p_i x_i}\mid (x_i,y_i)\in \tilde{A}_i(m^\prime_i,\lambda^\prime_i)}.
\end{equation*}
\item Recall that the set $\A=\cb{u\in\L^J\mid 0\in R(u)}$ is the acceptance set of $R$. For $m^\prime \in\L^J_+$, let
\begin{equation*}
\tilde{B}(m^\prime) \coloneqq \argmax_{u\in\A}\sum_{i\in\O}(m_i^\prime )^{\mathsf{T}}u_i
\end{equation*}
and assume that $\tilde{B}(m^\prime)\neq \emptyset$.
Then,
\begin{equation*}
\partial_{m}(-\tilde{\beta})(m^\prime)\supseteq \co \cb{u=( u_1,\ldots, u_I) \mid u\in \tilde{B}(m^\prime)}.
\end{equation*}
\end{enumerate}
\end{proposition}

\begin{proof}
The proof of this proposition is similar to the proofs of Propositions~\ref{subgradientf} and~\ref{subgradientbeta}. Therefore, it is omitted.
\end{proof}

\subsubsection{Recovery of primal solution}\label{P2recovery}

The primal Benson algorithm requires an optimal solution $(x_{(v)},y_{(v)},\a_{(v)})$ of the problem $(P^\prime_2(v))$. Therefore, in Theorem~\ref{P2primal}, we suggest a procedure to recover an optimal primal solution from the solution of the master problem $(MP_2(v))$.

\begin{theorem}\label{P2primal}
Let $\mathcal{L}=\cb{1,\ldots,k}$ be the index set at the last iteration of the dual bundle method with the approximate stopping condition \eqref{approxstop2} for some $\varepsilon>0$. Let $n+1$ be the first descent iteration after the appriximate stopping condition is satisfied and let $\mathcal{L}^\prime=\cb{1,\ldots,n}$. For $(MP_2(v))$ with centers $\bar{m}^{(k)},\bar{\lambda}^{(k)}$ and index set $\mathcal{L}^{\prime}$, let $\tau=(\tau_i^{(\ell)})_{i\in\O,\ell\in\mathcal{L}^{\prime}},\theta=(\theta^{(\ell)})_{\ell\in\mathcal{L}^{\prime}},\sigma\in\R^M,\psi\in\R,\nu=(\nu_i)_{i\in\O}$ be the Lagrangian dual variables assigned to the constraints \eqref{p2con1}, \eqref{p2con2}, \eqref{p2con3}, \eqref{p2con4}, \eqref{p2con5}, respectively, with $\tau_i^{(\ell)}\geq 0,\theta^{(\ell)}\geq 0,\nu_i\in\R^J_+$ for each $i\in\O,\ell\in\mathcal{L}^{\prime}$. Let $(x_i^{(\ell)}, y_i^{(\ell)})$ be an optimal solution of the subproblem in line~6 of Algorithm~\ref{decomp2} for each $i\in\O$ and $\ell\in\mathcal{L}^{\prime}$. Let
\[
\of{\tau^{(n+1)}=(\tau_i^{(\ell,n+1)})_{i\in\O,\ell\in\mathcal{L}^\prime},\theta^{(n+1)}=(\theta^{(\ell,n+1)})_{\ell\in\mathcal{L}^\prime},\sigma^{(n+1)},\psi^{(n+1)},\nu^{(n+1)}=(\nu_i^{(n+1)})_{i\in\O}}
\]
be a dual optimal solution for $(MP_2(v))$. Let $x_{(v)}=((x_{(v)})_i)_{i\in\O},y_{(v)}=((y_{(v)})_i)_{i\in\O}$ be defined by
\[
(x_{(v)})_i\coloneqq \sum_{\ell\in\mathcal{L}^\prime}\tau_i^{(\ell)}x_i^{(\ell)},\quad (y_{(v)})_i\coloneqq \sum_{\ell\in\mathcal{L}^\prime}\tau_i^{(\ell)}y_i^{(\ell)}.
\]
Let
\[
\a_{(v)}\coloneqq \inf\cb{\alpha\in\R\mid v+\alpha\1\in R(C\bar{x}+Q\bar{y})}.
\]
Then, $(x_{(v)},y_{(v)},\a_{(v)})$ is an approximately optimal solution of $(P^\prime_2(v))$ in the following sense:
\begin{enumerate}[(a)]
\item $((x_{(v)})_i,(y_{(v)})_i)\in\F_i$ for each $i\in\O$.
\item $v+\a_{(v)}\1\in R(Cx_{(v)}+Qy_{(v)})$.
\item As $\varepsilon\rightarrow 0$, it holds $(x_{(v)})_i-\sigma^{(n+1)}\rightarrow 0$ for each $i\in\O$.
\item As $\varepsilon\rightarrow 0$, it holds $\a_{(v)}\rightarrow \mathscr{P}_2(v)$.
\end{enumerate}
\end{theorem}

The proof of Theorem~\ref{P2primal} is given in Appendix~\ref{proofof4}.

\subsubsection{Recovery of a solution to \eqref{d2vdef}}

In addition to a primal optimal solution $(x_{(v)},y_{(v)},\a_{(v)})$, the primal Benson algorithm also requires an optimal solution $\gamma_{(v)}$ of the dual problem $(LD_2(v))$ (\emph{see} Section~\ref{p2section}). Therefore, in Theorem~\ref{ld2sol}, we suggest a procedure to recover this solution from the solution of the master problem $(MP_2(v))$.

\begin{theorem}\label{ld2sol}
In the setting of Theorem~\ref{P2primal}, let
\begin{equation}
\gamma_{(v)} = \sum_{i\in\O}m^{(n+1)}_i.
\end{equation}
Then, $\gamma_{(v)}$ is an approximately optimal solution of \eqref{d2vdef} in the following sense: as $\varepsilon\rightarrow 0$, it holds
\begin{equation}\label{triangle2}
\abs{\inf_{(x,y)\in\X,\a\in\R}\of{\a+\inf_{z\in R(Cx+Qy)-v-\alpha\1}\gamma_{(v)}^{\mathsf{T}}z}-\mathscr{P}_2(v)}\rightarrow 0.
\end{equation}
\end{theorem}

The proof of Theorem~\ref{ld2sol} is given in Appendix~\ref{proofof5}.

\section{Computational Study}\label{computation}

In order to test our methods, we solve a multi-objective risk-averse portfolio optimization problem under transaction costs. We consider a one-period market with $J$ risky assets. Each asset $j\in\J=\cb{1,\ldots,J}$ has a random return $r^{j}\in\L$. At the beginning of the period, it costs $\theta^{jk}\in\R$ units of asset $j$ for an agent to buy one unit of asset $k\in\J$. At the end of the period, the random transaction cost of buying one unit of asset $k$ is $\pi^{jk}\in\L$ units of asset $j$. 

The risk-averse agent has a capital $c\in\R_{++}$ units of asset $1$ to be invested in the $J$ assets. Let $x^j\in\R_+$ denote the number of physical units of asset $j$ purchased by the agent; hence, she spends $x^j\theta^{1j}$ units of asset $1$ for this purchase. At the end of the period, the agent observes the random return of each asset as well as the random transaction costs between the assets. The value of each asset $j$ is $(1+r^j)x^j$ and it is transacted to purchase the $J$ assets with a transaction cost of $\pi^{jk}$ for asset $k$. Let $q^{jk}\in\L_+$ denote the number of physical units of asset $k$ purchased by selling some units of asset $j$. Let $y^k\in\L_+$ denote the total number of physical units of asset $k$ purchased by the agent so that $y^k=\sum_{j\in\J}q^{jk}$. The objective is to minimize the risk of the random cost vector $-y\in\L^J$ using a multivariate convex risk measure $R$. This problem can be formulated as follows:

\begin{align*}
&\text{min}\; \;                    z\;\;\text{w.r.t.}\;\;\R_+^J\\
&\text{s.t. }\;\;                      z\in R(-y)\\
& \quad\quad\; \sum_{j \in \J} \theta^{1j} x^j = c\\
& \quad\quad\;  (1+ r^j_i)x^j = \sum_{k \in \J}\pi^{jk}_i q^{jk}_i,\quad \forall j\in \J, i\in\O\\
& \quad\quad\;       y^j_i=\sum_{k \in \J}q^{kj}_i,\quad \forall j\in \J, i\in\O\\
& \quad\quad\;       z\in\R^J,\;x\in \R^J_+,\; y_i\in\R^J_+,\; q_i\in\R^{J \times J}_+, \quad \forall i\in\O.
\end{align*}
Note that $x \in \R_+^J$ is the first-stage, $y_i\in\R^J_+,\; q_i\in\R^{J \times J}_+, i\in\O$ are the second-stage decision variables.

All computational experiments are conducted on a PC with 8.00 GB of RAM and an Intel(R) Core(TM) i7-4790 CPU@3.60 GHz processor. We use Matlab implementations of Algorithms \ref{decomp1} and \ref{decomp2} where CVX 1.22 is used to solve master problems and CPLEX 12.6 is used to solve subproblems.

We generate two classes of instances where the number of assets $J$ is either 2 or 3. In both cases, we assume $c=1$. We set  $\theta^{12} = 1.0815$, $\theta^{13} = 0.9094$. The return of asset 1 is uniformly distributed between $-0.1$ and $0.2$, denoted by $r^1 \sim U[-0.1,0.2]$. Similarly, we assume  $r^2 \sim U[-0.05,0.1]$ and $r^3 \sim U[-0.15,0.3]$. The random transaction costs among the assets are assumed to have the following distributions: $\pi^{12} \sim U[1,1.1]$, $\pi^{21} \sim U[0.9,1]$, $\pi^{13} \sim U[0.9,1]$, $\pi^{31} \sim U[1,1.1]$, $\pi^{23} \sim U[0.8,1]$, $\pi^{32}\sim U[1,1.2]$, and $\pi^{11} = \pi^{22}=\pi^{33} = 1$.

First of all, in Example~\ref{comp}, we compare our dual bundle method with CVX on the problem $(P_1(w))$ of weighted sum scalarization. Our dual bundle method takes the advantage of scenario-wise decompositions while CVX solves the problem as a standard convex optimization problem without decompositions.

\begin{example}\label{comp}
We compare the CPU times (in seconds) of the dual bundle method and the CVX solver on $(P_1(w))$ instances with two- and three-dimensional multivariate entropic risk measure and different numbers of scenarios ($I$). In each instance, we use a fixed weight vector $w$.

\begin{table}[h]
	\centering
	\begin{tabular}{c|cc}
		\hline \hline 
		$I$     & Dual Bundle Method  & CVX \\ \hline \hline 
		1000  & 869.22 & 75.98 \\
		2500  & 2130.85 & 588.85 \\
		5000  & 4170.55 & 3091.44 \\
		10000 & 8452.47 & ** \\
	\end{tabular}%
\caption{Computational performances of the dual bundle method and CVX for a two-dimensional multivariate entropic risk measure instance with weight vector $w=(1/2, 1/2)$}
	\label{decomcompare1}%
\end{table}%

\begin{table}[h]
	\centering
\begin{tabular}{c|cc}
		\hline \hline 
		$I$     & Dual Bundle Method  & CVX \\ \hline \hline 
		50    & 56.39 & 35.22 \\
		100   & 252.73 & 98.73 \\
		250   & 1838.38 & 346.87 \\
		500   & 6309.39 & ** \\
	\end{tabular}%
	\caption{Computational performances of the dual bundle method and CVX for a three-dimensional multivariate entropic risk measure instance with weight vector $w=(1/3, 1/3, 1/3)$}
		\label{decomcompare2}%
\end{table}%

As observed in Table~\ref{decomcompare1} and Table~\ref{decomcompare2}, the CVX solver overperforms the dual bundle method for smaller numbers of scenarios. However, as the number of scenarios increases, the dual bundle method overperforms the CVX solver. For instance, for $I=10000$ in Table~\ref{decomcompare1}, the CVX solver cannot solve the problem due to a memory error. The same situation is observed for $I=500$ in Table~\ref{decomcompare2}.
\end{example}

For the remainder of this section, we use the multivariate CVaR (\emph{see} Example~\ref{multiCVaR}) and the multivariate entropic risk measure (\emph{see} Example~\ref{multient}) for the choice of $R$. For each risk measure, we consider the biobjective ($J=2$ assets) and the three-objective ($J=3$ assets) cases. We run the primal and dual Benson algorithms with different error parameters ($\epsilon$) and report the total number of scalar optimization problems solved ($\#$opt.), the number of vertices in the final outer approximation ($\#$vert.) and the CPU time in seconds (time).

\begin{example}
(Two-dimensional multivariate CVaR)
We consider $J=2$ assets under $I=500$ scenarios. The parameters of the multivariate CVaR are chosen as $\nu^1=0.8, \nu^2=0.9$. We use error parameter values $\epsilon \in \cb{10^{-2}, 10^{-3}, 10^{-4}}$. The computational results are reported in Table~\ref{table2cvar}. It can be seen that the performances of the primal and dual algorithms are close to each other.

The inner (red lines) and outer (blue lines) approximations of the upper image $\P$ and the lower image $\D$ are given in Figures~\ref{fig1} and \ref{fig2}. These figures are obtained by the primal algorithm. Since the corresponding figures for the dual algorithm are similar, they are omitted. Clearly, the algorithm provides finer approximations for the upper and lower images when $\epsilon$ is reduced from $10^{-3}$ to $10^{-4}$.

\begin{table}[H]
  \centering
\begin{tabular}{llrrr}\hline \hline 
          & $\epsilon$  &  $\#$opt. & $\#$vert. &  time  \\ \hline\hline 
    \multirow{3}[0]{*}{Primal Algorithm} & $10^{-2}$  & 5     & 3     &	2675.69 \\
          & $10^{-3}$ & 11    & 6     	 & 10513.06 \\
          & $10^{-4}$ & 23    & 13	 &	11391.12 \\ \hline
    \multirow{3}[0]{*}{Dual Algorithm} & $10^{-2}$  & 5     & 4     	& 2819.92\\
          & $10^{-3}$& 13    & 8   	& 7021.55
 \\
          & $10^{-4}$ & 25    & 15		& 10007.75
 \\ \hline
    \end{tabular}%
  \caption{Computational results for the two-dimensional multivariate CVaR}\label{table2cvar}
\end{table}%

\begin{figure}[H]
   \centering

   \includegraphics[trim={0cm 0cm 0cm 1.21cm}, clip,width=1\linewidth,height=0.25\textheight,keepaspectratio]{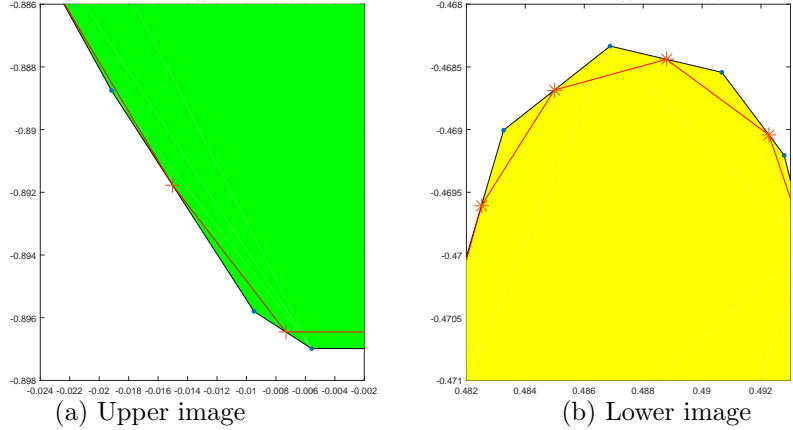}
\vspace{-.8cm}
\caption*{(a) Upper image\hspace{4.2cm}(b) Lower image}
 \caption[]{Inner and outer approximations obtained by the primal algorithm for $\epsilon=10^{-3}$}\label{fig1}
\end{figure}

\begin{figure}[H]
\centering
   \includegraphics[trim={0cm 0cm 0cm 1.21cm}, clip,width=1\linewidth,height=0.25\textheight,keepaspectratio]{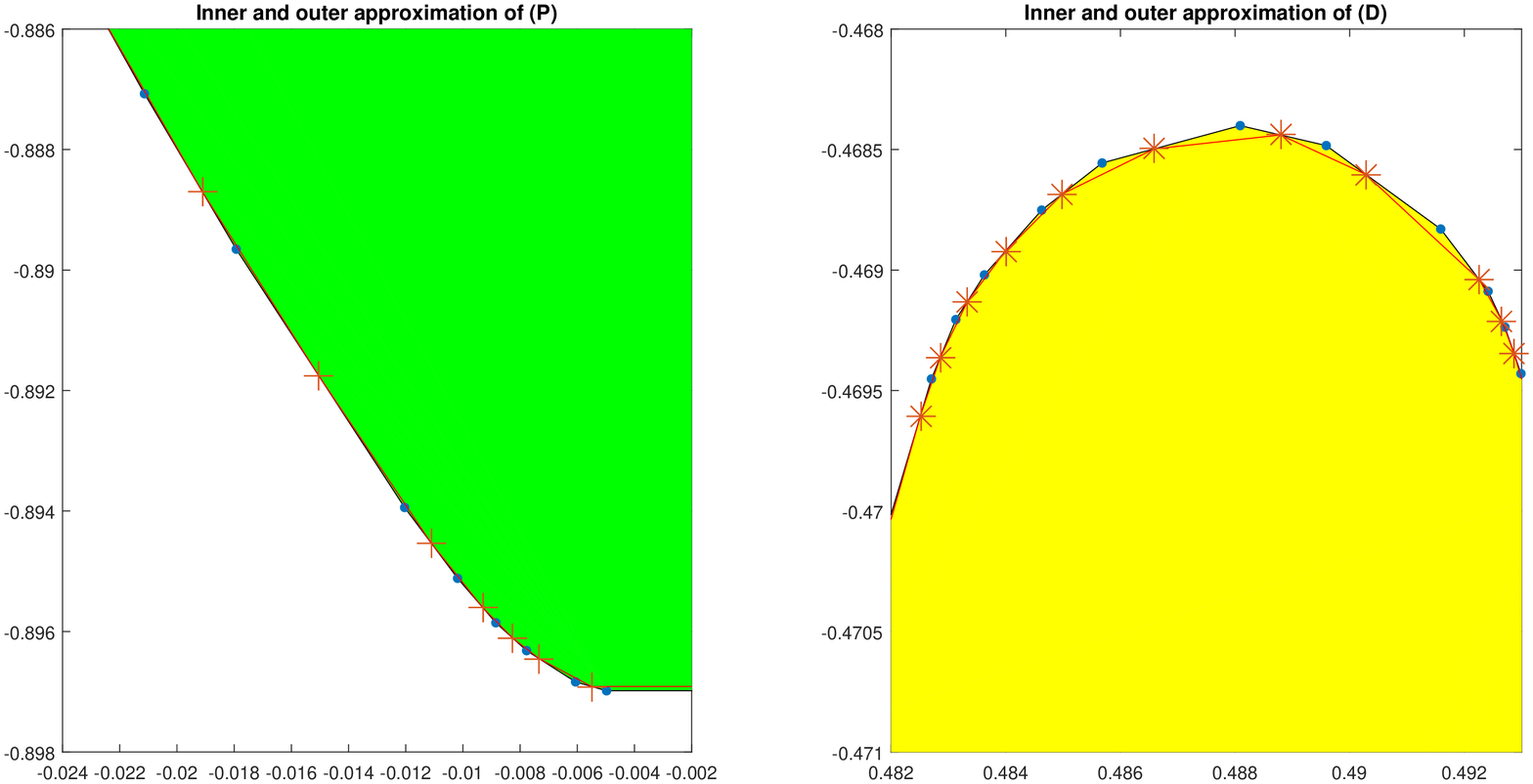}
\vspace{-.8cm}
\caption*{(a) Upper image\hspace{4.2cm}(b) Lower image}
 \caption[]{Inner and outer approximations obtained by the primal algorithm for $\epsilon=10^{-4}$}\label{fig2}
\end{figure}

\end{example}

\begin{example}\label{2dentexample}
(Two-dimensional multivariate entropic risk measure)
We consider $J=2$ assets under $I=500$ scenarios. The parameters of the multivariate entropic risk measure are chosen as $\delta^1=\delta^2=0.1$ and the cone $C$ is generated by the vectors $(2,1)$ and $(1,2)$. We use error parameter values $\epsilon \in \cb{0.1, 0.05, 0.01}$. 

The computational results are reported in Table~\ref{table2ent}. In this example, the dual algorithm solves more optimization problems and enumerates more vertices than the primal algorithm in significantly shorter time. The inner and outer approximations of the upper image $\P$ and the lower image $\D$ obtained by the primal algorithm are given in Figures~\ref{fig3} and \ref{fig4}. Since the corresponding figures for the dual algorithm are similar, they are omitted.





\begin{table}[H]
  \centering
  \begin{tabular}{llrrr}\hline\hline
          & $\epsilon$  &  $\#$opt. & $\#$vert.   & time \\ \hline\hline
    \multirow{3}[0]{*}{Primal Algorithm} & 0.1 &	25 &	13  &	37706.90\\
          & 0.05	& 37 &	19 	& 84730.81 \\
          & 0.01	& 83 &	42	 &	144848.62\\ \hline
    \multirow{3}[0]{*}{Dual Algorithm} &0.1	& 31 &	17 &	13955.43\\
          & 0.05 &	47&	25&	14088.08 \\
          & 0.01 &	85&	44&	17121.26 \\ \hline
    \end{tabular}%
  \caption{Computational results for the two-dimensional multivariate entropic risk measure}\label{table2ent}
  \end{table}%

\begin{figure}[H]
   \centering
   \includegraphics[trim={0cm 0cm 0cm 1.21cm}, clip,width=1\linewidth,height=0.25\textheight,keepaspectratio]{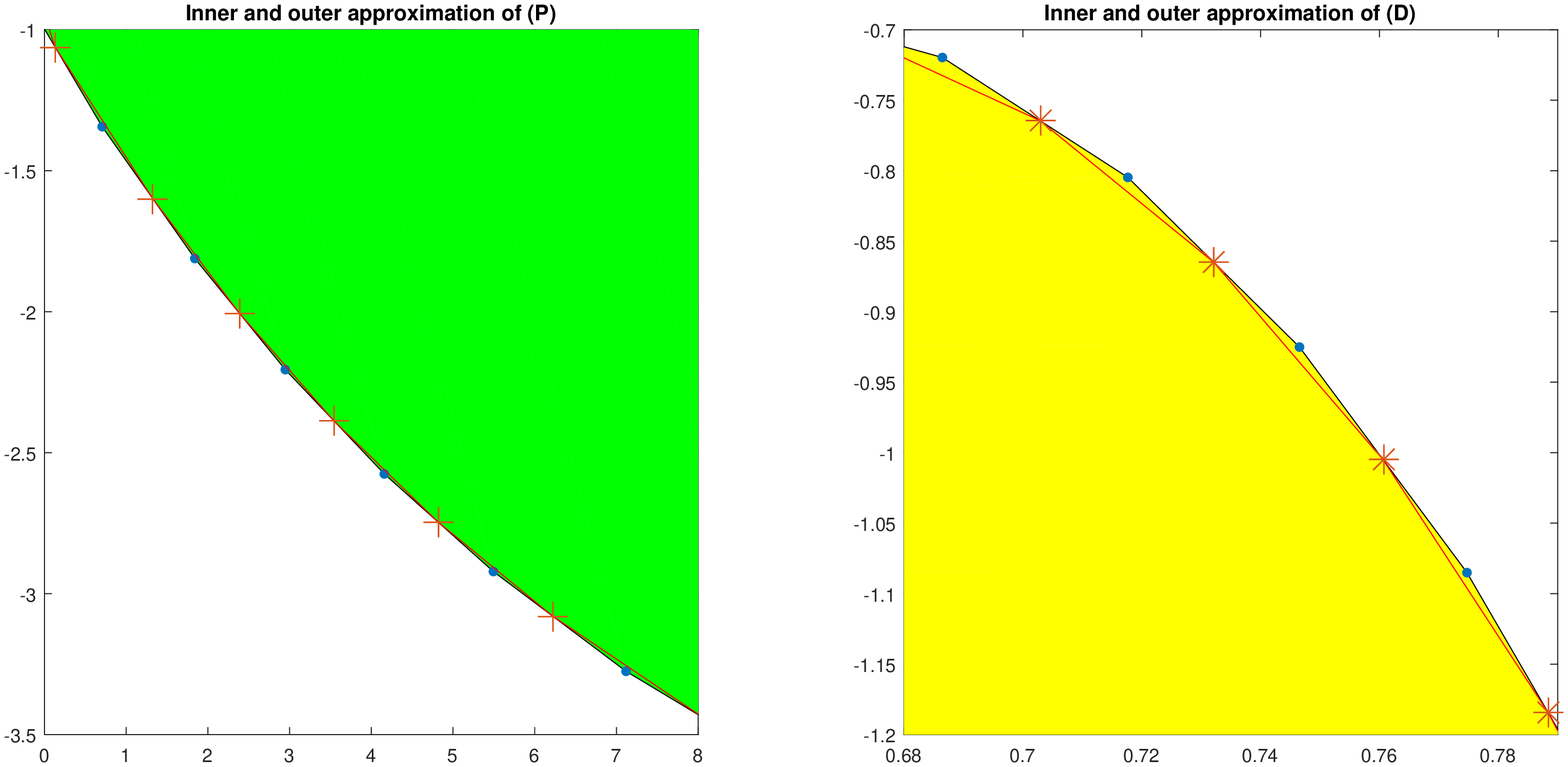}
\vspace{-.8cm}
\caption*{(a) Upper image\hspace{4.2cm}(b) Lower image}
 \caption[]{Inner and outer approximations obtained by the primal algorithm for $\epsilon=0.05$}\label{fig3}
\end{figure}

\begin{figure}[H]
   \centering
   \includegraphics[trim={0cm 0cm 0cm 1.21cm}, clip,width=1\linewidth,height=0.25\textheight,keepaspectratio]{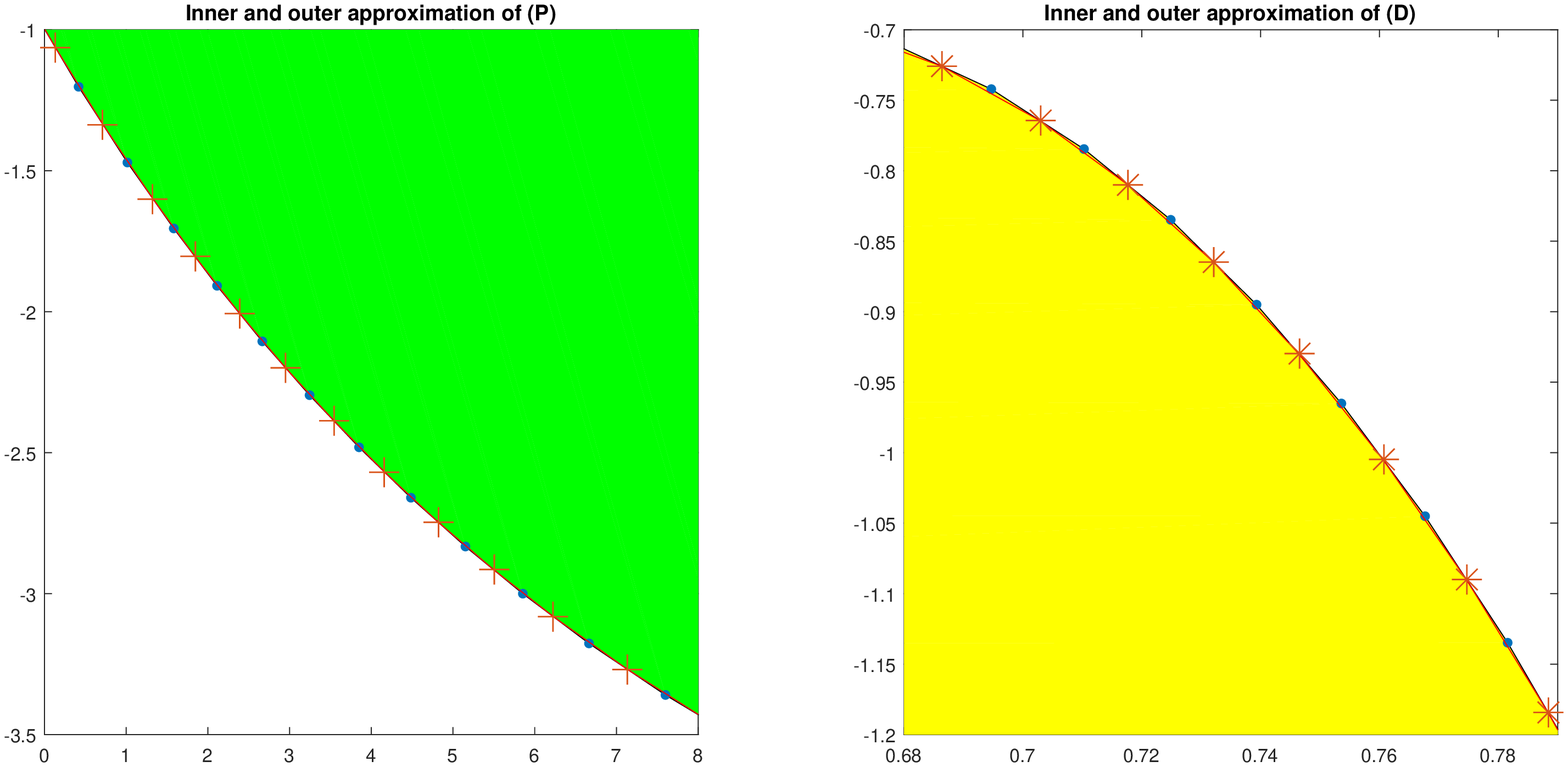}
\vspace{-.8cm}
\caption*{(a) Upper image\hspace{4.2cm}(b) Lower image}
 \caption[]{Inner and outer approximations obtained by the primal algorithm for $\epsilon=0.01$}\label{fig4}
\end{figure}

\end{example}

\begin{example}
(Three-dimensional multivariate CVaR)
We consider $J=3$ assets under $I=250$ scenarios. The parameters of the multivariate CVaR are chosen as $\nu^1=0.8, \nu^2=0.9$. We use error parameter values $\epsilon \in \cb{10^{-2}, 10^{-3}, 10^{-4}}$.

The computational results are reported in Table~\ref{table3cvar}. For $\epsilon=10^{-2}$ and $\epsilon=10^{-3}$, the primal algorithm terminates in shorter time while, for $\epsilon=10^{-4}$, the dual algorithm is faster.

The outer approximations of the upper image $\P$ and the lower image $\D$ obtained by the primal and dual algorithms are given in Figures~\ref{3cvar1}-\ref{3cvar3}. Note that the dots represent the vertices of some polyhedra even if they are not connected by line segments.

As observed in these figures, the primal algorithm provides a better approximation of the lower image compared to the dual algorithm. However, the approximation of the upper image provided by the dual algorithm is better than the one by the primal algorithm.

The multivariate CVaR is defined in terms of the positive part function $(\cdot)^+$, which is piecewise linear. As a result, the upper and lower images of the problem are polyhedral sets and the vertices of their outer approximations in Figures~\ref{3cvar1}-\ref{3cvar3} are generally dense around certain line segments.

\begin{table}[H]
  \centering
  \begin{tabular}{llrrr}\hline \hline 
         & $\epsilon$  &  $\#$opt. & $\#$vert. & time \\ \hline\hline 
    \multirow{3}[0]{*}{Primal Algorithm} & $10^{-2}$ & 21   & 9     & 16856.84 \\
         & $10^{-3}$ & 82   & 32  & 67555.09 \\
         & $10^{-4}$ & 468  & 162   & 319862.68 \\  \hline
    \multirow{3}[0]{*}{Dual Algorithm} & $10^{-2}$ & 24   & 11    & 20303.43 \\
         & $10^{-3}$ & 98   & 36   & 79397.49 \\
         & $10^{-4}$ & 448  & 152  & 249081.86 \\  \hline
    \end{tabular}%
  \caption{Computational results for the three-dimensional multivariate CVaR}\label{table3cvar}
  \end{table}%

\begin{figure}[H]
 \centering
   \begin{subfigure}[b]{1\textwidth}
   \includegraphics[width=1\linewidth,height=0.3\textheight,keepaspectratio]{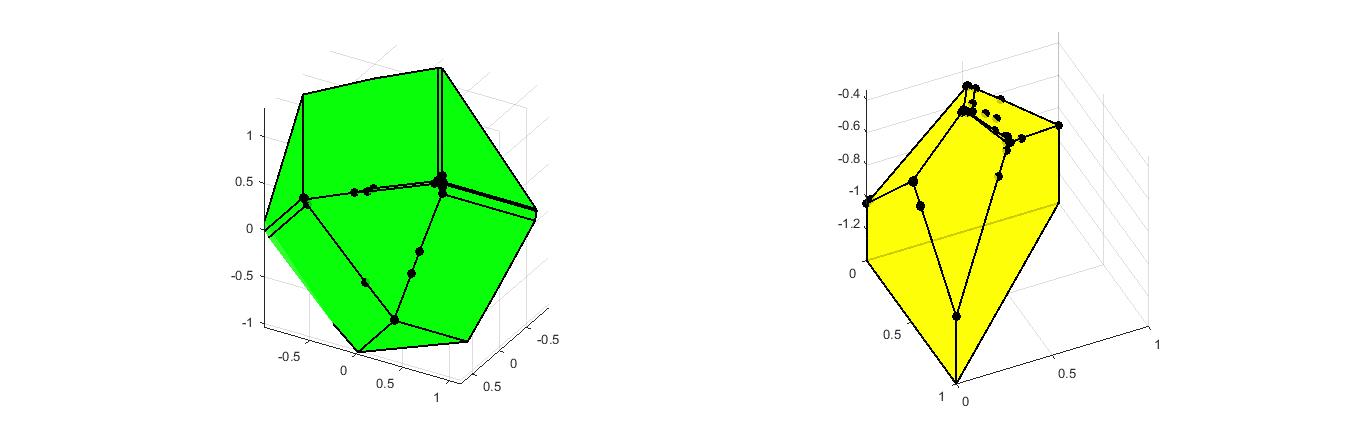}
\caption*{ Upper image\hspace{5.5cm} Lower image}
    \caption{Primal algorithm}
\end{subfigure}

   \begin{subfigure}[b]{1\textwidth}
   \includegraphics[width=1\linewidth,height=0.3\textheight,keepaspectratio]{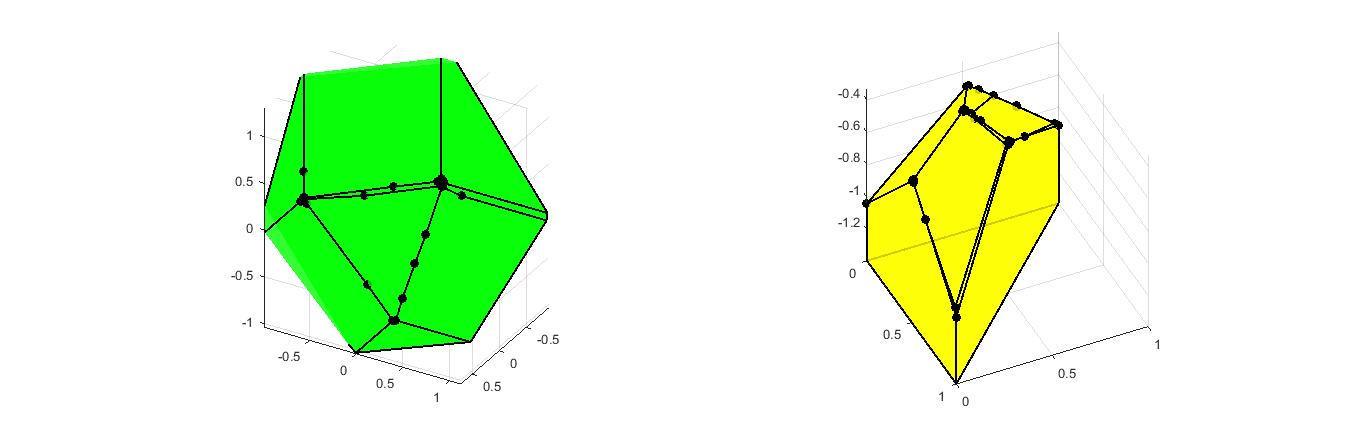}
\caption*{ Upper image\hspace{5.5cm} Lower image}
    \caption{Dual algorithm}
\end{subfigure}
\caption[]{Outer approximations obtained by the primal and dual algorithms for $I=250$ and $\epsilon=10^{-2}$}\label{3cvar1}
\end{figure}

\begin{figure}[H]
 \centering
\begin{subfigure}[b]{1\textwidth}
   \includegraphics[width=1\linewidth,height=0.3\textheight,keepaspectratio]{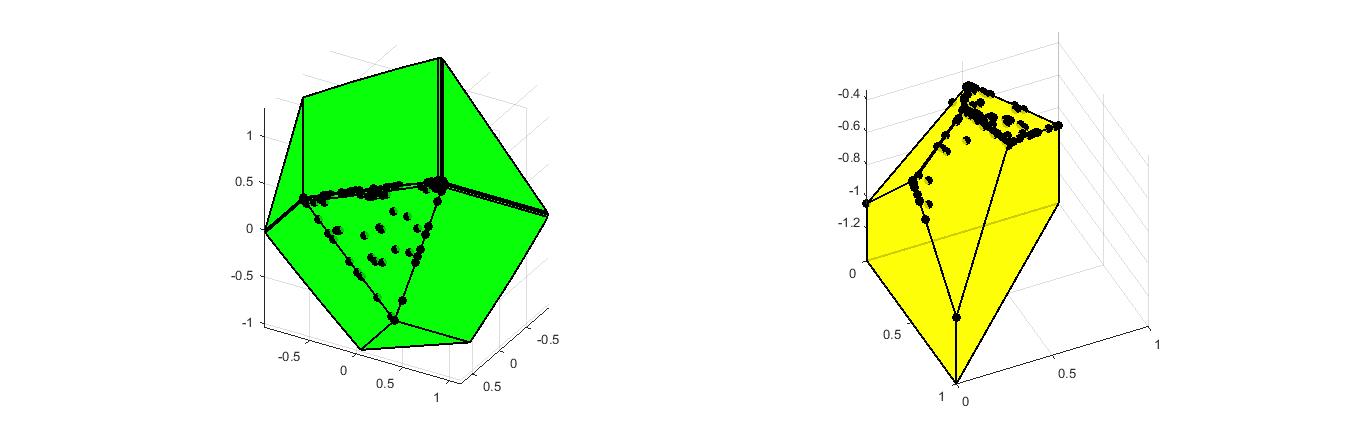}
\caption*{ Upper image\hspace{5.5cm} Lower image}
    \caption{Primal algorithm}
\end{subfigure}

\begin{subfigure}[b]{1\textwidth}
   \includegraphics[width=1\linewidth,height=0.3\textheight,keepaspectratio]{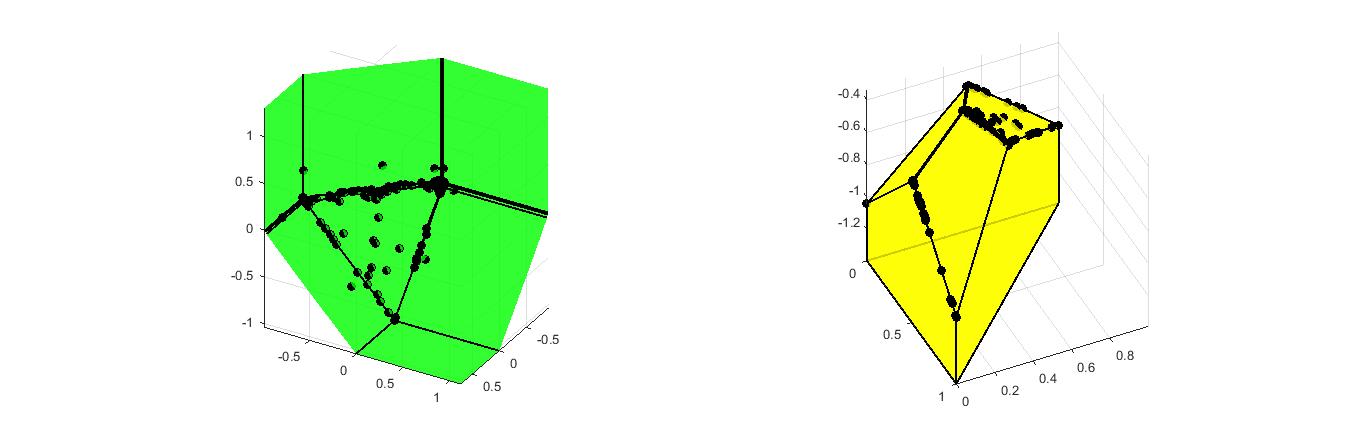}
\caption*{ Upper image\hspace{5.5cm} Lower image}
    \caption{Dual algorithm}
\end{subfigure}
\caption[]{Outer approximations obtained by the primal and dual algorithms for $I=250$ and $\epsilon=10^{-3}$}\label{3cvar2}
\end{figure}

\begin{figure}[H]
 \centering
\begin{subfigure}[b]{1\textwidth}
   \includegraphics[width=1\linewidth,height=0.3\textheight,keepaspectratio]{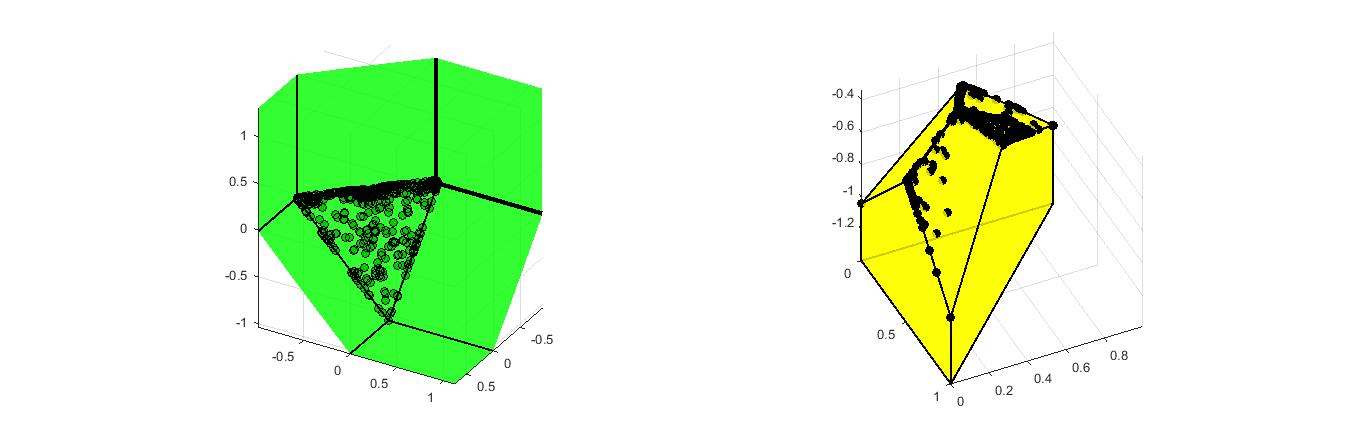}
\caption*{ Upper image\hspace{5.5cm} Lower image}
   \caption{Primal algorithm}
\end{subfigure}

\begin{subfigure}[b]{1\textwidth}
   \includegraphics[width=1\linewidth,height=0.3\textheight,keepaspectratio]{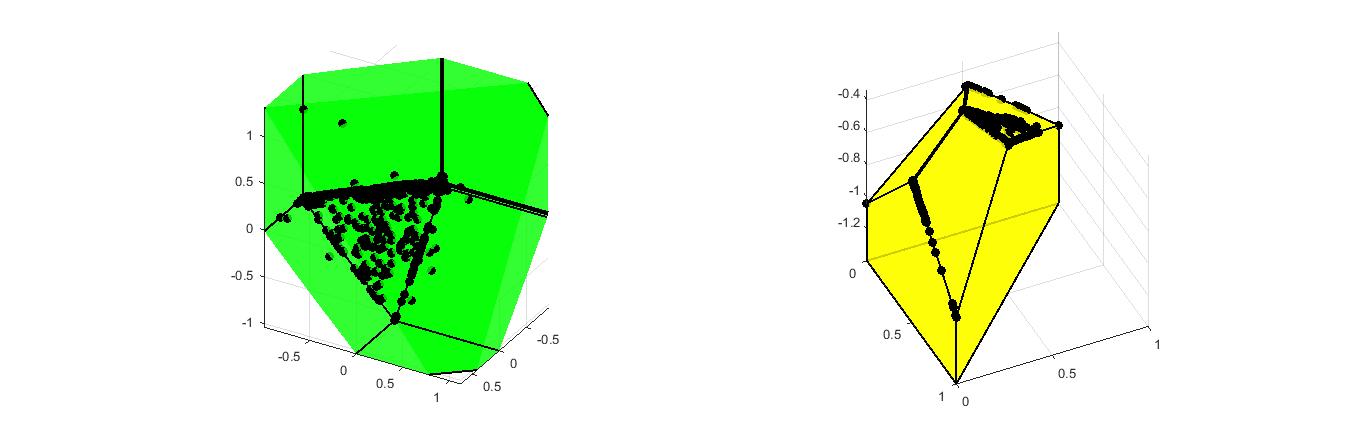}
\caption*{ Upper image\hspace{5.5cm} Lower image}
   \caption{Dual algorithm}
\end{subfigure}
\caption[]{Outer approximations obtained by the primal and dual algorithms for $I=250$ and $\epsilon=10^{-4}$}\label{3cvar3}
\end{figure}
\end{example}

\begin{example}
(Three-dimensional multivariate entropic risk measure)
We consider $J=3$ assets under $I=100$ scenarios. The parameters of the multivariate entropic risk measure are chosen as $\delta^1=\delta^2=\delta^3=0.1$ and the cone $C$ is generated by the vectors $(1,2,3), (3,2,1)$. We use error parameter values $\epsilon \in \cb{0.1, 0.05, 0.01}$. 

The computational results are reported in Table~\ref{table2ent}. We are not able to solve this problem using the primal algorithm as the dual bundle method for $(P_2(v))$ does not converge for some vertices $v$. This is in line with what is reported in \cite[Example~5.4]{convexbenson} for a four-objective problem with multivariate entropic risk measure. The results of the dual algorithm are provided in Table~\ref{table3ent} and Figure~\ref{fig:fig1}.

As the multivariate entropic risk measure is defined in terms of the exponential utility function, which is strictly convex, the upper and lower images are non-polyhedral sets. For this reason, the polyhedral outer approximations of these sets have a more uniform density of vertices over their surfaces compared to the outer approximations for the multivariate CVaR.

\begin{table}[H]
  \centering
 \begin{tabular}{llrrr}\hline\hline 
         & $\epsilon$  &  $\#$opt. & $\#$vert.  & time \\ \hline\hline 
    \multirow{3}[0]{*}{Dual Algorithm} & 0.1  & 196  & 61    & 48742.57 \\
         & 0.05 & 319  & 99   & 82237.89 \\
         & 0.01 & 670  & 211  & 168460.45 \\ \hline
    \end{tabular}%
   \caption{Computational results for the three-dimensional multivariate entropic risk measure}\label{table3ent}
\end{table}%

\begin{figure}[H]
   \centering
   \begin{subfigure}[b]{1\textwidth}
   \includegraphics[width=1\linewidth,height=0.3\textheight,keepaspectratio]{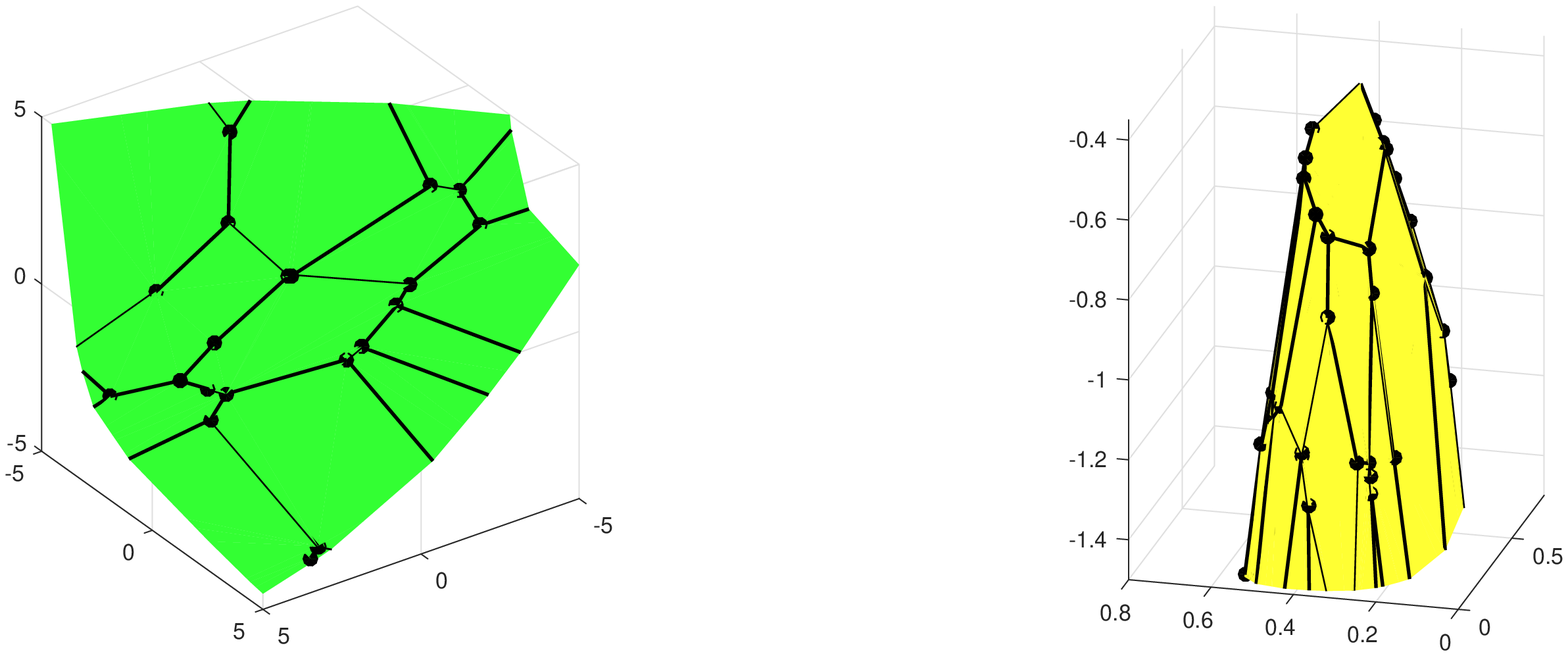}
\caption*{ Upper image\hspace{5.5cm} Lower image}
    \caption{$\epsilon=0.1$}
   \label{fig:Ng1}
\end{subfigure}

\begin{subfigure}[b]{1\textwidth}
   \includegraphics[width=1\linewidth,height=0.3\textheight,keepaspectratio]{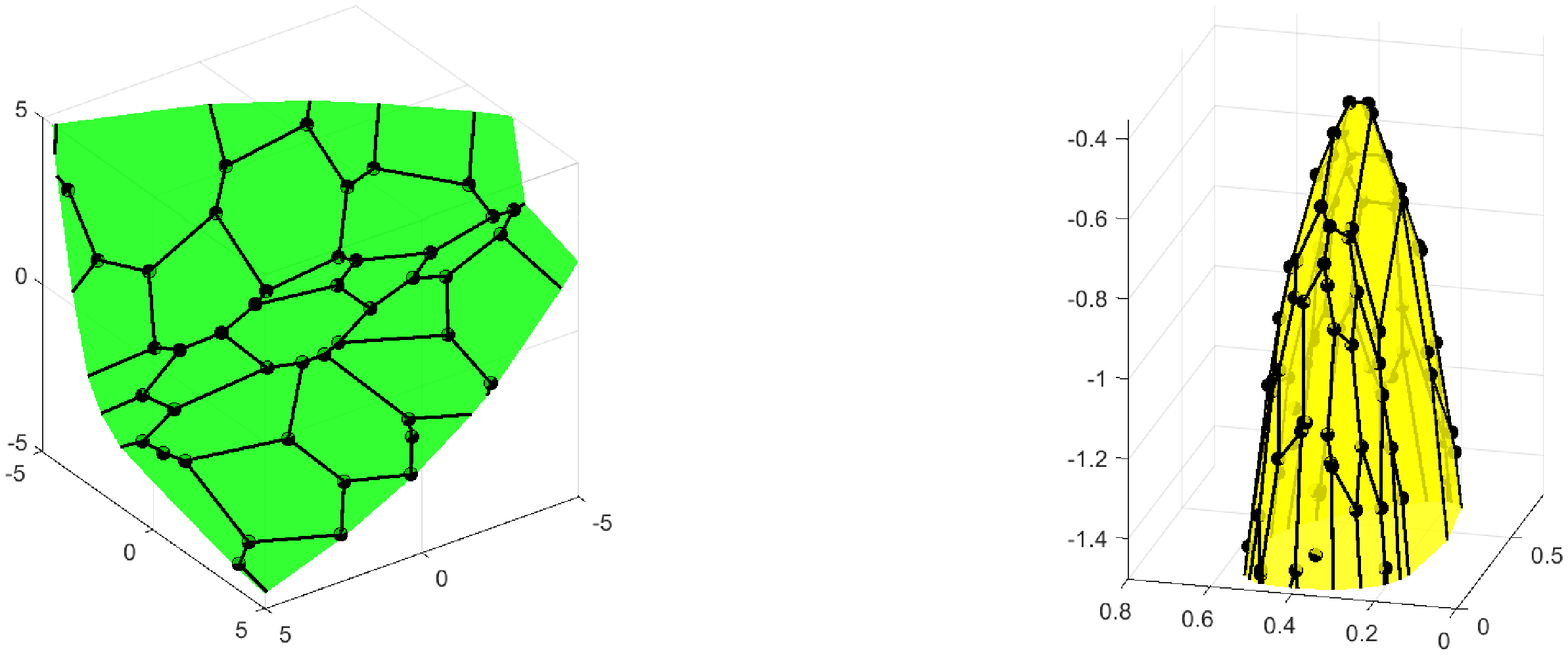}
\caption*{ Upper image\hspace{5.5cm} Lower image}
    \caption{$\epsilon=0.05$}
   \label{fig:Ng2}
\end{subfigure}

\begin{subfigure}[b]{1\textwidth}
   \includegraphics[width=1\linewidth,height=0.3\textheight,keepaspectratio]{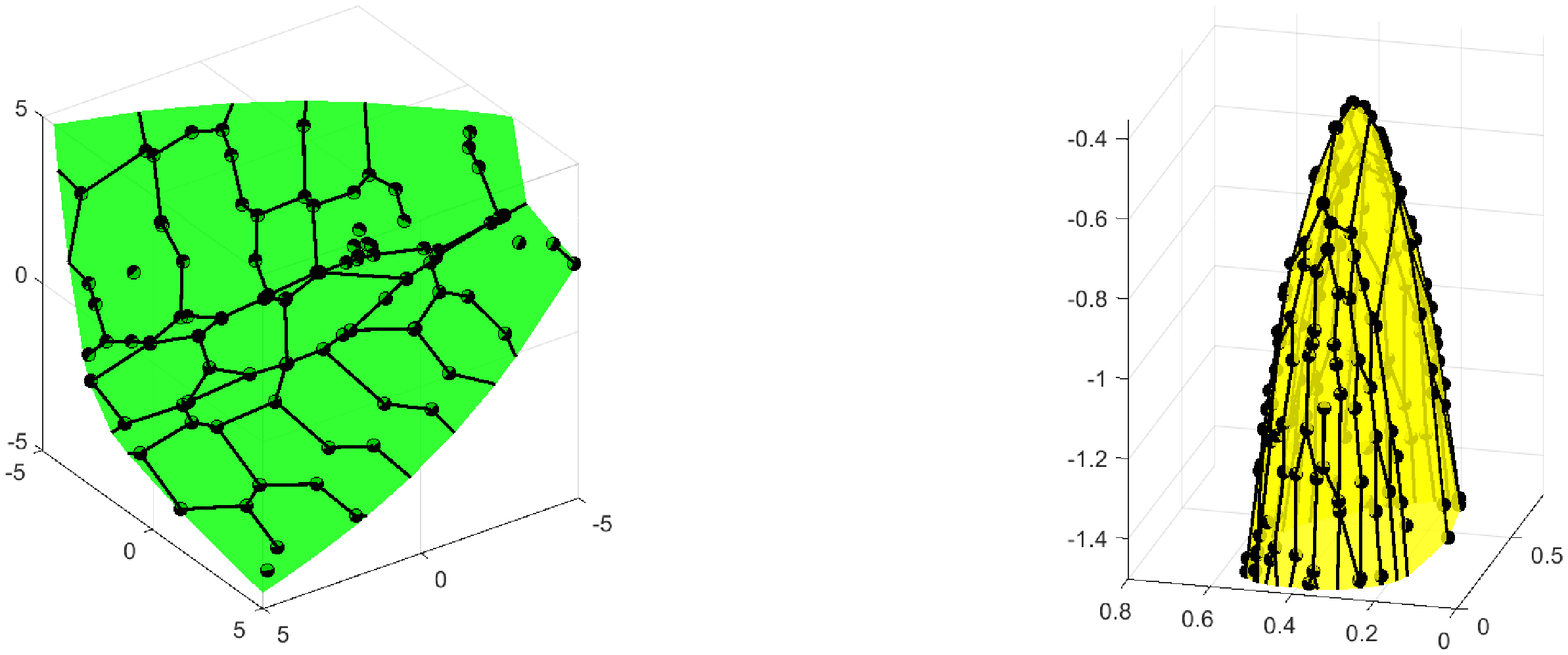}
\caption*{ Upper image\hspace{5.5cm} Lower image}
   \caption{$\epsilon=0.01$}
   \label{fig:Ng3}
\end{subfigure}

\caption[]{Outer approximation obtained by the dual algorithm for $I=100$}
\label{fig:fig1}
\end{figure}

\end{example}

\appendix

\section{Proof of Theorem~\ref{P1primal}}\label{proofof2}

In the setting of Theorem~\ref{P1primal}, let $(\mu^{(n+1)}, \lambda^{(n+1)}, \vartheta^{(n+1)}, \eta^{(n+1)})$ be an optimal solution for $(MP_1(w))$ with index set $\mathcal{L}^\prime=\cb{1,\ldots,n}$. Recall that
\begin{align*}
& F^{(k)}= \sum_{i\in\O}  f_i(\mu_i^{(k)},\lambda_i^{(k)},w)-\b(\mu^{(k)},w),\\
& \bar{F}^{(k)}=\sum_{i\in\O}  f_i(\bar{\mu}^{(k-1)}_i,\bar{\lambda}^{(k-1)}_i,w)-\b(\bar{\mu}^{(k-1)},w).
\end{align*}
Let us define
\begin{align*}
&\phi^{(n+1,k)}\coloneqq \sum_{i\in \O}\varrho \norm{\mu^{(n+1)}_i-\bar{\mu}^{(k)}_i}^2 +\sum_{i\in \O}\varrho \norm{\lambda^{(n+1)}_i-\bar{\lambda}^{(k)}_i}^2,\\
&\bar{\varepsilon}^{(n+1)}\coloneqq\varepsilon+\phi^{(n+1,k)}.
\end{align*}
Furthermore, let $\bar{x}=(\bar{x}_i)_{i\in\O}, \bar{y}=(\bar{y}_i)_{i\in\O}$ be defined as
\begin{equation}\label{xbarybar}
\bar{x}_i\coloneqq \sum_{\ell\in\mathcal{L}^\prime}\tau_i^{(\ell,n+1)}x_i^{(\ell)},\quad \bar{y}_i\coloneqq \sum_{\ell\in\mathcal{L}^\prime}\tau_i^{(\ell,n+1)}y_i^{(\ell)},
\end{equation}
for every $i\in\O$.

\begin{lemma}\label{Lagrangelemma}
The following relationships hold for $(\mu^{(n+1)}$, $\lambda^{(n+1)}$, $\vartheta^{(n+1)}$, $\eta^{(n+1)}$, $\tau^{(n+1)}$, $\theta^{(n+1)}$, $\sigma^{(n+1)}$, $\Psi^{(n+1)}$, $\nu^{(n+1)})$:
\begin{align}
&\sum_{\ell\in\mathcal{L}^\prime}\tau_i^{(\ell,n+1)}=1,\quad \forall i\in\O,\label{convexcomba}\\
&\sum_{\ell\in\mathcal{L}^\prime}\theta^{(\ell,n+1)}=1,\label{convexcombb}\\
&\varrho (\lambda_i^{(n+1)}-\bar{\lambda}^{(k)}_i)=\frac12\of{\sum_{\ell\in\mathcal{L}^\prime}\tau_i^{(\ell,n+1)}g_{\lambda_i^{(\ell)}}-\sigma^{(n+1)} p_i},\quad \forall i\in\O\label{firstorder}\\
& -2\varrho(\mu_i^{(n+1)}\negthinspace-\negthinspace\bar{\mu}^{(k)}_i)\negthinspace+\negthinspace\sum_{\ell\in\mathcal{L}^{\prime}}\tau_i^{(\ell,n+1)}g_{\mu_i^{(\ell)}}\negthinspace +\negthinspace\sum_{\ell\in\mathcal{L}^{\prime}}\theta^{(\ell,n+1)}\rho_{\mu^{(\ell)}_i}-\Psi^{(n+1)}+\nu_i^{(n+1)}=0,\, \forall i\in\O.\label{mufirstorder}
\end{align}
\end{lemma}

\begin{proof}
The Lagrangian for $(MP_1(w))$ with centers $\bar{\mu}^{(k)},\bar{\lambda}^{(k)}$ and index set $\mathcal{L}^{\prime}$ is
\begin{align}\label{lagrangian}
&L(\mu,\lambda,\vartheta,\eta,\tau,\theta,\sigma,\Psi,\nu)\\
&=\sum_{i\in\O}\vartheta_i+\eta-\sum_{i\in\O}\varrho\norm{\mu_i-\bar{\mu}^{(k)}_i}^2-\sum_{i\in\O}\varrho\norm{\lambda_i-\bar{\lambda}^{(k)}_i}^2\notag \\
& \quad +\sum_{i\in\O}\sum_{\ell\in\mathcal{L}^{\prime}}\tau_i^{(\ell)}\of{f_i(\mu_i^{(\ell)},\lambda_i^{(\ell)},w) + g_{\mu^{(\ell)}_i}^{\mathsf{T}}(\mu_i-\mu^{(\ell)}_i) + g_{\lambda^{(\ell)}_i}^{\mathsf{T}}(\lambda_i-\lambda_i^{(\ell)})-\vartheta_i}\notag\\
&\quad +\sum_{\ell\in\mathcal{L}^{\prime}}\theta^{(\ell)}\of{-\beta(\mu^{(\ell)},w) + \sum_{i\in\O}\rho_{\mu^{(\ell)}_i}^{\mathsf{T}}(\mu_i-\mu_i^{(\ell)})-\eta}\notag \\
&\quad -\sigma^{\mathsf{T}}\sum_{i\in\O}p_i\lambda_i+\Psi^{\mathsf{T}}\of{\mathbf{1}-\sum_{i\in\O}\mu_i}+\sum_{i\in\O}\nu_i^{\mathsf{T}}\mu_i.\notag
\end{align}
The dual objective function is defined by
\[
h(\tau,\theta,\sigma,\Psi,\nu)=\sup_{\mu\in\L^J,\lambda\in\L^M,\vartheta\in\L,\eta\in\R}L(\mu,\lambda,\vartheta,\eta,\tau,\theta,\sigma,\Psi,\nu),
\]
and the dual problem is
\begin{align*}
&\text{min }\;       h(\tau,\theta,\sigma,\Psi,\nu)\;\tag{$D-MP_1(w)$}\\
&\text{s.t.}\quad \tau_i^{(\ell)}\geq 0,\theta^{(\ell)}\geq0,\sigma\in\R^M,\Psi\in\R^J,\nu_i\in\R^J_+.
\end{align*}
Note that $(\mu^{(n+1)},\lambda^{(n+1)},\vartheta^{(n+1)},\eta^{(n+1)})$ is an optimal solution for $(MP_1(w))$ with centers $\bar{\mu}^{(k)},\bar{\lambda}^{(k)}$ and index set $\mathcal{L}^{\prime}$, and
\[
\of{\tau^{(n+1)}=(\tau_i^{(\ell,n+1)})_{i\in\O,\ell\in\mathcal{L}^\prime},\theta^{(n+1)}=(\theta^{(\ell,n+1)})_{\ell\in\mathcal{L}^\prime},\sigma^{(n+1)},\Psi^{(n+1)},\nu^{(n+1)}=(\nu_i^{(n+1)})_{i\in\O}}
\]
is the corresponding optimal solution for $(D-MP_1(w))$.
The maximization of the Lagrangian over $\vartheta\in\L$ and $\eta\in\R$ gives \eqref{convexcomba} and \eqref{convexcombb}, respectively, as constraints for the dual problem. The maximization of the Lagrangian over $\lambda\in\L^M$ gives the first-order condition \eqref{firstorder}. Finally, the maximization of the Lagrangian over $\mu\in\L^J$ gives the first-order condition \eqref{mufirstorder}.
\end{proof}

\begin{lemma}\label{feasibility}
The following statements hold for every $i\in\O$:
\begin{enumerate}[(a)]
\item 
As $\varepsilon\rightarrow0$, $\bar{x}_i-\sigma^{(n+1)}\rightarrow 0$.
\item $(\bar{x}_i,\bar{y}_i)\in\mathcal{F}_i$.
\item $(\bar{x},\bar{y})$ will be an element of the set $\cb{(x,y)\in\F\mid p_i(x_i-\E\sqb{x})=0,\;\forall i\in\O}$ as $\varepsilon\rightarrow 0$ in the sense that $(\bar{x},\bar{y})\in\F$ and $p_i(\bar{x}_i-\E\sqb{\bar{x}})\rightarrow 0$ for every $i\in\O$ as $\varepsilon\rightarrow 0$.
\end{enumerate}
\end{lemma}

\begin{proof}
\begin{enumerate}[(a)]
\item 
We have, for every $i\in\O$,
\begin{align}\label{norminequality}
\varrho\sum_{i\in\O}\norm{\lambda^{(n+1)}_i-\bar{\lambda}_i^{(k)}}^2\leq
\phi^{(n+1,k)}.
\end{align}
As $\varepsilon\rightarrow 0$, we have $n\rightarrow\infty$ and, by Lemma~7.17 in \cite{ruszbook},
\begin{equation}\label{normconvergence}
\phi^{(n+1,k)}\rightarrow 0.
\end{equation}

Hence, by \eqref{firstorder}, \eqref{xbarybar}, \eqref{norminequality}, \eqref{normconvergence}, we get
\[
 \frac12\of{\sum_{\ell\in\mathcal{L}^\prime}\tau_i^{(\ell,n+1)}p_i x_i^{(\ell)}-\sigma^{(n+1)} p_i}=\frac12p_i (\bar{x}_i-\sigma^{(n+1)})\rightarrow 0
\]
as $\varepsilon\rightarrow 0$. Then, the claim follows.

\item By line 6 of Algorithm~\ref{decomp1}, we have $(x_i^{(\ell)},y_i^{(\ell)})\in\F_i$ for every $i\in\O$ and $\ell\in\mathcal{L}^\prime$. Since $\F_i$ is convex, $\tau^{(\ell,n+1)}_i\geq 0$ and $\sum_{\ell\in\mathcal{L}^\prime}\tau_i^{(\ell,n+1)}=1$, we have $(\bar{x}_i,\bar{y}_i)\in\F_i$ for every $i\in\O$.

\item Note that $(\bar{x},\bar{y})\in\F$ for every $\varepsilon>0$. Since $\F$ is a closed set, the limit of this point as $\varepsilon\rightarrow 0$ is also an element of $\F$. On the other hand, for every $i\in\O$, $\bar{x}_i-\sigma^{(n+1)}\rightarrow 0$ as $\varepsilon\rightarrow 0$ by part (a). Hence, $\E\sqb{\sigma^{(n+1)}-\bar{x}}=\sum_{i\in\O}p_i(\sigma^{(n+1)}-\bar{x}_{i})\rightarrow 0$ as $\varepsilon\rightarrow 0$. Therefore, for every $i\in\O$,
\[
p_i(\bar{x}_i-\E\sqb{\bar{x}})=p_i(\bar{x}_i-\sigma^{(n+1)})+p_i\E\sqb{\sigma^{(n+1)}-\bar{x}}\rightarrow 0
\]
as $\varepsilon\rightarrow 0$.
\end{enumerate}
\end{proof}

Recall that $\rho_{\mu^{(\ell)}}=(\rho_{\mu_1^{(\ell)}},\ldots,\rho_{\mu_I^{(\ell)}})$ is a subgradient of $-\beta(\cdot,w)$ at $\mu^{(\ell)}$, for each $\ell\in\mathcal{L}^{\prime}$. From \eqref{beta}, there exists some $u^{(\ell)}\in\A$, $\ell\in\mathcal{L}^{\prime}$, such that
\begin{equation}\label{u-ell}
\rho_{\mu_i^{(\ell)}}=-w\cdot u^{(\ell)}_i
\end{equation}
for all $i\in\O$. Since $\A$ is convex, $\theta^{(\ell)}\geq 0$ and $\sum_{\ell\in\mathcal{L}^{\prime}}\theta^{(\ell)}=1$, it follows that
\begin{equation}\label{ubar}
\bar{u}\coloneqq \sum_{\ell\in\mathcal{L}^{\prime}}\theta^{(\ell)}u^{(\ell)}=\of{\sum_{\ell\in\mathcal{L}^{\prime}}\theta^{(\ell)}u_1^{(\ell)},\ldots,\sum_{\ell\in\mathcal{L}^{\prime}}\theta^{(\ell)}u_I^{(\ell)}}\in\A.
\end{equation}

\begin{lemma}
The followings hold:
\begin{enumerate}[(a)]
\item For each $i\in\O$,
\begin{equation}\label{varthetaslackness}
\vartheta^{(n+1)}_i=w^{\mathsf{T}}\sqb{\mu^{(n+1)}_i \cdot (C\bar{x}_i+Q_i \bar{y}_i)}+p_i (\lambda^{(n+1)}_i)^{\mathsf{T}}\bar{x}_i.
\end{equation}
\item For each $i\in\O$,
\begin{equation}\label{approxopt}
\vartheta^{(n+1)}_i - f_i(\mu^{(n+1)}_i, \lambda^{(n+1)}_i, w)\leq\bar{\varepsilon}^{(n+1)}.
\end{equation}
Moreover,
\begin{equation}\label{approxopt0}
\sum_{i\in\O}\vartheta^{(n+1)}_i + \eta^{(n+1)} - \sum_{i\in\O}f_i(\mu^{(n+1)}_i, \lambda^{(n+1)}_i, w)+\beta(\mu^{(n+1)},w)\leq \bar{\varepsilon}^{(n+1)}.
\end{equation}
\item
\begin{equation}\label{eta-beta}
\eta^{(n+1)}+\beta(\mu^{(n+1)},w)\leq \bar{\varepsilon}^{(n+1)}.
\end{equation}
\item $-w\cdot \bar{u}$ is an $\bar{\varepsilon}^{(n+1)}$-subgradient of $-\beta(\cdot,w)$ at $\mu^{(n+1)}$ in the sense that, for every $\mu\in\M_1^J$,
\begin{equation}\label{betaapproxsbg}
-\beta(\mu,w) \leq -\beta(\mu^{(n+1)},w)+\bar{\varepsilon}^{(n+1)}-\sum_{i\in\O}w^{\mathsf{T}}(\mu_i-\mu_i^{(n+1)})\cdot \bar{u}_i.
\end{equation}
\end{enumerate}
\end{lemma}

\begin{proof}
Consider $(MP_1(w))$ with centers $\bar{\mu}^{(k)},\bar{\lambda}^{(k)}$ and index set $\mathcal{L}^{\prime}$.
\begin{enumerate}[(a)]
\item Note that constraint \eqref{grad1} can be rewritten as
\[
\vartheta_i \leq w^{\mathsf{T}}\sqb{\mu_i \cdot (Cx_i^{(\ell)}+Q_i y_i^{(\ell)})}+p_i \lambda_i^{\mathsf{T}}x^{(\ell)}_i, \quad \forall i\in\O,  \ell\in\mathcal{L}^\prime
\]
since
\[
f_i(\mu_i^{(\ell)},\lambda_i^{(\ell)},w)=w^{\mathsf{T}}\sqb{\mu^{(\ell)}_i \cdot (Cx_i^{(\ell)}+Q_i y_i^{(\ell)})}+p_i (\lambda^{(\ell)}_i)^{\mathsf{T}}x^{(\ell)}_i
\]
and $g_{\mu^{(\ell)}_i}=w\cdot(Cx_i^{(\ell)}+Q_iy_i^{(\ell)})$, $g_{\lambda_i^{(\ell)}} = p_ix_i^{(\ell)}$. From the complementary slackness conditions for constraint \eqref{grad1} and using \eqref{xbarybar}, \eqref{convexcomba}, we get
\begin{align*}
\vartheta^{(n+1)}_i &=\sum_{\ell\in\mathcal{L}^\prime}\vartheta^{(n+1)}_i\tau^{(\ell,n+1)}_i\\
&= \sum_{\ell\in\mathcal{L}^\prime}\of{w^{\mathsf{T}}\sqb{\mu^{(n+1)}_i \cdot (Cx_i^{(\ell)}+Q_i y_i^{(\ell)})}+p_i (\lambda^{(n+1)}_i)^{\mathsf{T}}x^{(\ell)}_i}\tau^{(\ell,n+1)}_i\notag \\
&=w^{\mathsf{T}}\sqb{\mu^{(n+1)}_i \cdot (C\bar{x}_i+Q_i \bar{y}_i)}+p_i (\lambda^{(n+1)}_i)^{\mathsf{T}}\bar{x}_i.\notag
\end{align*}
Hence, \eqref{varthetaslackness} follows.

\item Note that
\begin{align*}
&\vartheta^{(n+1)}_i - f_i(\mu^{(n+1)}_i, \lambda^{(n+1)}_i, w)-\phi^{(n+1,k)}\\
&\leq \sum_{i\in\O}\vartheta^{(n+1)}_i + \eta^{(n+1)} - \sum_{i\in\O}f_i(\mu^{(n+1)}_i, \lambda^{(n+1)}_i, w)+\beta(\mu^{(n+1)},w)-\phi^{(n+1,k)} \\
&\leq \sum_{i\in\O}\vartheta^{(n+1)}_i + \eta^{(n+1)} - \sum_{i\in\O}f_i(\bar{\mu}^{(k)}_i, \bar{\lambda}^{(k)}_i, w)+\beta(\bar{\mu}^{(k)},w)-\phi^{(n+1,k)}\\
&\leq \sum_{i\in\O}\vartheta^{(k+1)}_i + \eta^{(k+1)} - \sum_{i\in\O}f_i(\bar{\mu}^{(k)}_i, \bar{\lambda}^{(k)}_i, w)+\beta(\bar{\mu}^{(k)},w)-\phi^{(k+1,k)}\\
&\leq \varepsilon.
\end{align*}
Here, the first inequality follows since $\vartheta_i^{(n+1)}\geq f_i(\mu_i^{(n+1)},\lambda_i^{(n+1)},w)$ for each $i\in\O$ and $\eta^{(n+1)}\geq -\beta(\mu^{(n+1)},w)$, the second inequality follows since
\begin{align*}
&F^{(n+1)}=\sum_{i\in\O}f_i(\mu^{(n+1)}_i, \lambda^{(n+1)}_i, w)-\beta(\mu^{(n+1)},w)\\
&\geq \bar{F}^{(k+1)}=\sum_{i\in\O}f_i(\bar{\mu}^{(k)}_i, \bar{\lambda}^{(k)}_i, w)-\beta(\bar{\mu}^{(k)},w)
\end{align*}
due to the center update rule in line~14 of Algorithm~\ref{decomp1}, the third inequality follows since the master problem with index set $\mathcal{L}=\cb{1,\ldots,k}$ and center $(\bar{\mu}^{(k)},\bar{\lambda}^{(k)})$ has a smaller optimal value than the master problem with index set $\mathcal{L}^{\prime}=\cb{1,\ldots,n}$ and center $(\bar{\mu}^{(k)},\bar{\lambda}^{(k)})$. Finally, the last inequality is by the approximate stopping condition \eqref{approxstop}.
Hence, \eqref{approxopt} and \eqref{approxopt0} follow.

\item Similar to part (b), we can show that \eqref{eta-beta} holds.

\item Note that constraint \eqref{grad2} can be rewritten as
\[
\eta \leq \sum_{i\in\O}\rho_{\mu^{(\ell)}_i}^{\mathsf{T}}\mu_i, \quad \forall \ell\in\mathcal{L}^{\prime}
\]
since $-\beta(\mu^{(\ell)},w)=\sum_{i\in\O}\rho_{\mu^{(\ell)}_i}^{\mathsf{T}}\mu_i^{(\ell)}$ by the definition of subgradient. From the complementary slackness conditions,
\begin{equation}\label{etam}
\eta^{(n+1)}=\sum_{\ell\in\mathcal{L}^{\prime}}\theta^{(\ell)}\sum_{i\in\O}\rho_{\mu^{(\ell)}_i}^{\mathsf{T}}\mu^{(n+1)}_i=\sum_{i\in\O}\mu^{(n+1)}_i\sum_{\ell\in\mathcal{L}^{\prime}}\theta^{(\ell)}\rho_{\mu^{(\ell)}_i}^{\mathsf{T}}.
\end{equation}
From \eqref{u-ell}, \eqref{ubar}, \eqref{etam}, it follows that
\begin{equation}\label{etaslackness}
\eta^{(n+1)} = -w^{\mathsf{T}}\sum_{i\in\O}\mu^{(n+1)}_i\cdot \bar{u}_i=-w^{\mathsf{T}}\E^{\mu^{(n+1)}}\sqb{\bar{u}}.
\end{equation}
For every $\mu\in\M_1^J$,
\begin{align*}
-\beta(\mu,w)&\leq -w^{\mathsf{T}}\E^{\mu}\sqb{\bar{u}}\\
&=-w^{\mathsf{T}}\E^{\mu^{(n+1)}}\sqb{\bar{u}}+w^{\mathsf{T}}\of{\E^{\mu^{(n+1)}}\sqb{\bar{u}}-\E^{\mu}\sqb{\bar{u}}}\\ &=-w^{\mathsf{T}}\E^{\mu^{(n+1)}}\sqb{\bar{u}}-\sum_{i\in\O}w^{\mathsf{T}}(\mu_i-\mu_i^{(n+1)})\cdot \bar{u}_i \notag \\
& \leq -\beta(\mu^{(n+1)},w)+\bar{\varepsilon}^{(n+1)}-\sum_{i\in\O}w^{\mathsf{T}}(\mu_i-\mu_i^{(n+1)})\cdot \bar{u}_i,
\end{align*}
where the first inequality follows from \eqref{beta}, and the last inequality follows from \eqref{eta-beta} and \eqref{etaslackness}. Hence, the claim follows.
\end{enumerate}
\end{proof}

\begin{lemma}\label{optimality}
The followings hold:
\begin{enumerate}[(a)]
\item As $\varepsilon\rightarrow0$,
\begin{align*}
&w^{\mathsf{T}}\E^{\mu^{(n+1)}}\sqb{C\bar{x}+Q\bar{y}}-\beta(\mu^{(n+1)},w)\rightarrow\mathscr{P}_1(w),\\ &\sum_{i\in\O}f_i(\mu_i^{(n+1)},\lambda_i^{(n+1)},w)-\beta(\mu^{(n+1)},w)\rightarrow\mathscr{P}_1(w).
\end{align*}
\item Let $z_{(w)}$ be a minimizer of the problem
\[
\inf_{z\in R(C\bar{x}+Q\bar{y})}w^{\mathsf{T}}z.
\]
Then, $w^{\mathsf{T}}z_{(w)}\rightarrow\mathscr{P}_1(w)$ as $\varepsilon\rightarrow 0$.
\end{enumerate}
\end{lemma}

\begin{proof}
\begin{enumerate}[(a)]
\item Note that
\begin{align}\label{optimalitybound}
\sum_{i\in\O}\vartheta^{(n+1)}_i&= \sum_{i\in\O}\of{w^{\mathsf{T}}\sqb{\mu^{(n+1)}_i \cdot (C\bar{x}_i+Q_i \bar{y}_i)}+p_i (\lambda^{(n+1)}_i)^{\mathsf{T}}\bar{x}_i} \notag \\
&= w^{\mathsf{T}}\E^{\mu^{(n+1)}}\sqb{ C\bar{x}+Q \bar{y}}+\sum_{i\in\O}p_i (\lambda^{(n+1)}_i)^{\mathsf{T}}\bar{x}_i,
\end{align}
where the first equality follows from \eqref{varthetaslackness}. We may rewrite $\sum_{i\in\O}p_i(\lambda_i^{(n+1)})^{\mathsf{T}}\bar{x}_i$ as
\begin{align}\label{plambdax}
\sum_{i\in\O}p_i(\lambda_i^{(n+1)})^{\mathsf{T}}\bar{x}_i&=\sum_{i\in\O}p_i(\lambda_i^{(n+1)})^{\mathsf{T}}(\bar{x}_i-\sigma^{(n+1)})+\sum_{i\in\O}p_i(\lambda_i^{(n+1)})^{\mathsf{T}}\sigma^{(n+1)}\\
&=\sum_{i\in\O}p_i(\lambda_i^{(n+1)})^{\mathsf{T}}(\bar{x}_i-\sigma^{(n+1)})\notag
\end{align}
since $\sigma^{(n+1)}$ is deterministic and $\E\sqb{\lambda^{(n+1)}}=0$. Note that
\begin{equation}\label{lambdabound}
\abs{\sum_{i\in\O} p_i (\lambda_i^{(n+1)})^{\mathsf{T}}(\bar{x}_i-\sigma^{(n+1)})}\leq \sum_{i\in\O}p_i {|\lambda^{(n+1)}_i|}^{\mathsf{T}}|\bar{x}_i-\sigma^{(n+1)}|,
\end{equation}
where $|z|\coloneqq(|z^1|,\ldots,|z^J|)$ for $z\in\R^J$. By the proof of Theorem~7.16 in \cite{ruszbook}, we have $|(\lambda_i^{(n+1)})^j|\leq \tilde{C}$ for every $i\in\O, j\in\mathcal{J}$, where $\tilde{C}>0$ is some constant. Using \eqref{plambdax} and \eqref{lambdabound}, we get
\begin{equation}\label{optimalitybound2}
\abs{\sum_{i\in\O}p_i(\lambda_i^{(n+1)})^{\mathsf{T}}\bar{x}_i}\leq  \tilde{C} \sum_{i\in\O}p_i\1^\mathsf{T}\abs{\bar{x}_i-\sigma^{(n+1)}}.
\end{equation}
Therefore, by Lemma~\ref{feasibility}(a), \eqref{optimalitybound} and \eqref{optimalitybound2},
\begin{equation}\label{absolute}
\abs{\sum_{i\in\O}\vartheta^{(n+1)}_i- w^{\mathsf{T}}\E^{\mu^{(n+1)}}\sqb{ C\bar{x}+Q \bar{y}}}= \abs{\sum_{i\in\O}p_i (\lambda^{(n+1)}_i)^{\mathsf{T}}\bar{x}_i}\rightarrow 0
\end{equation}
as $\varepsilon\rightarrow 0$. On the other hand, by \eqref{approxopt0} and using the fact that $\vartheta_i^{(n+1)}$, $\eta^{(n+1)}$ are upper approximations for $f_i(\mu_i^{(n+1)},\lambda_i^{(n+1)},w)$, $-\beta(\mu^{(n+1)},w)$, respectively, we get
\begin{align}\label{absolute2}
0 & \leq \sum_{i\in\O}\vartheta^{(n+1)}_i-\beta(\mu^{(n+1)},w)-\sum_{i\in\O}f_i(\mu_i^{(n+1)},\lambda_i^{(n+1)},w)+\beta(\mu^{(n+1)},w) \\
&\leq \sum_{i\in\O}\vartheta^{(n+1)}_i+\eta^{(n+1)}-\sum_{i\in\O}f_i(\mu_i^{(n+1)},\lambda_i^{(n+1)},w)+\beta(\mu^{(n+1)},w)\leq \bar{\varepsilon}^{(n+1)}.\notag
\end{align}
By \eqref{normconvergence}, $\bar{\varepsilon}^{(n+1)}=\varepsilon+\phi^{(n+1,k)}\rightarrow 0$ as $\varepsilon\rightarrow 0$. By Lemma~7.17 in \cite{ruszbook}, as $\varepsilon\rightarrow 0$, $\sum_{i\in\O}\vartheta^{(n+1)}_i+\eta^{(n+1)}$ converges to the optimal value of $(D_1(w))$, which is $\mathscr{P}_1(w)$. Hence, $\sum_{i\in\O}f_i(\mu_i^{(n+1)},\lambda_i^{(n+1)},w)-\beta(\mu^{(n+1)},w)$ also converges to $\mathscr{P}_1(w)$ as $\varepsilon\rightarrow 0$. Finally, by triangle inequality, \eqref{absolute} and \eqref{absolute2} yield
\begin{align*}
\abs{w^{\mathsf{T}}\E^{\mu^{(n+1)}}\sqb{C\bar{x}+Q\bar{y}}-\beta(\mu^{(n+1)},w)-\sum_{i\in\O}f_i(\mu_i^{(n+1)},\lambda_i^{(n+1)},w)+\beta(\mu^{(n+1)},w)}\leq \abs{\sum_{i\in\O}p_i (\lambda^{(n+1)}_i)^{\mathsf{T}}\bar{x}_i}+\bar{\varepsilon}^{(n+1)}.
\end{align*}
From \eqref{absolute}, the right hand side converges to zero as $\varepsilon\rightarrow 0$. We conclude that $w^{\mathsf{T}}\E^{\mu^{(n+1)}}\sqb{C\bar{x}+Q\bar{y}}-\beta(\mu^{(n+1)},w)$ also converges to $\mathscr{P}_1(w)$ as $\varepsilon\rightarrow 0$.

\item By \eqref{scalar-dual}, we have
\begin{align*}
w^{\mathsf{T}}z_{(w)}&=\sup_{\mu\in\M_1^J}\of{w^{\mathsf{T}}\E^{\mu}\sqb{C\bar{x}+Q\bar{y}}-\beta(\mu,w)}\\
&=\sup_{\mu\in\L^J}\cb{w^{\mathsf{T}}\E^{\mu}\sqb{C\bar{x}+Q\bar{y}}-\beta(\mu,w)\mid \sum_{i\in\O}\mu_i=\mathbf{1},\mu_i\in\R^J_+,\forall i\in\O}.
\end{align*}
The corresponding Lagrangian dual problem is given by
\[
\inf_{\Psi\in\R^J,\nu\in\L^J_+}\sup_{\mu\in\L^J}\of{w^{\mathsf{T}}\E^{\mu}\sqb{C\bar{x}+Q\bar{y}}-\beta(\mu,w)+\Psi^{\mathsf{T}}\of{\mathbf{1}-\sum_{i\in\O}\mu_i}+\sum_{i\in\O}\nu_i^{\mathsf{T}}\mu_i}.
\]
The optimization over $\mu_1,\ldots,\mu_I$ gives the first-order condition
\begin{equation}\label{mufirstorder2}
w\cdot(C\bar{x}_i+Q_i\bar{y}_i)+\rho_{\mu_i}-\Psi+\nu_i=0, \quad \forall i\in\O,
\end{equation}
for some $\rho_{\mu}=(\rho_{\mu_1},\ldots,\rho_{\mu_I})\in\partial_{\mu}(-\beta)(\mu,w)$.
We argue that $(\mu^{(n+1)},\Psi^{(n+1)},\nu^{(n+1)})$ satisfies \eqref{mufirstorder2} in an approximate sense. Note that we may rewrite \eqref{mufirstorder} as
\begin{align}\label{optcond}
&-2\varrho(\mu_i^{(n+1)}-\bar{\mu}^{(k)}_i)+\sum_{\ell\in\mathcal{L}^{\prime}}\tau_i^{(\ell,n+1)}w\cdot(Cx^{(\ell)}_i+Q_iy_i^{(\ell)})-\sum_{\ell\in\mathcal{L}^{\prime}}\theta^{(\ell,n+1)}w\cdot u_i^{(\ell)}-\Psi^{(n+1)}+\nu_i^{(n+1)}\notag\\
&=-2\varrho(\mu_i^{(n+1)}-\bar{\mu}^{(k)}_i)+w\cdot(C\bar{x}_i+Q_i\bar{y}_i)-w\cdot \bar{u}_i-\Psi^{(n+1)}+\nu_i^{(n+1)}\\
&=0.\notag
\end{align}
As shown in \eqref{betaapproxsbg}, $-w\cdot\bar{u}=-w\cdot(\bar{u}_1,\ldots,\bar{u}_I)$ is an $\bar{\varepsilon}^{(n)}$-subgradient of $-\beta(\cdot,w)$ at $\mu^{(n+1)}$. Hence, there exists a subgradient $\rho_{\mu^{(n+1)}}$ of $-\beta(\cdot,w)$ at $\mu^{(n+1)}$ for each $n$ such that $-w\cdot\bar{u}-\rho_{\mu^{(n+1)}}\rightarrow 0$ as $\varepsilon\rightarrow 0$. On the other hand, by \eqref{normconvergence}, $\mu_i^{(n+1)}-\bar{\mu}^{(k)}_i\rightarrow 0$ as $\varepsilon\rightarrow0$. Therefore, from \eqref{optcond}, \eqref{mufirstorder2} is satisfied by $(\mu^{(n+1)},\Psi^{(n+1)},\nu^{(n+1)})$ and $-w\cdot\bar{u}$ approximately in the sense that
\[
w\cdot(C\bar{x}_i+Q_i\bar{y}_i)-w\cdot \bar{u}_i-\Psi^{(n+1)}+\nu^{(n+1)}_i\rightarrow 0
\]
and $\abs{w^{\mathsf{T}}\E^{\mu^{(n+1)}}\sqb{C\bar{x}+Q\bar{y}}-\beta(\mu^{(n+1)},w)-w^{\mathsf{T}}z_{(w)}}\rightarrow 0$ in the limit as $\varepsilon\rightarrow 0$.
\end{enumerate}
\end{proof}

\begin{proof}[Proof of Theorem~\ref{P1primal}]
Note that $x_{(w)}=\bar{x}$ and $y_{(w)}=\bar{y}$, where $\bar{x},\bar{y}$ are defined by \eqref{xbarybar}. Parts (a) and (c) follow directly from Lemma~\ref{feasibility}. Parts (b) follows directly from the definition of $z_{(w)}$. Part (d) follows from Lemma~\ref{optimality}.
\end{proof}

\section{Proof of Theorem~\ref{P2primal}}\label{proofof4}

In the setting of Theorem~\ref{P2primal}, let $(m^{(n+1)},\lambda^{(n+1)},\vartheta^{(n+1)},\eta^{(n+1)})$ be an optimal solution for $(MP_2(v))$ with index set $\mathcal{L}^\prime=\cb{1,\ldots,n}$. Recall that
\begin{align*}
&F^{(k)}= \sum_{i\in\O}  \tilde{f}_i(m_i^{(k)},\lambda_i^{(k)})-\tilde{\beta}(m^{(k)})-\sum_{i\in\O}(m^{(k)}_i)^{\mathsf{T}}v,\\
&\bar{F}^{(k)}=\sum_{i\in\O}  \tilde{f}_i(\bar{m}^{(k-1)}_i,\bar{\lambda}^{(k-1)}_i)-\tilde{\b}(\bar{m}^{(k-1)})-\sum_{i\in\O}(\bar{m}^{(k-1)}_i)^{\mathsf{T}}v.
\end{align*}
Let us re-define
\begin{align*}
&\phi^{(n+1,k)}\coloneqq \sum_{i\in \O}\varrho \norm{m^{(n+1)}_i-\bar{m}^{(k)}_i}^2 +\sum_{i\in \O}\varrho \norm{\lambda^{(n+1)}_i-\bar{\lambda}^{(k)}_i}^2,\\
&\bar{\varepsilon}^{(n+1)}\coloneqq\varepsilon+\phi^{(n+1,k)}.
\end{align*}
Furthermore, let $\bar{x}=(\bar{x}_i)_{i\in\O}, \bar{y}=(\bar{y}_i)_{i\in\O}$ be defined as
\begin{equation}\label{xbarybar2}
\bar{x}_i\coloneqq \sum_{\ell\in\mathcal{L}^\prime}\tau_i^{(\ell,n+1)}x_i^{(\ell)},\quad \bar{y}_i\coloneqq \sum_{\ell\in\mathcal{L}^\prime}\tau_i^{(\ell,n+1)}y_i^{(\ell)},
\end{equation}
for every $i\in\O$.

\begin{lemma}\label{Lagrangelemma2}
The following relationships hold for $(m^{(n+1)}$, $\lambda^{(n+1)}$, $\vartheta^{(n+1)}$, $\eta^{(n+1)}$, $\tau^{(n+1}$, $\theta^{(n+1)}$, $\sigma^{(n+1)}$, $\psi^{(n+1)}$, $\nu^{(n+1)})$:
\begin{align}
&\sum_{\ell\in\mathcal{L}^\prime}\tau_i^{(\ell,n+1)}=1,\quad \forall i\in\O,\label{convexcomba2}\\
&\sum_{\ell\in\mathcal{L}^\prime}\theta^{(\ell,n+1)}=1,\label{convexcombb2}\\
&\varrho (\lambda_i^{(n+1)}-\bar{\lambda}^{(k)}_i)=\frac12\of{\sum_{\ell\in\mathcal{L}^\prime}\tau_i^{(\ell,n+1)}g_{\lambda_i^{(\ell)}}-\sigma^{(n+1)} p_i},\quad \forall i\in\O\label{firstorder2}\\
& -2\varrho(m_i^{(n+1)}\negthinspace-\negthinspace\bar{m}^{(k)}_i)\negthinspace+\negthinspace\sum_{\ell\in\mathcal{L}^{\prime}}\tau_i^{(\ell,n+1)}g_{m_i^{(\ell)}}\negthinspace +\negthinspace\sum_{\ell\in\mathcal{L}^{\prime}}\theta^{(\ell,n+1)}\rho_{m^{(\ell)}_i}-\psi^{(n+1)}\1 +\nu_i^{(n+1)}=0,\, \forall i\in\O.\label{mufirstorder22}
\end{align}
\end{lemma}

\begin{proof}
The Lagrangian for $(MP_2(v))$ with centers $\bar{\mu}^{(k)},\bar{\lambda}^{(k)}$ and index set $\mathcal{L}^{\prime}$ is
\begin{align}\label{lagrangian2}
&L(m,\lambda,\vartheta,\eta,\tau,\theta,\sigma,\psi,\nu)\\
&=\sum_{i\in\O}\vartheta_i+\eta-\sum_{i\in\O}\varrho\norm{m_i-\bar{m}^{(k)}_i}^2-\sum_{i\in\O}\varrho\norm{\lambda_i-\bar{\lambda}^{(k)}_i}^2\notag \\
& \quad +\sum_{i\in\O}\sum_{\ell\in\mathcal{L}^{\prime}}\tau_i^{(\ell)}\of{\tilde{f}_i(m_i^{(\ell)},\lambda_i^{(\ell)}) + g_{m^{(\ell)}_i}^{\mathsf{T}}(m_i-m^{(\ell)}_i) + g_{\lambda^{(\ell)}_i}^{\mathsf{T}}(\lambda_i-\lambda_i^{(\ell)})-\vartheta_i}\notag\\
&\quad +\sum_{\ell\in\mathcal{L}^{\prime}}\theta^{(\ell)}\of{-\tilde{\beta}(m^{(\ell)}) + \sum_{i\in\O}\rho_{m^{(\ell)}_i}^{\mathsf{T}}(m_i-m_i^{(\ell)})-\eta}\notag \\
&\quad -\sigma^{\mathsf{T}}\sum_{i\in\O}p_i\lambda_i+\psi\of{1-\sum_{i\in\O}m_i^{\mathsf{T}}\1}+\sum_{i\in\O}\nu_i^{\mathsf{T}}m_i.\notag
\end{align}
The dual objective function is defined by
\[
h(\tau,\theta,\sigma,\psi,\nu)=\sup_{m\in\L^J,\lambda\in\L^M,\vartheta\in\L,\eta\in\R}L(m,\lambda,\vartheta,\eta,\tau,\theta,\sigma,\psi,\nu),
\]
and the dual problem is
\begin{align*}
&\text{min }\;       h(\tau,\theta,\sigma,\psi,\nu)\;\tag{$D-MP_2(v)$}\\
&\text{s.t.}\quad \tau_i^{(\ell)}\geq 0,\theta^{(\ell)}\geq0,\sigma\in\R^M,\psi\in\R,\nu_i\in\R^J_+.
\end{align*}
Note that $(m^{(n+1)},\lambda^{(n+1)},\vartheta^{(n+1)},\eta^{(n+1)})$ is an optimal solution for $(MP_2(v))$ with centers $\bar{m}^{(k)},\bar{\lambda}^{(k)}$ and index set $\mathcal{L}^{\prime}$, and
\[
\of{\tau^{(n+1)}=(\tau_i^{(\ell,n+1)})_{i\in\O,\ell\in\mathcal{L}^\prime},\theta^{(n+1)}=(\theta^{(\ell,n+1)})_{\ell\in\mathcal{L}^\prime},\sigma^{(n+1)},\psi^{(n+1)},\nu^{(n+1)}=(\nu_i^{(n+1)})_{i\in\O}}
\]
is the corresponding optimal solution for $(D-MP_2(v))$.
The maximization of the Lagrangian over $\vartheta\in\L$ and $\eta\in\R$ gives \eqref{convexcomba2} and \eqref{convexcombb2}, respectively, as constraints for the dual problem. The maximization of the Lagrangian over $\lambda\in\L^M$ gives the first-order condition \eqref{firstorder2}. Finally, the maximization of the Lagrangian over $m\in\L^J$ gives the first-order condition \eqref{mufirstorder22}.
\end{proof}

\begin{lemma}\label{feasibility2}
The following statements hold for every $i\in\O$:
\begin{enumerate}[(a)]
\item As $\varepsilon\rightarrow0$, $\bar{x}_i-\sigma^{(n+1)}\rightarrow 0$.
\item $(\bar{x}_i,\bar{y}_i)\in\mathcal{F}_i$.
\item $(\bar{x},\bar{y})$ will eventually be an element of the set $\cb{(x,y)\in\F\mid p_i(x_i-\E\sqb{x})=0,\;\forall i\in\O}$ as $\varepsilon\rightarrow 0$ in the sense that $(\bar{x},\bar{y})\in\F$ and $p_i(\bar{x}_i-\E\sqb{\bar{x}})\rightarrow 0$ for every $i\in\O$ as $\varepsilon\rightarrow 0$.
\end{enumerate}
\end{lemma}

\begin{proof}
The proof of this lemma is similar to the proof of Lemma~\ref{feasibility}. Therefore, it is omitted.
\end{proof}

Recall that $\rho_{m^{(\ell)}}=(\rho_{m_1^{(\ell)}},\ldots,\rho_{m_I^{(\ell)}})$ is a subgradient of $-\tilde{\beta}(\cdot,w)$ at $m^{(\ell)}$, for each $\ell\in\mathcal{L}^{\prime}$. From \eqref{betatilde}, there exists some $u^{(\ell)}\in\A$, $\ell\in\mathcal{L}^{\prime}$, such that
\begin{equation}\label{u-ell2}
\rho_{m_i^{(\ell)}}=-u^{(\ell)}_i
\end{equation}
for all $i\in\O$. Since $\A$ is convex, $\theta^{(\ell)}\geq 0$ and $\sum_{\ell\in\mathcal{L}^{\prime}}\theta^{(\ell)}=1$, it follows that
\begin{equation}\label{ubar2}
\bar{u}\coloneqq \sum_{\ell\in\mathcal{L}^{\prime}}\theta^{(\ell)}u^{(\ell)}=\of{\sum_{\ell\in\mathcal{L}^{\prime}}\theta^{(\ell)}u_1^{(\ell)},\ldots,\sum_{\ell\in\mathcal{L}^{\prime}}\theta^{(\ell)}u_I^{(\ell)}}\in\A.
\end{equation}

\begin{lemma}
The followings hold:
\begin{enumerate}[(a)]
\item For each $i\in\O$,
\begin{equation}\label{varthetaslackness2}
\vartheta^{(n+1)}_i=(m^{(n+1)}_i)^{\mathsf{T}} (C\bar{x}_i+Q_i \bar{y}_i)+p_i (\lambda^{(n+1)}_i)^{\mathsf{T}}\bar{x}_i.
\end{equation}
\item For each $i\in\O$,
\begin{equation}\label{approxopt2}
\vartheta^{(n+1)}_i - \tilde{f}_i(m^{(n+1)}_i, \lambda^{(n+1)}_i)\leq\bar{\varepsilon}^{(n+1)}.
\end{equation}
Moreover,
\begin{equation}\label{approxopt02}
\sum_{i\in\O}\vartheta^{(n+1)}_i + \eta^{(n+1)} - \sum_{i\in\O}\tilde{f}_i(m^{(n+1)}_i, \lambda^{(n+1)}_i)+\tilde{\beta}(m^{(n+1)})\leq \bar{\varepsilon}^{(n+1)}.
\end{equation}
\item
\begin{equation}\label{eta-beta2}
\eta^{(n+1)}+\tilde{\beta}(m^{(n+1)})\leq \bar{\varepsilon}^{(n+1)}.
\end{equation}
\item $-\bar{u}$ is an $\bar{\varepsilon}^{(n+1)}$-subgradient of $-\tilde{\beta}(\cdot)$ at $m^{(n+1)}$ in the sense that, for every $m\in\M_f^J$,
\begin{equation}\label{betaapproxsbg2}
-\tilde{\beta}(m) \leq -\tilde{\beta}(m^{(n+1)})+\bar{\varepsilon}^{(n+1)}-\sum_{i\in\O}(m_i-m_i^{(n+1)})^{\mathsf{T}} \bar{u}_i.
\end{equation}
\end{enumerate}
\end{lemma}

\begin{proof}
Consider $(MP_2(v))$ with centers $\bar{m}^{(k)},\bar{\lambda}^{(k)}$ and index set $\mathcal{L}^{\prime}$.
\begin{enumerate}[(a)]
\item Note that constraint \eqref{p2con1} can be rewritten as
\[
\vartheta_i \leq m_i ^{\mathsf{T}} (Cx_i^{(\ell)}+Q_i y_i^{(\ell)})+p_i \lambda_i^{\mathsf{T}}x^{(\ell)}_i, \quad \forall i\in\O,  \ell\in\mathcal{L}^\prime
\]
since
\[
\tilde{f}_i(m_i^{(\ell)},\lambda_i^{(\ell)})=(m_i^{(\ell)})^\mathsf{T} (Cx^{(\ell)}_i+Q_i y^{(\ell)}_i)+p_i(\lambda^{(\ell)}_i)^\mathsf{T}x^{(\ell)}_i
\]
and $g_{m^{(\ell)}_i}=Cx_i^{(\ell)}+Q_iy_i^{(\ell)}$, $g_{\lambda_i^{(\ell)}} = p_ix_i^{(\ell)}$. From the complementary slackness conditions for constraint \eqref{p2con1} and using \eqref{xbarybar2}, \eqref{convexcomba2}, we get
\begin{align*}
\vartheta^{(n)}_i &=\sum_{\ell\in\mathcal{L}^\prime}\vartheta^{(n+1)}_i\tau^{(\ell,n+1)}_i\\
&= \sum_{\ell\in\mathcal{L}^\prime}\of{(m^{(n+1)}_i)^{\mathsf{T}} (Cx_i^{(\ell)}+Q_i y_i^{(\ell)})+p_i (\lambda^{(n+1)}_i)^{\mathsf{T}}x^{(\ell)}_i}\tau^{(\ell,n+1)}_i\notag \\
&=(m^{(n+1)}_i)^{\mathsf{T}} (C\bar{x}_i+Q_i \bar{y}_i)+p_i (\lambda^{(n+1)}_i)^{\mathsf{T}}\bar{x}_i.\notag
\end{align*}
Hence, \eqref{varthetaslackness2} follows.

\item Note that
\begin{align*}
&\vartheta^{(n+1)}_i - \tilde{f}_i(m^{(n+1)}_i, \lambda^{(n+1)}_i)-\phi^{(n+1,k)}\\
&\leq \sum_{i\in\O}\vartheta^{(n+1)}_i + \eta^{(n+1)} -\sum_{i\in\O}(m_i^{(n+1)})^{\mathsf{T}}v -\sum_{i\in\O}\tilde{f}_i(m^{(n+1)}_i, \lambda^{(n+1)}_i)+\tilde{\beta}(m^{(n+1)})+\sum_{i\in\O}(m_i^{(n+1)})^{\mathsf{T}}v-\phi^{(n+1,k)} \\
&\leq \sum_{i\in\O}\vartheta^{(n+1)}_i + \eta^{(n+1)}-\sum_{i\in\O}(m_i^{(n+1)})^{\mathsf{T}}v - \sum_{i\in\O}\tilde{f}_i(\bar{m}^{(k)}_i, \bar{\lambda}^{(k)}_i)+\tilde{\beta}(\bar{m}^{(k)})+\sum_{i\in\O}(\bar{m}_i^{(k)})^{\mathsf{T}}v-\phi^{(n+1,k)}\\
&\leq \sum_{i\in\O}\vartheta^{(k+1)}_i + \eta^{(k+1)} -\sum_{i\in\O}(m_i^{(k+1)})^{\mathsf{T}}v - \sum_{i\in\O}\tilde{f}_i(\bar{m}^{(k)}_i, \bar{\lambda}^{(k)}_i)+\tilde{\beta}(\bar{m}^{(k)})+\sum_{i\in\O}(\bar{m}_i^{(k)})^{\mathsf{T}}v-\phi^{(k+1,k)}\\
&\leq \varepsilon.
\end{align*}
Here, the first inequality follows since $\vartheta_i^{(n+1)}\geq \tilde{f}_i(m_i^{(n+1)},\lambda_i^{(n+1)})$ for each $i\in\O$ and $\eta^{(n+1)}\geq -\tilde{\beta}(m^{(n+1)})$, the second inequality follows since
\begin{align*}
&F^{(n+1)}=\sum_{i\in\O}\tilde{f}_i(m^{(n+1)}_i, \lambda^{(n+1)}_i)-\tilde{\beta}(m^{(n+1)})-\sum_{i\in\O}(m_i^{(n+1)})^{\mathsf{T}}v\\
&\geq \bar{F}^{(k+1)}=\sum_{i\in\O}\tilde{f}_i(\bar{m}^{(k)}_i, \bar{\lambda}^{(k)}_i)-\tilde{\beta}(\bar{m}^{(k)})-\sum_{i\in\O}(\bar{m}_i^{(k)})^{\mathsf{T}}v
\end{align*}
due to the center update rule in line~14 of Algorithm~\ref{decomp2}, the third inequality follows since the master problem with index set $\mathcal{L}=\cb{1,\ldots,k}$ and center $(\bar{m}^{(k)},\bar{\lambda}^{(k)})$ has a smaller optimal value than the master problem with index set $\mathcal{L}^{\prime}=\cb{1,\ldots,n}$ and center $(\bar{m}^{(k)},\bar{\lambda}^{(k)})$. Finally, the last inequality is by the approximate stopping condition \eqref{approxstop2}. Hence, \eqref{approxopt2} and \eqref{approxopt02} follow.

\item Similar to part (b), we can show that \eqref{eta-beta2} holds.

\item Note that constraint \eqref{p2con2} can be rewritten as
\[
\eta \leq \sum_{i\in\O}\rho_{m^{(\ell)}_i}^{\mathsf{T}}m_i, \quad \forall \ell\in\mathcal{L}^{\prime}
\]
since $-\tilde{\beta}(m^{(\ell)})=\sum_{i\in\O}\rho_{m^{(\ell)}_i}^{\mathsf{T}}m_i^{(\ell)}$ by the definition of subgradient. From the complementary slackness conditions,
\begin{equation}\label{etam2}
\eta^{(n+1)}=\sum_{\ell\in\mathcal{L}^{\prime}}\theta^{(\ell)}\sum_{i\in\O}\rho_{m^{(\ell)}_i}^{\mathsf{T}}m^{(n+1)}_i=\sum_{i\in\O}m^{(n+1)}_i\sum_{\ell\in\mathcal{L}^{\prime}}\theta^{(\ell)}\rho_{m^{(\ell)}_i}^{\mathsf{T}}.
\end{equation}
From \eqref{u-ell2}, \eqref{ubar2}, \eqref{etam2}, it follows that
\begin{equation}\label{etaslackness2}
\eta^{(n+1)} = -\sum_{i\in\O}(m^{(n+1)}_i)^{\mathsf{T}} \bar{u}_i.
\end{equation}
For every $m\in\M_f^J$,
\begin{align*}
-\tilde{\beta}(m)&\leq -\sum_{i\in\O}m_i^{\mathsf{T}}\bar{u}_i\\
&=-\sum_{i\in\O}(m^{(n+1)}_i)^{\mathsf{T}}\bar{u}_i-\sum_{i\in\O}(m_i-m_i^{(n+1)})^{\mathsf{T}} \bar{u}_i\\
& \leq -\tilde{\beta}(m^{(n+1)})+\bar{\varepsilon}^{(n+1)}-\sum_{i\in\O}(m_i-m_i^{(n+1)})^{\mathsf{T}} \bar{u}_i,
\end{align*}
where the first inequality follows from \eqref{betatilde}, and the last inequality follows from \eqref{eta-beta2} and \eqref{etaslackness2}. Hence, the claim follows.
\end{enumerate}
\end{proof}

\begin{lemma}\label{optimality2}
The followings hold:
\begin{enumerate}[(a)]
\item As $\varepsilon\rightarrow0$,
\begin{align*}
&\sum_{i\in\O}(m^{(n+1)}_i)^{\mathsf{T}}(C\bar{x}_i+Q_i\bar{y}_i)-\tilde{\beta}(m^{(n+1)})-\sum_{i\in\O}(m_i^{(n+1)})^{\mathsf{T}}v\rightarrow\mathscr{P}_2(v),\\ &\sum_{i\in\O}\tilde{f}_i(m_i^{(n+1)},\lambda_i^{(n+1)})-\tilde{\beta}(m^{(n+1)})-\sum_{i\in\O}(m_i^{(n+1)})^{\mathsf{T}}v\rightarrow\mathscr{P}_2(v).
\end{align*}
\item Let
\begin{equation}\label{alphabar}
\bar{\alpha}\coloneqq \inf\cb{\alpha\in\R\mid v+\alpha\1\in R(C\bar{x}+Q\bar{y})}.
\end{equation}
Then, $\bar{\alpha}\rightarrow\mathscr{P}_2(v)$ as $\varepsilon\rightarrow 0$.
\end{enumerate}
\end{lemma}

\begin{proof}
\begin{enumerate}[(a)]
\item By \eqref{varthetaslackness2},
\begin{equation}\label{newequation}
\sum_{i\in\O}\vartheta^{(n+1)}_i=\sum_{i\in\O}(m^{(n+1)}_i)^{\mathsf{T}} (C\bar{x}_i+Q_i \bar{y}_i)+\sum_{i\in\O}p_i (\lambda^{(n+1)}_i)^{\mathsf{T}}\bar{x}_i.
\end{equation}
Using Lemma~\ref{feasibility2}(a), \eqref{newequation} and \eqref{optimalitybound2}, we obtain
\begin{equation}\label{optimalitybound22}
\abs{\sum_{i\in\O}\vartheta^{(n+1)}_i - \sum_{i\in\O}(m^{(n+1)}_i)^{\mathsf{T}} (C\bar{x}_i+Q_i \bar{y}_i)}=\abs{\sum_{i\in\O}p_i (\lambda^{(n+1)}_i)^{\mathsf{T}}\bar{x}_i}\rightarrow 0
\end{equation}
as $\varepsilon\rightarrow 0$. On the other hand, by \eqref{approxopt02} and using the fact that $\vartheta_i^{(n+1)}, \eta^{(n+1)}$ are upper approximations for $\tilde{f}_i(m_i^{(n+1)},\lambda_i^{(n+1)}), -\tilde{\b}(m^{(n+1)})$, respectively, we get
\begin{align}\label{absolute21}
0 &\leq \sum_{i\in\O}\vartheta^{(n+1)}_i-\tilde{\b}(m^{(n+1)})-\sum_{i\in\O}(m_i^{(n+1)})^{\mathsf{T}}v-\sum_{i\in\O}\tilde{f}_i(m_i^{(n+1)},\lambda_i^{(n+1)})+\tilde{\b}(m^{(n+1)})+\sum_{i\in\O}(m_i^{(n+1)})^{\mathsf{T}}v \\
&\leq \sum_{i\in\O}\vartheta^{(n+1)}_i+\eta^{(n+1)}-\sum_{i\in\O}\tilde{f}_i(m_i^{(n+1)},\lambda_i^{(n+1)})+\tilde{\b}(m^{(n+1)})\leq \bar{\varepsilon}^{(n+1)}.\notag
\end{align}

By an analog of \eqref{normconvergence}, $\bar{\varepsilon}^{(n+1)}=\varepsilon+\phi^{(n+1,k)}\rightarrow 0$ as $\varepsilon\rightarrow 0$. By Lemma~7.17 in \cite{ruszbook}, as $\varepsilon\rightarrow 0$, $\sum_{i\in\O}\vartheta^{(n+1)}_i+\eta^{(n+1)}-\sum_{i\in\O}(m_i^{(n+1)})^{\mathsf{T}}v$ converges to the optimal value of $(D_2(v))$, which is $\mathscr{P}_2(v)$. Hence, $\sum_{i\in\O}\tilde{f}_i(m_i^{(n+1)},\lambda_i^{(n+1)})-\tilde{\beta}(m^{(n+1)})-\sum_{i\in\O}(m_i^{(n+1)})^{\mathsf{T}}v$ also converges to $\mathscr{P}_2(v)$ as $\varepsilon\rightarrow 0$. Finally, by triangle inequality, \eqref{optimalitybound22} and \eqref{absolute21} yield
\begin{align*}
&\Bigg\vert\sum_{i\in\O}(m^{(n+1)}_i)^{\mathsf{T}} (C\bar{x}_i+Q_i \bar{y}_i)-\tilde{\beta}(m^{(n+1)})-\sum_{i\in\O}(m_i^{(n+1)})^{\mathsf{T}}v\\
&-\of{\sum_{i\in\O}\tilde{f}_i(m_i^{(n+1)},\lambda_i^{(n+1)})-\tilde{\beta}(m^{(n+1)})-\sum_{i\in\O}(m_i^{(n+1)})^{\mathsf{T}}v}\Bigg\vert\leq \abs{\sum_{i\in\O}p_i (\lambda^{(n+1)}_i)^{\mathsf{T}}\bar{x}_i}+\bar{\varepsilon}^{(n+1)}.
\end{align*}
From \eqref{optimalitybound22}, the right hand side of the above inequality converges to zero as $\varepsilon\rightarrow 0$. We conclude that $\sum_{i\in\O}(m^{(n+1)}_i)^{\mathsf{T}} (C\bar{x}_i+Q_i \bar{y}_i)-\tilde{\beta}(m^{(n+1)})-\sum_{i\in\O}(m_i^{(n+1)})^{\mathsf{T}}v$ also converges to $\mathscr{P}_2(v)$ as $\varepsilon\rightarrow 0$.

\item By Lemma~\ref{conversionlemma} and Lemma~\ref{scalarizationlemma}, we have
\begin{align*}
\bar{\alpha}&=\sup\cb{\gamma^\mathsf{T}\of{\E^\mu\sqb{C\bar{x}+Q\bar{y}}-v}-\b(\mu,\gamma)\mid \mu\in\M_1^J,\; \gamma^\mathsf{T}\1=1, \gamma\in\R^J_+ }\\
&=\sup\cb{\sum_{i\in\O}m_i^{\mathsf{T}}\of{C\bar{x}_i+Q_i\bar{y}_i}-\sum_{i\in\O}m_i^{\mathsf{T}}v-\tilde{\b}(m)\mid m\in\M_f^J,\; \sum_{i\in\O}m_i^\mathsf{T}\1=1 }\\
&=\sup_{m\in\L^J}\cb{\sum_{i\in\O}m_i^{\mathsf{T}}\of{C\bar{x}_i+Q_i\bar{y}_i}-\sum_{i\in\O}m_i^{\mathsf{T}}v-\tilde{\b}(m)\mid \sum_{i\in\O}m_i^\mathsf{T}\1=1, m_i\in\R^J_+,\forall i\in\O }.
\end{align*}
The corresponding Lagrangian dual problem is given by
\[
\inf_{\psi\in\R^J,\nu\in\L^J_+}\sup_{m\in\L^J}\of{\sum_{i\in\O}m_i^{\mathsf{T}}\of{C\bar{x}_i+Q_i\bar{y}_i}-\sum_{i\in\O}m_i^{\mathsf{T}}v-\tilde{\b}(m)+\psi^{\mathsf{T}}\of{\mathbf{1}-\sum_{i\in\O}m_i}+\sum_{i\in\O}\nu_i^{\mathsf{T}}m_i}.
\]
The optimization over $m_1,\ldots,m_I$ gives the first-order condition
\begin{equation}\label{mufirstorder3}
C\bar{x}_i+Q_i\bar{y}_i-v+\rho_{m_i}-\psi+\nu_i=0, \quad \forall i\in\O,
\end{equation}
for some $\rho_{m}=(\rho_{m_1},\ldots,\rho_{m_I})\in\partial_{m}(-\tilde{\beta})(m)$.
We argue that $(m^{(n+1)},\psi^{(n+1)},\nu^{(n+1)})$ satisfies \eqref{mufirstorder3} in an approximate sense. Note that we may rewrite \eqref{mufirstorder22} as
\begin{align}\label{optcond2}
&-2\varrho(m_i^{(n+1)}\negthinspace-\negthinspace\bar{m}^{(k)}_i)\negthinspace+\negthinspace\sum_{\ell\in\mathcal{L}^{\prime}}\tau_i^{(\ell,n+1)}\of{Cx^{(\ell)}_i+Q_iy_i^{(\ell)}}\negthinspace +\negthinspace\sum_{\ell\in\mathcal{L}^{\prime}}\theta^{(\ell,n+1)}u_i^{(\ell)}-\psi^{(n+1)}\1 +\nu_i^{(n+1)}\notag \\
&=-2\varrho(m_i^{(n+1)}\negthinspace-\negthinspace\bar{m}^{(k)}_i)\negthinspace+\negthinspace\of{C\bar{x}_i+Q_i\bar{y}_i}\negthinspace +\negthinspace\bar{u}_i-\psi^{(n+1)}\1 +\nu_i^{(n+1)}\\
&=0.\notag
\end{align}
As shown in \eqref{betaapproxsbg2}, $-\bar{u}=-(\bar{u}_1,\ldots,\bar{u}_I)$ is an $\bar{\varepsilon}^{(n+1)}$-subgradient of $-\tilde{\beta}(\cdot)$ at $m^{(n+1)}$. Hence, there exists a subgradient $\rho_{m^{(n+1)}}$ of $-\tilde{\beta}(\cdot)$ at $m^{(n+1)}$ for each $n$ such that $-\bar{u}-\rho_{m^{(n+1)}}\rightarrow 0$ as $\varepsilon\rightarrow 0$. On the other hand, by an analog of \eqref{normconvergence}, $m_i^{(n+1)}-\bar{m}^{(k)}_i\rightarrow 0$ as $\varepsilon\rightarrow0$. Therefore, from \eqref{optcond2}, \eqref{mufirstorder3} is satisfied by $(m^{(n+1)},\psi^{(n+1)},\nu^{(n+1)})$ and $-\bar{u}$ approximately in the sense that
\[
C\bar{x}_i+Q_i\bar{y}_i-\bar{u}_i-\psi^{(n+1)}+\nu^{(n+1)}_i\rightarrow 0
\]
and
\begin{equation}\label{asymptotic}
\abs{\sum_{i\in\O}(m^{(n+1)}_i)^{\mathsf{T}}\of{C\bar{x}_i+Q_i\bar{y}_i}-\sum_{i\in\O}(m^{(n+1)}_i)^{\mathsf{T}}v-\tilde{\b}(m^{(n+1)})-\bar{\alpha}}\rightarrow 0
\end{equation}
in the limit as $\varepsilon\rightarrow 0$.
\end{enumerate}
\end{proof}

\begin{proof}[Proof of Theorem~\ref{P2primal}]
Note that $x_{(v)}=\bar{x}$, $y_{(v)}=\bar{y}$, $\a_{(v)}=\bar{\a}$,where $\bar{x},\bar{y},\bar{\a}$ are defined by \eqref{xbarybar2} and \eqref{alphabar}. Parts (a) and (c) follow directly from Lemma~\ref{feasibility2}. Parts (b) and (d) follow from Lemma~\ref{optimality2}.
\end{proof}

\section{Proof of Theorem~\ref{ld2sol}}\label{proofof5}

\begin{proof}[Proof of Theorem~\ref{ld2sol}]

For the objective function of $(LD_2(v))$ evaluated at $\gamma=\gamma_{(v)}$, we have
\begin{align}
&\inf_{(x,y)\in\X,\a\in\R}\of{\a+\inf_{z\in R(Cx+Qy)-v-\alpha\1}\gamma_{(v)}^{\mathsf{T}}z}\notag\\
&=\inf_{\substack{(x,y)\in\F,\a\in\R,\\p_i(x_i-\E\sqb{x})=0, \forall i\in\O}}\of{\a+\inf_{z\in R(Cx+Qy)-v-\alpha\1}\gamma_{(v)}^{\mathsf{T}}z}\notag\\
&=\inf_{\substack{(x,y)\in\F,\\p_i(x_i-\E\sqb{x})=0, \forall i\in\O}} \sup_{\substack{m\in\M_f^J,\\ \sum_{i\in\O}m_i=\gamma_{(v)} }}\of{\sum_{i\in\O}m_i^{\mathsf{T}}(Cx_{i}+Q_iy_i)-\tilde{\beta}(m)-\gamma^{\mathsf{T}}_{(v)}v}\notag \\
&=\sup_{\lambda\in \L^M}\inf_{(x,y)\in\F}\of{ \sup_{\substack{m\in\M_f^J,\\ \sum_{i\in\O}m_i=\gamma_{(v)} }}\of{\sum_{i\in\O}m_i^{\mathsf{T}}(Cx_{i}+Q_iy_i)-\tilde{\beta}(m)-\sum_{i\in\O}m_i^{\mathsf{T}}v}+\sum_{i\in\O}p_i\lambda_i^{\mathsf{T}}(x_i-\E\sqb{x})}\notag\\
&=\sup_{\substack{\lambda\in \L^M,\\\E\sqb{\lambda}=0}}\inf_{(x,y)\in\F}\of{ \sup_{\substack{m\in\M_f^J,\\ \sum_{i\in\O}m_i=\gamma_{(v)} }}\of{\sum_{i\in\O}m_i^{\mathsf{T}}(Cx_{i}+Q_iy_i)-\tilde{\beta}(m)-\sum_{i\in\O}m_i^{\mathsf{T}}v}+\sum_{i\in\O}p_i\lambda_i^{\mathsf{T}}x_i}\notag\\
&=\sup_{\substack{\lambda\in \L^M,m\in\M_f^J,\\\E\sqb{\lambda}=0,\\ \sum_{i\in\O}m_i=\gamma_{(v)}}}\of{ \sum_{i\in\O}\of{\inf_{(x_i,y_i)\in\F_i}\of{m_i^{\mathsf{T}}(Cx_{i}+Q_iy_i)+p_i\lambda_i^{\mathsf{T}}x_i}-m_i^{\mathsf{T}}v }-\tilde{\beta}(m)}\notag\\
&=\sup_{\substack{\lambda\in \L^M,m\in\M_f^J,\\\E\sqb{\lambda}=0,\\ \sum_{i\in\O}m_i=\gamma_{(v)}}}\of{\sum_{i\in\O}(\tilde{f}_i(m_i,\lambda_i)-m_i^{\mathsf{T}}v )-\tilde{\beta}(m)}\label{maxtildef1}\\
&\leq \sup_{\substack{\lambda\in \L^M,m\in\M_f^J,\\\E\sqb{\lambda}=0,\\ \sum_{i\in\O}m_i^{\mathsf{T}}\1=1}}\of{\sum_{i\in\O}(\tilde{f}_i(m_i,\lambda_i)-m_i^{\mathsf{T}}v )-\tilde{\beta}(m)}\label{maxtildef2}\\
&=\mathscr{P}_2(v),\notag
\end{align}
where the first equality is obvious, the second equality follows by \eqref{scalar-dual}, \eqref{p2con4} and Lemma~\ref{conversionlemma}, the third and fourth equalities follow by the dualization of the nonanticipativity constraints as in the proof of Theorem~\ref{p1thm}, the fifth equality is by the interchange of infimum and supremum using \cite{sion}, the sixth equality follows by the definition of $\tilde{f}_i(\cdot,\cdot)$ in \eqref{tildef}. The inequality in \eqref{maxtildef2} follows since the feasible region of the problem in \eqref{maxtildef1} is a subset of the feasible region of the problem in \eqref{maxtildef2}. The last equality is by Theorem~\ref{p2thm}.

Since $(\lambda^{(n+1)},m^{(n+1)})$ is a feasible solution for the maximization problem in \eqref{maxtildef1}, we have
\[
\sum_{i\in\O}(\tilde{f}_i(m^{(n+1)}_i,\lambda^{(n+1)}_i)-(m^{(n+1)}_i)^{\mathsf{T}}v )-\tilde{\beta}(m^{(n+1)})\leq \sup_{\substack{\lambda\in \L^M,m\in\M_f^J,\\\E\sqb{\lambda}=0,\\ \sum_{i\in\O}m_i=\gamma_{(v)}}}\of{\sum_{i\in\O}(\tilde{f}_i(m_i,\lambda_i)-m_i^{\mathsf{T}}v )-\tilde{\beta}(m)}\leq \mathscr{P}_2(v).
\]
As $\varepsilon\rightarrow\infty$, by Lemma~\ref{optimality2}(a), the first expression converges to $\mathscr{P}_2(v)$, and so does the second by sandwich theorem. This finishes the proof of \eqref{triangle2}.

\end{proof}

\end{document}